\documentclass[11pt]{article}
\usepackage[colorlinks=true,linkcolor=blue,citecolor=Red]{hyperref}
\usepackage{breakcites}

\usepackage{mathtools}

\usepackage{subfigure}

\usepackage{algorithm}
\usepackage{algpseudocode}

\usepackage{bbm}

\usepackage{latexsym}
\usepackage{graphics}
\usepackage{amsmath}
\usepackage[algo2e]{algorithm2e}
\usepackage{xspace}
\usepackage{amssymb}
\usepackage{psfrag}
\usepackage{epsfig}
\usepackage{amsthm}
\usepackage{url}
\usepackage{pst-all}
\usepackage{bbm}

\usepackage{color}

\definecolor{Red}{rgb}{1,0,0}
\definecolor{Blue}{rgb}{0,0,1}
\definecolor{Olive}{rgb}{0.41,0.55,0.13}
\definecolor{Yarok}{rgb}{0,0.5,0}
\definecolor{Green}{rgb}{0,1,0}
\definecolor{MGreen}{rgb}{0,0.8,0}
\definecolor{DGreen}{rgb}{0,0.55,0}
\definecolor{Yellow}{rgb}{1,1,0}
\definecolor{Cyan}{rgb}{0,1,1}
\definecolor{Magenta}{rgb}{1,0,1}
\definecolor{Orange}{rgb}{1,.5,0}
\definecolor{Violet}{rgb}{.5,0,.5}
\definecolor{Purple}{rgb}{.75,0,.25}
\definecolor{Brown}{rgb}{.75,.5,.25}
\definecolor{Grey}{rgb}{.5,.5,.5}

\newcommand{\ind}{\mathbbm{1}}

\newcommand{\G}{\mathbb{G}}

\newcommand{\Z}{\mathbb{Z}}

\newcommand{\R}{\mathbb{R}}
\newcommand{\N}{\mathbb{N}}

\renewcommand{\Z}{\mathbb{Z}}

\newcommand{\ip}[2]{\langle{#1},{#2}\rangle}

\renewcommand{\R}{\mathbb{R}}
\renewcommand{\Z}{\mathbb{Z}}
\renewcommand{\G}{\mathbb{G}}

\newcommand{\distr}{\stackrel{d}{=}}

\newcommand{\M}{{\bf \mathcal{M}}}
\newcommand{\cN}{{\bf \mathcal{N}}}
\newcommand{\Overlap}{\mathcal{\mathcal{O}}}

\newcommand{\ignore}[1]{\relax}

\newlength\myindent
\newtheorem{theorem}{Theorem}[section]
\newtheorem{remark}[theorem]{Remark}
\newtheorem{lemma}[theorem]{Lemma}
\newtheorem{conjecture}[theorem]{Conjecture}

\newtheorem{proposition}[theorem]{Proposition}

\newtheorem{definition}[theorem]{Definition}

\newcommand{\ER}{Erd{\"o}s-R\'{e}nyi }

\renewcommand{\ip}[2]{\left\langle#1,#2\right\rangle}


\newcounter{parentnumber}

\makeatletter
\def\BState{\State\hskip-\ALG@thistlm}
\makeatother

\definecolor{Red}{rgb}{1,0,0}
\definecolor{Blue}{rgb}{0,0,1}
\definecolor{Olive}{rgb}{0.41,0.55,0.13}
\definecolor{Green}{rgb}{0,1,0}
\definecolor{MGreen}{rgb}{0,0.8,0}
\definecolor{DGreen}{rgb}{0,0.55,0}
\definecolor{Yellow}{rgb}{1,1,0}
\definecolor{Cyan}{rgb}{0,1,1}
\definecolor{Magenta}{rgb}{1,0,1}
\definecolor{Orange}{rgb}{1,.5,0}
\definecolor{Violet}{rgb}{.5,0,.5}
\definecolor{Purple}{rgb}{.75,0,.25}
\definecolor{Brown}{rgb}{.75,.5,.25}
\definecolor{Grey}{rgb}{.5,.5,.5}
\definecolor{Pink}{rgb}{1,0,1}
\definecolor{DBrown}{rgb}{.5,.34,.16}
\definecolor{Black}{rgb}{0,0,0}

\newcommand{\A}{\mathcal{A}}

\usepackage{float}
\newcommand{\bs}{\boldsymbol{\sigma}}
\usepackage{float}

\title{Planted Random Number Partitioning Problem}
\author{{\sf Eren C. K{\i}z{\i}lda\u{g}}\thanks{Department of Statistics, Columbia University. e-mail: {\tt eck2170@columbia.edu}.}}
\begin{document}
\maketitle
\begin{abstract}
We consider the random number partitioning problem (\texttt{NPP}): given a list $X\sim \cN(0,I_n)$ of numbers, find a partition $\bs\in\{-1,1\}^n$ with a small objective value $H(\bs)=\frac{1}{\sqrt{n}}\left|\ip{\bs}{X}\right|$. The \texttt{NPP} is widely studied in computer science; it also has many practical applications, including covariate balancing in the design of randomized controlled trials. In this paper, we propose a planted version of the \texttt{NPP}: fix a $\bs^*$ and generate $X\sim \cN(0,I_n)$ conditional on $H(\bs^*)\le 3^{-n}$. The \texttt{NPP} and its planted counterpart are  statistically distinguishable as the smallest objective value under the former is $\Theta(\sqrt{n}2^{-n})$ w.h.p.

Our first focus is on the values of $H(\bs)$. We show that, perhaps surprisingly, planting does not induce partitions with an objective value substantially smaller than $2^{-n}$: $\min_{\bs \ne \pm \bs^*}H(\bs) = \widetilde{\Theta}(2^{-n})$ w.h.p. Furthermore, we completely characterize the smallest $H(\bs)$ achieved at any fixed distance from $\bs^*$. Our second focus is on the algorithmic problem of efficiently finding a partition $\bs$, not necessarily equal to $\pm\bs^*$, with a small $H(\bs)$. We show that planted \texttt{NPP} exhibits an intricate geometrical property known as the multi Overlap Gap Property ($m$-OGP) for values $2^{-\Theta(n)}$. We then leverage the $m$-OGP to show that stable algorithms satisfying a certain anti-concentration property fail to find a $\bs$ with $H(\bs)=2^{-\Theta(n)}$. 

Our results are the first instance of the $m$-OGP being
 established and leveraged to rule out stable algorithms for a planted model. More importantly, they show that the $m$-OGP framework can also apply to planted models, if the algorithmic goal is to return a solution with a small objective value. Our results suggest that the \texttt{NPP} and its planted counterpart behave similarly. In light of this, we conjecture that while the two versions are statistically distinguishable, it is impossible to distinguish them in polynomial time with non-trivial error probability. We believe that these intriguing features make the planted \texttt{NPP} a useful model for further investigation. 
\end{abstract}
\newpage
\tableofcontents
\newpage
\section{Introduction}
Given a list of numbers $X=(X_1,\dots,X_n)\in\R^n$, the \emph{number partitioning problem} (\texttt{NPP}) seeks to find a partition $\bs \in\Sigma_n\triangleq \{-1,1\}^n$ which is \emph{balanced}. That is, $\left|\ip{\bs}{X}\right|\triangleq |\textstyle \sum_{i\le n}\bs(i)X_i|$ is small. In this paper, we propose a planted version of the \texttt{NPP}, called the p-\texttt{NPP}, and thoroughly investigate this model. We fix an unknown $\bs^*\in \Sigma_n$ and a $C>0$, and generate $X\sim \cN(0,I_n)$ conditional on $\left|\ip{\bs^*}{X}\right| \le \sqrt{n}C^{-n}$. Namely, we `plant' a partition $\bs^*$ with objective value at most $\sqrt{n}C^{-n}$. In the \emph{unplanted} case where $X\sim \cN(0,I_n)$ unconditionally,~\cite{karmarkar1986probabilistic}  established that $\min_{\bs\in\Sigma_n}\left|\ip{\bs}{X}\right|=\Theta(\sqrt{n}2^{-n})$ w.h.p. Throughout, we denote the \emph{unplanted} version of the \texttt{NPP} as u-\texttt{NPP}, whenever a distinction is necessary.

Our particular focus is on the p-\texttt{NPP} when $C=3$, i.e. $\left|\ip{\bs^*}{X}\right|\le \sqrt{n}3^{-n}$, a value that w.h.p.\,does not occur in the u-\texttt{NPP}. In particular, the p-\texttt{NPP} and the u-\texttt{NPP} are statistically distinguishable. This choice of $C$ is merely for convenience; our results below apply to any $C>2$.

The \texttt{NPP} is of great significance in statistics. In particular, it is closely related to covariate balancing in the design of randomized controlled trials, see below. Additionally, planted models have been studied both in computer science, in  particular to understand random constraint satisfaction problems, and in statistics. Inference problems with a planted structure are central to modern high-dimensional statistics, where the algorithmic goal is to recover a certain structure `hidden' in a much larger, noisy one. One popular example is the stochastic block model~\cite{abbe2017community}, which has a diverse set of applications, including the study of social networks~\cite{holland1983stochastic,wang1987stochastic,newman2002random,goldenberg2010survey}. Other examples include sparse principal component analysis, regression, and the hidden clique problem, see~\cite{wu2021statistical} for details. 

The main motivation of our paper is two-fold. First, in light of the importance of the \texttt{NPP} and of planted models in relevant domains, it is natural to study a planted version of the \texttt{NPP}. Secondly, and perhaps more importantly, we employ the p-\texttt{NPP} to demonstrate that a powerful framework developed for understanding the limits of efficient algorithms for unplanted models can in fact extend to planted models. As we detail below, a particularly active line of work aims to link intricate geometry to computational hardness in many random models using the Overlap Gap Property framework. This theory is quite developed for unplanted models, such as random graphs and spin glasses. In particular, it yields lower bounds against many powerful classes of algorithms, and these lower bounds are nearly sharp for many important models (see below). However, due to certain challenges, this theory is limited for planted models. More specifically, it only yields lower bounds against local Monte Carlo Markov Chain type algorithms. Given the significance of planted models in modern statistics, understanding their computational limits is essential. Our paper constitutes a step towards this goal. Furthermore, as we discuss in Section~\ref{sec:open-problems}, our ideas may help the study of the limits of efficient algorithms for certain popular models in high-dimensional statistics, such as sparse regression and the hidden clique problem. 

We now elaborate on some of the prior work on the \texttt{NPP} and planted models.
\subsubsection{Number Partitioning Problem}
The \texttt{NPP} is one-dimensional version of the \emph{vector balancing problem} (\texttt{VBP}): given $X_1,\dots,X_n\in\R^d$, find a partition $\bs\in\Sigma_n$ which is as `balanced' as possible. That is, the \texttt{VBP} seeks to solve
\begin{equation}\label{eq:vbp}
\min_{\bs\in \Sigma_n}\left\|\sum_{1\le i\le n}\bs(i)X_i\right\|_\infty,\quad\text{where}\quad \bs = (\bs(i):1\le i\le n)\in\Sigma_n.
\end{equation}
\paragraph{Statistical Applications} The \texttt{VBP} is closely related to the design of \emph{randomized controlled trials} (RCTs) in statistics literature, on which we now elaborate. RCTs seek to make accurate inference for an additive treatment effect, such as a vaccine or a new drug. In the design of RCTs, $n$ individuals with covariance information $X_i\in\R^d$, $1\le i\le n$, are given; the goal is to split them into two groups, dubbed as the  \emph{treatment} and the \emph{control} group, that are balanced in terms of covariates. In order to make accurate inference for the treatment effect under investigation, good covariate balance is essential. RCTs are often considered as gold standard for clinical trials, see~\cite{krieger2019nearly,harshaw2019balancing,turner2020balancing} for a more elaborate discussion. The study of the role of covariate balance and the randomness in the design of RCTs has a rich history in statistics literature; for pointers to relevant literature, see e.g.~\cite{efron1971forcing,bertsimas2015power,kallus2018optimal,kapelner2021harmonizing,kapelner2022optimal,wang2022rerandomization}. Other practical applications of \texttt{NPP}/\texttt{VBP} include multiprocessor scheduling, design of VLSI circuits and cryptography~\cite{tsai1992asymptotic,merkle1978hiding}, see~\cite{coffman1991probabilistic} for further applications.

\paragraph{Computer Science and Statistical Physics} In addition to its rich applications, the \texttt{NPP} is also a widely studied model in computer science, statistical physics, and discrepancy theory. The \texttt{NPP} is in the list of \emph{six `basic' NP-complete problems} by Garey and Johnson~\cite{gareyjohnsonbook}; it is in fact the only problem in that list involving numbers. Furthermore, it is one of the first NP-complete models exhibiting a phase transition, which was conjectured by Mertens~\cite{mertens1998phase} using very elegant yet non-rigorous tools of statistical mechanics, and was confirmed rigorously by Borgs, Chayes and Pittel~\cite{borgs2001phase}. In statistical physics, the \texttt{NPP} is the first model for which the local REM conjecture is rigorously verified~\cite{borgs2009proof,borgs2009proof2}, see~\cite{bauke2004universality} for details regarding this conjecture. 

\paragraph{Discrepancy Theory} The \texttt{NPP} and the \texttt{VBP} are particularly related to \emph{combinatorial discrepancy theory}~\cite{spencer1985six,matousek1999geometric,chazelle_2000}. Given a matrix $M\in\R^{d\times n}$, a key objective in discrepancy theory is computing or bounding its discrepancy:
\[
\mathcal{D}(M) \triangleq \min_{\bs \in \Sigma_n} \bigl\|M\bs\bigr\|_\infty.
\]
When $M$ consists of columns $X_1,\dots,X_n\in\R^d$, the \texttt{VBP}~\eqref{eq:vbp} indeed seeks to minimize $\mathcal{D}(M)$. In discrepancy literature, both \emph{worst-case} and \emph{average-case} settings were considered. A landmark result in the worst-case setting is due to Spencer~\cite{spencer1985six}, dubbed as ``six standard deviations suffice": $D(M)\le 6\sqrt{n}$ if $d=n$ and $\max_{1\le i\le n}\|X_i\|_\infty\le 1$. Spencer's result is non-constructive, see e.g.~\cite{bansal2010constructive,lovett2015constructive,levy2017deterministic,rothvoss2017constructive} for algorithmic guarantees. 
 
 In the average-case, the entries of $M\in\R^{d\times n}$ are random. When $M$ consists of i.i.d.\,$\cN(0,1)$ entries and $d=o(n)$, it turns out $\mathcal{D}(M)=\Theta(\sqrt{n}2^{-n/d})$ w.h.p. The case $d=1$ corresponds to the \texttt{NPP}; it is established using the \emph{second moment method} by Karmarkar, Karp, Lueker and Odlyzko~\cite{karmarkar1986probabilistic}. 
 As for the higher dimensions, the case $d=O(1)$ is due to Costello~\cite{costello2009balancing} and  $\omega(1)\le d\le o(n)$ is due to Turner, Meka, and Rigollet~\cite{turner2020balancing}. In the \emph{proportional} regime $d=\Theta(n)$, i.e. $d,n\to\infty$ while $d/n$ is held constant say at some $\alpha\in(0,\infty)$,  the model is very related to the symmetric binary perceptron (\texttt{SBP})~\cite{aubin2019storage}. In this case, Perkins and Xu~\cite{perkins2021frozen} and Abbe, Li, and Sly~\cite{abbe2021proof} established concurrently  that $\mathcal{D}(M)=\bigl(1+o(1)\bigr)f(\alpha)\sqrt{n}$ w.h.p.\,for some explicit $f(\cdot)$. For more on the connection between the \texttt{SBP} and discrepancy theory and algorithmic lower bounds, see~\cite{abbe2021binary,gamarnik2022algorithms,kizildag2022algorithms,gamarnik2023geometric}. For average-case discrepancy in other regimes, see~\cite{potukuchi2018discrepancy,hoberg2019fourier,altschuler2022discrepancy}. The discrepancy literature also seeks algorithmic guarantees: devise an efficient algorithm admitting $M$ as its input and returning a $\bs\in\Sigma_n$ for which $\|M\bs\|_\infty$ is small. For relevant liteature, see~\cite{bansal2010constructive,chandrasekaran2014integer,lovett2015constructive,bansalspenceronline,bansal2021online} and the references therein. 
 \paragraph{Algorithms for the \texttt{NPP} and the \texttt{VBP}} 
 Of particular interest to us are algorithms for the \texttt{NPP} and the \texttt{VBP}. The state-of-the-art algorithm for the \texttt{NPP} is due to Karmarkar and Karp's \emph{largest differencing method} (LDM)~\cite{karmarkar1982differencing}: given $X\sim \cN(0,I_n)$, the LDM returns in polynomial-time a $\bs_{\rm ALG}\in\Sigma_n$ such that $\left|\ip{\bs_{\rm ALG}}{X}\right|=2^{-\Theta(\log^2 n)}$ w.h.p.~\cite{yakir1996differencing}. Recalling that $\min_{\bs \in\Sigma_n}\left|\ip{\bs}{X}\right|=\Theta(\sqrt{n}2^{-n})$ w.h.p.~\cite{karmarkar1986probabilistic}, this highlights a striking gap between the existential guarantee and the best known algorithmic guarantee. This is an instance of a \emph{statistical-to-computational gap}, and we elaborate further on this below. Later,~\cite{turner2020balancing} extended this algorithm to the  \texttt{VBP}: for $M\in\R^{d\times n}$ and $d=O(\sqrt{\log n})$, their algorithm returns a $\bs_{\rm ALG}$ with $\|M\bs_{\rm ALG}\|_\infty = 2^{-\Omega(\log^2 n/d)}$ w.h.p. For a more elaborate discussion on algorithms for the \texttt{NPP}/\texttt{VBP}, see~\cite[Section~1]{gamarnik2021algorithmic}.
   \subsubsection{Planted Models}
In this paper, we study a \emph{planted version} of the \texttt{NPP}, where a partition $\bs^*\in\Sigma_n$ is fixed and an $X\sim \cN(0,I_n)$ is sampled conditional on $\left|\ip{\bs^*}{X}\right|\le \sqrt{n}3^{-n}$. This is in contrast with sampling  $X\sim \cN(0,I_n)$ first and studying $\left|\ip{\bs}{X}\right|,\bs\in\Sigma_n$ afterwards. We now elaborate further on the planting idea. As a case in point, consider the random $k$-SAT model. Following~\cite{achlioptas2003threshold}, define a $k$-clause to be a disjunction of $k$ Boolean variables. A random $k$-SAT formula, $\Phi_{k,n,M}$, is then constructed by sampling $M$ independent $k$-clauses $\mathcal{C}_1,\dots,\mathcal{C}_M$ uniformly at random and taking their conjunction: $\Phi_{k,n,M} = \mathcal{C}_1\wedge \mathcal{C}_2\wedge \cdots \wedge \mathcal{C}_M$. This model has been thoroughly investigated; prior work spans from the existence of a satisfying assignment (i.e. a $\bs\in\{0,1\}^n$ with $\Phi_{k,n,M}(\bs)=1$)~\cite{achlioptas2002asymptotic,achlioptas2003threshold,ding2015proof} and the geometry of solution space~\cite{achlioptas2006solution,achlioptas2008algorithmic} to the limits of efficient algorithms~\cite{hetterich2016analysing,coja2017walksat,gamarnik2017performance,huang2021tight}. Note that sampling a satisfying assignment corresponds to studying a probability distribution on the set of assignment-formula pairs $(\bs,\Phi_{k,n,M})$. In a breakthrough paper~\cite{achlioptas2008algorithmic}, Achlioptas and Coja-Oghlan proposed the idea of studying a different distribution on the same set of pairs, dubbed as the planted model: 
first sample a $\bs\in\{0,1\}^n$ uniformly at random and then generate a random $k$-SAT formula $\Phi_{k,n,M}$ consistent with $\bs$ afterwards. That is, sample clauses $\mathcal{C}_1,\dots,\mathcal{C}_M$ independently, conditional on $\Phi_{k,n,M}(\bs)=1$. It turns out that if the two distributions on pairs $(\bs,\Phi_{k,n,M})$ are `close' in a certain sense, then one can obtain results for the unplanted model by studying the planted version instead. Studying the latter is often much easier. By doing so, they showed that the onset of hardness for known efficient algorithms (for random $k$-SAT) roughly corresponds to a phase transition in the geometry of the solution space. This work fueled a still active line of research linking solution space geometry to algorithmic hardness, including the Overlap Gap Property framework employed in this paper. Since then, planted models were employed heavily to investigate the solution space geometry and the condensation threshold of random CSPs~\cite{achlioptas2011solution,montanari2011reconstruction,coja2012condensation,bapst2016condensation,coja2017information}, and to study symmetric Ising perceptron~\cite{perkins2021frozen}. Additionally, planted models also appear frequently in high-dimensional statistics; examples include  hidden clique~\cite{el2022densest}, regression~\cite{wu2021statistical}, stochastic block model~\cite{abbe2017community}, as well as the learning setting~\cite{zdeborova2016statistical,goldt2019dynamics}.

Despite the fact that both the \texttt{NPP} and planted models received much attention from various communities, we are unaware of any prior work attempting to form a link between the two. One of our main motivations in the present paper is to bridge this gap, propose a planted version of the \texttt{NPP} and show that this model exhibits rather interesting features.  
\subsection{Main Results}
Fix $X=(X_1,\dots,X_n)\in \R^n$. For any $\bs\in\Sigma_n\triangleq\{-1,1\}^n$, define its \emph{Hamiltonian} by
\begin{equation}\label{eq:hamiltonian}
H(\bs) = n^{-\frac12}\left|\ip{\bs}{X}\right|.    
\end{equation}
Throughout, all probabilities (unless stated otherwise) are taken w.r.t.\,the planted measure:
\begin{equation}\label{eq:PL-MEASURE}
    \mathbb{P}_{\rm pl}[\mathcal{E}] = \mathbb{P}_{X\sim \cN(0,I_n)}\bigl[\mathcal{E}\big\lvert H(\bs^*)\le 3^{-n}\bigr],\quad\text{where $\mathcal{E}$ is an arbitrary event.}
\end{equation}
\paragraph{Value of Hamiltonian $H(\bs)$ for $\bs\ne \pm \bs^*$} Note that $H(\bs^*)=H(-\bs^*)\le 3^{-n}$ and that planting such a $\bs^*$ affects the entire set of numbers $X_1,\dots,X_n$ to be partitioned. Thus, a very natural first question we ask is: does planting induce partitions $\bs\ne \pm\bs^*$ with Hamiltonian substantially smaller than $2^{-n}$, e.g.\,$H(\bs)\le 2^{-cn}$ for some $1<c<\log_2 3$? 

Our first result shows that, perhaps rather surprisingly, the answer is no.
\begin{theorem}[Informal, see Theorem~\ref{thm:ground-state}]\label{thm:gs-informal}
We have $\min_{\bs\in\Sigma_n\setminus\{\pm \bs^*\}}H(\bs)=\widetilde{\Theta}(2^{-n})$ w.h.p. 
\end{theorem}
Our argument reveals in fact that the number of such $\bs$ grows as $n\to\infty$, see Remark~\ref{remark:NUM-GD-STATE}. Our argument is crucially based on controlling the size of a certain set of partitions via the \emph{second moment method}. This set consists of $\bs\in\Sigma_n$ with $H(\bs) =\widetilde{\Theta}(2^{-n})$, under the additional constraint  that $n^{-1}\left|\ip{\bs}{\bs^*}\right|\le\epsilon$ for an $\epsilon>0$. Namely, it consists of $\bs$ nearly orthogonal to $\pm\bs^*$. 
\paragraph{Hamiltonian at Fixed Distance from $\bs^*$} Theorem~\ref{thm:gs-informal} is established by considering $\bs$ that are nearly orthogonal to $\bs^*$. So, a natural next question is the smallest $H(\bs)$ achieved by $\bs$ at any fixed distance $\rho n$ from $\pm \bs^*$. Answering this would provide a better understanding of p-\texttt{NPP}. To that end, fix a $\rho\in(0,1)$ and let \[
\zeta(\rho)=\min_{\bs:d_H(\bs,\bs^*)=\rho n}H(\bs).
\]Our next result settles the value of $\zeta(\rho)$ for any $\rho$. 
\begin{theorem}[Informal, see Theorem~\ref{thm:zeta-rho}]\label{zeta-rho-inf}
    Fix any $\rho\in(0,1)$. Then, $\zeta(\rho) = \widetilde{\Theta}(2^{-nh_b(\rho)})$ w.h.p.
\end{theorem}
Here, $h_b(\cdot)$ is the binary entropy function logarithm base 2. Note that $\zeta(1/2)$ is indeed of order $\widetilde{\Theta}(2^{-n})$, which is in agreement with the earlier discussion. 

\paragraph{Landscape of the p-\texttt{NPP} and Limits of Efficient Algorithms} For many planted models, the algorithmic task often entails either (a) finding the planted solution itself or (b) finding a solution that `largely agrees' with the planted solution. Note, though, that p-\texttt{NPP} is an optimization problem. Consequently, one can drop the rather ambitious goal of finding $\bs^*$ itself or approximating it, and focus instead on a more modest goal of finding a $\bs$ with a small $H(\bs)$. To explore the algorithmic tractability of this task, we investigate the landscape of p-\texttt{NPP} via the \emph{multi Overlap Gap Property} ($m$-OGP), an intricate geometrical property that has been leveraged to rule out powerful classes of algorithms for many random computational problems, see below. Informally, the $m$-OGP asserts the non-existence of $m$-tuples of nearly equidistant near-optimal solutions, see Definition~\ref{def:overlap-set} for a formal statement.

Our next result establishes the presence of the $m$-OGP below values $2^{-\Theta(n)}$.
\begin{theorem}[Informal, see Theorem~\ref{thm:m-ogp-planted-npp}]\label{thm:m-ogp-informal}
    Fix any $0<\epsilon<1$. The p-\texttt{NPP} exhibits $m$-OGP below $2^{-\epsilon n}$ for appropriately chosen parameters.
\end{theorem}
Theorem~\ref{thm:m-ogp-informal} is the first $m$-OGP result for a planted model. 

We also show in Theorem~\ref{thm:ogp-planted} a certain geometrical property regarding the planted partition: for any $E=\Omega(\log n)$, there exists a value $d=\omega(1)$ such that $\min_{\bs:d_H(\bs,\bs^*)\le d}H(\bs)>2^{-E}$. 
This and Theorem~\ref{zeta-rho-inf} collectively reveal an intriguing property: planted $\bs^*$ and partitions with an `intermediate' Hamiltonian are separated by partitions having a substantially large Hamiltonian. 

Our final result leverages $m$-OGP established in Theorem~\ref{thm:m-ogp-informal} to rule out the class \emph{stable algorithms} with a certain \emph{anti-concentration} property. 
\begin{theorem}[Informal, see Theorem~\ref{thm:stable-hardness}]\label{thm:stable-hard-informal}
    Fix any $\epsilon>0$. Stable algorithms (satisfying a certain anti-concentration property) fail to return a $\bs$ with $H(\bs)\le 2^{-\epsilon n}$ w.p.\,greater than $1-e^{-\Theta(n)}$.  
\end{theorem}
Informally, an algorithm is stable if a small perturbation of its input changes its output (i.e.\,the partition it returns) only slightly. Stable algorithms capture powerful classes of algorithms, including approximate message passing (AMP)~\cite{gamarnikjagannath2021overlap}, low-degree polynomials~\cite{gamarnik2020low,wein2020optimal,bresler2021algorithmic} and Boolean circuits of low-depth~\cite{gamarnik2021circuit}. As for the anticoncentration property, the algorithms ruled out return the planted partition $\bs^*$ with probability at most a certain constant. As we detail in Section~\ref{sec:m-ogp-to-hardness}, this is a benign assumption.  

Theorem~\ref{thm:stable-hard-informal} is the first hardness result against the class of stable algorithms for a planted model. Its proof is based on the $m$-OGP result (Theorem~\ref{thm:m-ogp-informal}), a Ramsey-theoretic argument developed in~\cite{gamarnik2021algorithmic} as well as a certain change of measure argument required due to planting.
\subsection{Further Background and Related Work}
\paragraph{Statistical-to-Computational Gaps (\texttt{SCG}s)}
As we discussed earlier, the u-\texttt{NPP} exhibits an \texttt{SCG}. While $\min_{\bs\in\Sigma_n}H(\bs)=\Theta(2^{-n})$ w.h.p., the best known efficient algorithm performs strictly worse, only returning a $\bs$ with $H(\bs) = 2^{-\Theta(\log^2 n)}$ w.h.p. Observe that a $\bs$ with $H(\bs)=\Theta(2^{-n})$ can be found by brute force. So, this \texttt{SCG} suggests the presence of a regime, $\log^2 n\ll E\ll n$, for which partitions of Hamiltonian $2^{-E}$ exists but finding them may be computationally infeasible. 

Such gaps appear frequently in many algorithmic tasks in the \emph{average-case setting}, i.e.\,when the underlying model involves randomness. Examples include certain random combinatorial structures (e.g.\,random CSPs, random graphs, spin glasses) as well as models in high-dimensional statistics (e.g.\,planted clique, tensor decomposition), see~\cite{gamarnik2022disordered} for a survey. For average-cased models, standard complexity theory is often not helpful.\footnote{The standard complexity theory often deals with the worst-case hardness, see~\cite{ajtai1996generating,boix2021average,GK-SK-AAP} for some exceptions.} Nevertheless, a very fruitful line of research proposed various frameworks to obtain some rigorous evidence of formal hardness. These approaches include average-case reductions (most often from the planted clique~\cite{berthet2013computational,brennan2018reducibility,brennan2019optimal}) as well as unconditional lower bounds against restricted classes of algorithms such as low-degree polynomials~\cite{hopkins2018statistical,kunisky2022notes,wein2022average}, approximate message passing~\cite{zdeborova2016statistical,bandeira2018notes}, statistical query algorithms~\cite{diakonikolas2017statistical,feldman2017statistical,feldman2018complexity}, sum-of-squares hierarchy~\cite{hopkins2015tensor,hopkins2017power,raghavendra2018high,barak2019nearly}, and Monte Carlo Markov Chains (MCMC)~\cite{jerrum1992large,dyer2002counting}. Another such approach is based on the intricate geometry of the problem through the \emph{Overlap Gap Property}. 
\paragraph{Overlap Gap Property (OGP)} The OGP framework leverages intricate geometrical properties of solution space to rule out classes of algorithms. The OGP is a topological property; it informally states that (w.h.p.) any two near-optima are either `close' or `far': there exists $0<\nu_1<\nu_2<1$ such that $n^{-1}\ip{\bs}{\bs'}\in[0,\nu_1]\cup[\nu_2,1]$ for any pair of near-optima $\bs,\bs'\in\Sigma_n$. This property is a rigorous barrier against \emph{stable algorithms}, see below for details. See~\cite{gamarnik2021overlap} for a survey on OGP, and~\cite{mezard2005clustering,achlioptas2006solution,achlioptas2008algorithmic} for earlier works suggesting a potential link between intricate geometry and algorithmic hardness.
\paragraph{Algorithmic Implications of OGP}
The first instance of OGP ruling out algorithms is the problem of finding a large independent set in sparse \ER  random graphs on $n$ vertices with average degree $d$. This model exhibits an \texttt{SCG}: while the largest independent set is of size $2\frac{\log d}{d}n$~\cite{frieze1990independence,frieze1992independence,bayati2010combinatorial}, the best known efficient algorithm finds an independent set half the optimal size, $\frac{\log d}{d}n$~\cite{lauer2007large}. Gamarnik and Sudan~\cite{gamarnik2014limits,gamarnik2017} established that the set of independent sets of size larger than $(1+1/\sqrt{2})\frac{\log d}{d}n$ exhibits the OGP. Consequently, they showed that \emph{local algorithms} fail to find an independent set of size larger than $(1+1/\sqrt{2})\frac{\log d}{d}n$. Since then, the OGP framework has been proven useful for various models (e.g.\,random graphs~\cite{rahman2017local,gamarnik2020low,wein2020optimal}, spin glasses~\cite{chen2019suboptimality,gamarnikjagannath2021overlap,huang2021tight}, random CSPs~\cite{gamarnik2017performance,bresler2021algorithmic,gamarnik2022algorithms}) ruling out other powerful classes of algorithms (e.g.\,low-degree polynomials~\cite{gamarnik2020low,wein2020optimal,bresler2021algorithmic}, AMP~\cite{gamarnikjagannath2021overlap}, low-depth circuits~\cite{gamarnik2021circuit}, Langevin dynamics~\cite{gamarnik2020low,huang2021tight}, and most recently, online algorithms~\cite{gamarnik2023geometric,du2023algorithmic}). 
\paragraph{Multi OGP ($m$-OGP)} The work by Gamarnik and Sudan discussed above rules out local algorithms at value $(1+1/\sqrt{2})\frac{\log d}{d}n$, which is off from the best known algorithmic guarantee by an additional $1/\sqrt{2}$ factor. Subsequent work by Rahman and Vir{\'a}g~\cite{rahman2017local} removed this additional factor by considering a more intricate overlap pattern, involving many large independent sets. This approach is dubbed as the $m$-OGP. Yet another work by Gamarnik and Sudan~\cite{gamarnik2017performance} established the $m$-OGP for $m$-tuples of nearly equidistant solutions and subsequently obtained nearly tight hardness guarantees for the Not-All-Equal $k$-SAT problem. This forbidden overlap pattern is called the symmetric $m$-OGP. Symmetric $m$-OGP was also employed in~\cite{gamarnik2021algorithmic} for the u-\texttt{NPP}, giving some rigorous evidence of hardness below values $2^{-\omega(\sqrt{n\log n})}$; it is also our focus in the present paper. Further, symmetric $m$-OGP yielded nearly tight hardness guarantees for the symmetric Ising perceptron~\cite{gamarnik2022algorithms} and for the $p$-spin model with large $p$~\cite{gjk2023}. Furthermore,~\cite{kizildaug2023sharp} showed that for the Ising $p$-spin and random $k$-SAT models, the $m$-OGP exhibits a certain sharp phase transition.

Recent work put forth more intricate overlap patterns going beyond the symmetric $m$-OGP. These patterns consist of $m$-tuples, where the $i^{\rm th}$ solution has `intermediate' overlap with the first $i-1$ solutions, for any $2\le i\le m$. By doing so, tight hardness guarantees were obtained against low-degree polynomials in the context of (i) finding large independent sets in sparse random graphs~\cite{wein2020optimal} and (ii) finding a satisfying solution to the random $k$-SAT~\cite{bresler2021algorithmic}. Even more recently, Huang and Sellke proposed a very ingenious forbidden structure consisting of an ultrametric tree of solutions, dubbed as the \emph{Branching OGP}~\cite{huang2021tight}. By leveraging the Branching OGP, they obtained tight hardness guarantees for the mixed even $p$-spin model against the class of overlap concentrated algorithms. This class includes $O(1)$ iteration of AMP and gradient descent, as well as the Langevin dynamics running for $O(1)$ time. 
\paragraph{OGP for Planted Models} A variant of the OGP framework applies to planted models arising from statistical inference tasks, including sparse  
 linear regression~\cite{gamarnik2017high}, principal submatrix recovery~\cite{gamarnik2021overlapsubmatrix}, sparse PCA~\cite{arous2020free}, and planted clique~\cite{el2022densest,gamarnik2019landscape}; it gives lower bounds against MCMC and various local search heuristics. The details of this framework, however, are quite different from unplanted models. To set the stage, denote by $\bs^*$ the planted structure and let $\mathcal{L}(\bs)$ be any natural choice of \emph{loss function} (such as the log-likelihood in regression). Set
\[
\Gamma(\mathcal{O}) \triangleq \min_{\bs : \frac1n\ip{\bs}{\bs^*}=\mathcal{O}} \mathcal{L}(\bs). 
\]
That is, $\Gamma(\mathcal{O})$ is the smallest loss achievable by solutions $\bs$ having a fixed overlap $\mathcal{O}$ with the planted $\bs^*$. In this case, the OGP holds iff the function $\mathcal{O}\mapsto \Gamma(\mathcal{O})$, $\mathcal{O}\in(0,1)$, is non-monotonic, see~\cite{gamarnik2017high,arous2020free} for details. In particular, this version of the OGP concerns only with the pairs of solutions of form $(\bs,\bs^*)$, and it does not extend to $m$-tuples. We now highlight that the p-\texttt{NPP} investigated herein does not exhibit this version of OGP. Indeed, recall $\zeta(\cdot)$ from Theorem~\ref{thm:zeta-rho}---which is our choice of canonical loss function---and define
\[
\Gamma(\mathcal{O}) \triangleq \frac1n\log_2\left(\min_{\substack{\bs \in \Sigma_n :\frac1n\ip{\bs}{\bs^*} = \mathcal{O}}}\zeta(\mathcal{O})\right),\quad \mathcal{O}\in(0,1).
\]
Theorem~\ref{thm:zeta-rho} then yields that for any $\mathcal{O}\in(0,1)$, we have that $\Gamma(\mathcal{O}) = -h_b\left(\frac{1-\mathcal{O}}{2}\right)$ which is monotonic. On the other hand, Theorem~\ref{thm:m-ogp-planted-npp} shows that p-\texttt{NPP} does in fact exhibit the version of $m$-OGP discussed earlier. Consequently, one can rule out the class of stable algorithms. 
We find it quite interesting that p-\texttt{NPP} does not exhibit the OGP tailored for planted models, but it does exhibit the $m$-OGP tailored for unplanted models.
\subsection{Open Problems and Future Directions}\label{sec:open-problems}
\paragraph{Algorithmic Guarantees for the p-\texttt{NPP}} We established existential guarantees for the p-\texttt{NPP} and gave evidence of algorithmic hardness in a certain regime, without any corresponding positive algorithmic result. So, a very natural direction is devising efficient algorithms for the p-\texttt{NPP}. One candidate is the aforementioned algorithm by Karmarkar and Karp~\cite{karmarkar1982differencing} for the u-\texttt{NPP}; it would be interesting to explore its performance for the p-\texttt{NPP}. Due to conditioning on a `rare event' $H(\bs^*)\le 3^{-n}$, even a numerical simulation appears challenging. We leave this as an interesting direction for future work.
\paragraph{Hypothesis Testing} Yet another question is whether one can efficiently `distinguish' the p-\texttt{NPP} and the u-\texttt{NPP}. Observe that if $H(\bs^*)\le 2^{-cn}$ for a $c>1$, then the models are statistically distinguishable: values below $2^{-cn}$ (w.h.p.) do not occur in u-\texttt{NPP}. A very interesting question is the computational tractability of this task: is there a polynomial-time algorithm that distinguishes\footnote{Note that a random guessing distinguishes the planted and the unplanted model with probability $1/2$. So, any non-trivial distinguisher should have a success probability strictly greater than $1/2$.
A more challenging task is to distinguish w.p.\,$1-o(1)$.}  p-\texttt{NPP} and u-\texttt{NPP} for $c>1$ sufficiently large? To that end, Theorems~\ref{thm:ground-state} and \ref{thm:zeta-rho} collectively suggest that ignoring the planted $\pm\bs^*$, the p-\texttt{NPP} and the u-\texttt{NPP} behave somewhat similarly. In light of these, we make the following conjecture.
\begin{conjecture}\label{conj:distinguish-impossible}
Fix any $c>1$. There exists no polynomial-time algorithm that distinguishes the p-\texttt{NPP} from the u-\texttt{NPP} with probability greater than $\frac12$. \end{conjecture}
One can employ, e.g.\,the low-degree framework~\cite{kunisky2022notes,schramm2022computational}, to give some rigorous evidence towards Conjecture~\ref{conj:distinguish-impossible}. We leave this for future work.

\paragraph{Different Versions of Planting} Another direction is to investigate different versions of planted \texttt{NPP}. As a concrete example, fix a $\bs^*\in\Sigma_n$ and generate $X_1,\dots,X_n\in\R$ as follows. Let $X_i\sim \cN(0,1)$ be i.i.d.\,for $1\le i\le \lfloor n/2\rfloor$ and generate $X_{\lfloor n/2\rfloor+1},\dots,X_n$  conditional on the event that $\left|\ip{\bs^*}{X}\right|\le \sqrt{n}3^{-n}$, where $X=(X_1,\dots,X_n)\in\R^n$. (We thank Cindy Rush for suggesting this version of planting.) Such an alternative version of planting may in fact facilitate low-degree calculations regarding the aforementioned hypothesis testing problem; we leave this as an interesting direction for future work. 

\paragraph{$m$-OGP for other Planted Models} Planted models such as the planted clique, sparse PCA, and sparse regression are often of estimation/recovery type: they entail either finding the planted structure itself or obtaining  a solution having a good overlap with the planted structure. For this reason, the OGP analysis for planted models was so far limited to pairwise OGP only, where one coordinate of the pairs under investigation is the planted structure. As a result, it only rules out the MCMC methods and certain local algorithms, such as gradient descent. The p-\texttt{NPP}, however, is an optimization problem, where finding a $\bs$ not necessarily equal to $\pm \bs^*$ with a small $H(\bs)$ is still of algorithmic interest. For this model, Theorems~\ref{thm:m-ogp-planted-npp} and~\ref{thm:stable-hardness} collectively yield some evidence of hardness against stable algorithms targeting to find solutions not necessarily equal to $\pm \bs^*$. In light of these, the $m$-OGP approach is still useful for planted models that are of `optimization type'. A very interesting question is whether an $m$-OGP approach can rule out stable algorithms in a much more challenging planted estimation/recovery setting. For such models, we believe our approach might be of help as we elaborate now. As a case in point, consider the \emph{sparse regression} problem: efficiently recover a sparse  $\beta^*\in\R^p$ from its noisy linear measurements $Y=X\beta^*+W\in\R^n$, where the measurement matrix $X\in\R^{n\times p}$ and the noise vector $W\in\R^n$ consist of, say, i.i.d.\,sub-Gaussian entries. When the sample size $n$ is below a certain value (determined by the number of features $p$, the sparsity level $\|\beta^*\|_0$ and noise strength $\mathbb{E}[W_1^2]$), this model exhibits the aforementioned version of the OGP tailored for planted models~\cite{gamarnik2017high} w.r.t.\,the canonical loss function $\mathcal{L}(\beta) = \|Y-X\beta\|_2^2$. Consequently, various local search algorithms fail to recover $\beta^*$. For the purposes of the optimization, it is natural to consider the sublevel sets of $\mathcal{L}$. Now, suppose that there exists a $c_0$ such that the set of such sparse $\beta$ with $\mathcal{L}(\beta)\le c_0$ exhibits the $m$-OGP  in the sense of Definition~\ref{def:overlap-set} for a suitable $m$. One would then be able to show that stable algorithms with a similar anti-concentration property fail to find a sparse $\beta_{\rm ALG}$ with  $\mathcal{L}(\beta_{\rm ALG})\le c_0$. In light of the results of~\cite{gamarnik2017high}, this would imply that the (normalized) overlap between any such $\beta_{\rm ALG}$ and the planted $\beta^*$ is bounded away from $1$. It appears challenging though to establish the  $m$-OGP for appropriate sublevel sets of $\mathcal{L}$. The  situation is even more subtle for the planted clique model, since even the choice of a canonical loss function is not clear to begin with. We leave these as very interesting future directions. 
\paragraph{$m$-OGP for Different Forbidden Structures} Theorems~\ref{thm:m-ogp-planted-npp} and~\ref{thm:stable-hardness} give some rigorous evidence of hardness for the algorithmic problem of efficiently finding a $\bs$ with $H(\bs)\le 2^{-\Theta(n)}$. While we do not know the best algorithmic guarantee, we anticipate that the p-\texttt{NPP} exhibits an \texttt{SCG} similar to the u-\texttt{NPP}. In light of these, we ask the following: is there an $f:\N\to\R^+$ with $f\in o(n)$ as $n\to\infty$ such that the set of $\bs$ with $H(\bs)\le 2^{-f(n)}$ exhibits $m$-OGP for a different forbidden structure? In fact, this direction was already undertaken in~\cite{gamarnik2021algorithmic}; the presence of $m$-OGP with growing $m$, $m=\omega(1)$, below $2^{-\omega(\sqrt{n\log n})}$ was established. While it is plausible that a similar analysis may show the presence of $m$-OGP also below $2^{-\omega(\sqrt{n\log n})}$, we do not pursue this for simplicity. A far more interesting question is whether one can `break the $\sqrt{n\log n}$ barrier' altogether, see~\cite[Section~4.1]{gamarnik2021algorithmic} for details. One approach is to employ a more intricate overlap pattern, such as~\cite{wein2020optimal,bresler2021algorithmic} or the \emph{Branching OGP}~\cite{huang2021tight}.
\paragraph{Improving the Success Guarantee}Theorem~\ref{thm:stable-hardness} leverages the $m$-OGP with $m=O(1)$ to rule out algorithms succeeding w.p.\,at least $1-e^{-\Theta(n)}$. This is an `artifact' of a certain change of measure argument required for planting; prior work leveraging $m$-OGP with constant $m$ ruled out algorithms succeeding with a constant probability~\cite{gamarnik2021algorithmic,gamarnik2022algorithms}. An interesting technical question is improving the guarantee in Theorem~\ref{thm:stable-hardness} to a constant probability of success.
 \subsubsection*{Paper Organization}
 The rest of the paper is organized as follows. In Section~\ref{sec:energy}, we provide our results regarding the value of Hamiltonian. See Section~\ref{sec:ground} for the results on the ground-state (other than $\pm \bs^*$) and Section~\ref{sec:hamilton-fixed} for the values of Hamiltonian at fixed distance from the planted solution. Section~\ref{sec:landscape} is devoted to our landscape results. See in particular Section~\ref{sec:planted-iso} for landscape near the planted solution, and Section~\ref{sec:m-ogp} for our $m$-OGP results. We present our algorithmic hardness results in Section~\ref{sec:hardness}. Finally, we provide complete proofs in Section~\ref{sec:proofs}.
\section{Energy Levels of p-\texttt{NPP}}\label{sec:energy}
We commence this section by recalling the \emph{planted} model: sample an $X\sim \cN(0,I_n)$ conditional on $H(\bs^*)\le 3^{-n}$, where $\bs^*\in\Sigma_n$ is fixed and $H(\cdot)$ is the Hamiltonian~\eqref{eq:hamiltonian}. Recall that in the absence of planting, i.e.\,$X_i\sim \cN(0,I_n)$ unconditionally, $\min_{\bs\in\Sigma_n}H(\bs)=\Theta(2^{-n})$ w.h.p.~\cite{karmarkar1986probabilistic}. In particular, under the planted model, the Hamiltonian for both $\bs^*$ and $-\bs^*$ (due to symmetry) is at most $3^{-n}$, and the value $3^{-n}$ w.h.p.\,does not occur in the unplanted model. 

In this section, we investigate the values of $H(\bs)$ for $\bs \ne \pm \bs^*$. Our first focus is on the \emph{ground-states} of p-\texttt{NPP} other than $\pm \bs^*$: 
\[
\min_{\bs \in\Sigma_n\setminus\{\pm \bs^*\}}H(\bs).
\]
\subsection{Ground-States of p-\texttt{NPP} beyond Planted Partition}\label{sec:ground}
In light of the preceding discussion, a very natural question is whether planting induces $\bs\ne \pm\bs^*$ with $H(\bs)$ substantially smaller than $2^{-n}$, i.e.\,the smallest value  achievable in the unplanted model. We show that, perhaps rather surprisingly, the answer to this question is {\bf negative}. 
\begin{theorem}\label{thm:ground-state}
    Let $f(n):\mathbb{N}\to\R^+$ and $g(n):\mathbb{N}\to\R^+$ be arbitrary functions with $f(n) = o_n(1)$ and $g(n)=\omega_n(1)$ as $n\to\infty$. Then, w.h.p.\,under the planted measure~\eqref{eq:PL-MEASURE},
    \[
    f(n)2^{-n}\le \min_{\bs\in\Sigma_n\setminus
    \{\pm\bs^*\}} H(\bs)\le g(n)2^{-n}.
    \]
\end{theorem}
We prove Theorem~\ref{thm:ground-state} in Section~\ref{sec:pf-gs}. Several remarks are now in order.

Theorem~\ref{thm:ground-state} shows that $\min_{\bs \in\Sigma_n\setminus\{\pm \bs^*\}}H(\bs)=\widetilde{\Theta}(2^{-n})$ w.h.p. This is a somewhat surprising conclusion: `planting' a partition $\bs^*$ with a substantially smaller $H(\bs^*)$ does not induce partitions $\bs\ne \pm\bs^*$ with $H(\bs)$ much smaller than $2^{-n}$, which is the smallest value achievable in the unplanted model. Furthermore, the choice of $3^{-n}$ is merely for convenience: the same result holds when $H(\bs^*)\le 2^{-E}$ for any arbitrary $E$, however large.

\paragraph{Comparison with the Planted Clique Model} We now highlight a contrast between the p-\texttt{NPP} and the \emph{planted clique} model~\cite{jerrum1992large,berthet2013computational,deshpande2015improved,feldman2017statistical,barak2019nearly}, an important canonical model in the study of average-case complexity. The planted clique model is defined as follows.  Consider the \ER random graph $\mathbb{G}(n,\frac12)$ whose largest clique is of size $2\log_2 n$ w.h.p.\footnote{This result is a simple application of the moment method, see e.g.~\cite{alon2016probabilistic} for details.} Select a subset of $k$ vertices and plant all $\binom{k}{2}$ edges in between. That is, plant a clique of size $k$. Key questions regarding the planting clique model include (a) for what values of $k$ presence of a planted clique can be detected and (b) for what values of $k$, the task in ${\rm (a)}$ can be done efficiently. The regime of interest for this model is $2\log_2 n\ll k \ll \sqrt{n}$: the planted clique model is statistically distinguishable from $\mathbb{G}(n,\frac12)$ iff $k\gg 2\log_2 n$, whereas no efficient algorithm that can do so is known when $k\ll \sqrt{n}$, see the references above. Observe that planting a clique of size $k\gg 2\log_2 n$ does induce cliques of size much larger than $2\log_2 n$. Indeed, since any subset of the planted clique is a clique by itself, planting planting induces cliques of size $\ell$ for any $\ell\in\{2\log_2 n,2\log_2 n+1,\dots,k-1,k\}$. Interestingly though, per Theorem~\ref{thm:ground-state}, this is not the case for the p-\texttt{NPP}. This is quite surprising, since planting a partition with $H(\bs^*)$ is expected to impact the global structure of the model (whereas planting a clique of size say $\Omega(\log n)$ to an $n$-vertex graph is  a `local change').

\paragraph{Proof Technique for Theorem~\ref{thm:ground-state}} The proof is based on the \emph{first} and the \emph{second moment methods}. To that end, denote by $S_{\varphi}$ the set of all $\bs\in\Sigma_n\setminus\{\pm \bs^*\}$ such that $H(\bs)\le \varphi(n)2^{-n}$, where $\varphi:\N\to\R^+$ is arbitrary. We first show that for $f=o_n(1)$, $\mathbb{E}_{\rm pl}|S_f| = o_n(1)$. Consequently, Markov's inequality yields that $S_f=\varnothing$ w.h.p., establishing the lower bound. This technique is known as the \emph{first moment method}. The upper bound, on the other hand, is much more challenging and in particular based on the \emph{second moment method}: if $N$ is an integer-valued random variable that is almost surely non-negative, then $\mathbb{P}[N\ge 1]\ge \mathbb{E}[N]^2/\mathbb{E}[N^2]$. This fact is known as Paley-Zygmund inequality~\cite{alon2016probabilistic}\footnote{For a simple proof, see discussion following~\eqref{eq:2nd-mom-met}.}. In particular,
\[
\mathbb{E}[N^2]=(1+o_n(1))\mathbb{E}[N]^2 \implies   \mathbb{P}[N\ge 1]=1-o_n(1).
\]
Note that the upper bound in Theorem~\ref{thm:ground-state} is equivalent to showing $|S_g|\ge 1$ w.h.p. In order to apply the second moment method though, one has to control $\mathbb{E}_{\rm pl}[|S_g|^2]$, which unfortunately appears quite challenging. Controlling $\mathbb{E}_{\rm pl}[|S_g|^2]$ entails a certain sum over all $\bs_1,\bs_2\ne \pm \bs^*$, where one has to study all potential values of $n^{-1}\ip{\bs_i}{\bs_j}$, $1\le i<j\le 3$ for $\bs_3 = \bs^*$. This is required for bounding the determinant of a certain covariance matrix and controlling probability terms. In order to circumvent this issue, we identify a convenient subset of $S_g$ for which estimating the second moment of its cardinality is much easier. To that end, fix an $\epsilon>0$ and denote by $S_g(\epsilon)$ the set of all $\bs\in\Sigma_n$ with $H(\bs)\le g(n)2^{-n}$ under the additional constraint that $n^{-1}\ip{\bs}{\bs^*}\in [-\epsilon,\epsilon]$. Namely, $S_g(\epsilon)$ is the set of $\bs$ that are nearly orthogonal to $\bs^*$. Note that $S_g(\epsilon)\subseteq S_g$ and therefore $\mathbb{P}_{\rm pl}[|S_g|\ge 1]\ge \mathbb{P}_{\rm pl}[|S_g(\epsilon)|\ge 1]$. On the other hand, it turns out that estimating $\mathbb{E}_{\rm pl}[|S_g(\epsilon)|^2]$ is more tractable; the `dominant contribution' comes from pairs $(\bs_1,\bs_2)$ that are also nearly orthogonal, i.e.\,$n^{-1}\ip{\bs_1}{\bs_2}\in[-\epsilon,\epsilon]$. 
This calculation is based on certain ideas developed by Gamarnik and the author in~\cite{gamarnik2021algorithmic},
see Section~\ref{sec:pf-gs} for details.
\begin{remark}\label{remark:NUM-GD-STATE}
    We now highlight that the number of partitions $\bs$ with $H(\bs)$ of order $2^{-n}$ grows as $n\to\infty$. Take any random variable $N$ for which the second moment method is `tight': $\mathbb{E}[N^2] = \mathbb{E}[N]^2\bigl(1+o_n(1)\bigr)$. A more refined version of the Paley-Zygmund inequality~\cite{alon2016probabilistic} yields
\begin{equation}\label{eq:pz-refined}
    \mathbb{P}\bigl[N>\theta\mathbb{E}[N]\bigr]\ge (1-\theta)^2 \frac{\mathbb{E}[N]^2}{\mathbb{E}[N^2]},\quad \forall \theta\in[0,1].
\end{equation}
Letting $N=|S_g|$ and $\theta=f$ in~\eqref{eq:pz-refined}, where $f=o_n(1)$ is a function that grows arbitrarily slow, we obtain that w.h.p.\,$\log |S_g| = \log \mathbb{E}_{\rm pl}[|S_g|]\bigl(1+o_n(1)\bigr)$. So, $|S_g|=\omega(1)$.
\end{remark}
\subsection{Hamiltonian at Fixed Distance from Planted Partition}\label{sec:hamilton-fixed}
As we mentioned above, we prove Theorem~\ref{thm:ground-state} by showing that there exists in fact a $\bs$ with $d_H(\bs,\bs^*)\sim n/2$ such that $H(\bs) \le g(n)2^{-n}$. This prompts a quite natural question: what is the smallest $H(\bs)$ over all $\bs$ that are at a certain fixed distance from $\bs^*$? This would evidently provide a more detailed understanding of the random objective function $H(\bs)$ we study.

To that end, fix an arbitrary $\rho\in(0,1)$ and let
\begin{equation}\label{eq:zeta-rho}
   \zeta(\rho)\triangleq \min_{\substack{\bs\in\Sigma_n\setminus\{\pm \bs^*\}\\ d_H(\bs,\bs^*)=\rho n}} H(\bs).
\end{equation}
Our next result establishes the value of $\zeta(\rho)$ for any $\rho$. 
\begin{theorem}\label{thm:zeta-rho}
    Fix any $\rho\in(0,1)$, $f(n):\mathbb{N}\to\R^+$ and $g(n):\mathbb{N}\to\R^+$, where $f$ and $g$ are arbitrary with $f(n) = o_n(1)$ and $g(n)=\omega_n(1)$ as $n\to\infty$. Then,
    \[
    f(n)\sqrt{n}2^{-nh_b(\rho)}\le \zeta(\rho)\le g(n)\sqrt{n}2^{-nh_b(\rho)},
    \]
    w.h.p.\,under the planted measure~\eqref{eq:PL-MEASURE}, where $h_b$ is the binary entropy function logarithm base 2.
\end{theorem}
See Section~\ref{sec:pf-zeta-rho} for the proof of Theorem~\ref{thm:zeta-rho}. Several remarks are now in order.

Theorem~\ref{thm:zeta-rho} establishes that $\zeta(\rho) = \widetilde{\Theta}(2^{-nh_b(\rho)})$ w.h.p. The $\sqrt{n}$ factor, which does not appear in Theorem~\ref{thm:ground-state}, is an artifact of the Stirling's approximation to binomial coefficients, 
see Lemma~\ref{lemma:binom-entropy}; it can easily be removed if one instead considers $\min_{\bs \in B(\rho,\gamma)}H(\bs)$ for a sufficiently small $\gamma>0$, where $B(\rho,\gamma)=\{\bs\in\Sigma_n\setminus\{\pm \bs^*\}:(\rho-\gamma)n \le d_H(\bs,\bs^*)\le \rho n\}$.

The proof of Theorem~\ref{thm:zeta-rho} is similar to that of Theorem~\ref{thm:ground-state}. Namely, the lower bound is established via the first moment method and the upper bound is established via the second moment method. In order the control the second moment, we show that the dominant contribution to it comes from pairs $(\bs_1,\bs_2)$ for which $n^{-1}\ip{\bs_1}{\bs_2}\approx \bigl(n^{-1}\ip{\bs_1}{\bs^*}\bigr)^2 = (1-2\rho)^2$. At a technical level, this amounts to controlling a certain sum involving product of binomial coefficients, see Lemma~\ref{lemma:sum-binom} for details.

\begin{remark}\label{remark:exp-many}
    Using Theorem~\ref{thm:zeta-rho}, Paley-Zygmund inequality~\eqref{eq:pz-refined}, and Markov's inequality, we can in fact establish that the number of $\bs$ with $H(\bs) \le 2^{-\epsilon n}$ is exponential in $n$ for any $\epsilon\in(0,1)$.
\end{remark}
   \section{Landscape of p-\texttt{NPP} and the Overlap Gap Property}\label{sec:landscape}
In this section, we investigate the `landscape' of the p-\texttt{NPP} and the geometry of the set of its near-optima. Our first focus is on the landscape near $\pm\bs^*$.
\subsection{Landscape Near the Planted Partition}\label{sec:planted-iso}
Notice that Theorems~\ref{thm:ground-state} and~\ref{thm:zeta-rho} collectively imply an intriguing property: for any $\epsilon>0$, there is a $d_\epsilon>0$ such that w.h.p.\,there exists no $\bs$ with $H(\bs)\le 2^{-\epsilon n}$ that is within $d_\epsilon n$ from $\bs^*$ in Hamming distance. Namely, partitions $\bs$ of $H(\bs)\le 2^{-\epsilon n}$ are `far' from the planted $\bs^*$. 

We next extend this observation to 
 a much broader range of exponents. Define
 \begin{equation}\label{eq:S-D-E}
         S^*(d,E) = \bigl\{\bs\ne\pm\bs^*:d_H(\bs,\bs^*)\le d,H(\bs)\le 2^{-E}\bigr\}.
 \end{equation}
 Note that if $S^*(d,E)=\varnothing$, then $H(\bs)>2^{-E}$ for any $\bs$ with $d_H(\bs,\bs^*)\le d$. Our next result is as follows.
 \begin{theorem}\label{thm:ogp-planted}
 For any  $\epsilon>0$, there exists a $\beta\in(0,\frac12]$ such that 
        \[
        \mathbb{P}\bigl[S^*(\beta n,\epsilon n)=\varnothing\bigr]\ge 1-2^{-\Theta(n)}.
        \]
More generally, for any $E=\Omega(\log n)$, there exists a $d=\omega(1)$ such that
        \[
        \mathbb{P}\bigl[S^*(d,E)=\varnothing\bigr]\ge 1-2^{-\Theta(E)}.
        \]
\end{theorem}
Theorem~\ref{thm:ogp-planted} is established via a simple application of the \emph{first moment method}, see Section~\ref{sec:ogp-planted} for its proof. Several remarks are now in order.

Per Theorem~\ref{thm:ogp-planted}, for every $E=\Omega(\log n)$, there exists a $d$ such that w.h.p.\,$H(\bs)\ge 2^{-E}$ for every $\bs$ with $1\le d_H(\bs,\bs^*)\le d$. Further, when $E=\Theta(n)$, $d$ is also linear in $n$, $d=\Theta(n)$. Now, observe that Theorems~\ref{thm:ground-state}-\ref{thm:zeta-rho} collectively yield that for any $\epsilon \in(0,1]$, there exists (w.h.p.) $\bs\in\Sigma_n$ with $H(\bs)=\widetilde{\Theta}(2^{-\epsilon n})$.  In light of Theorem~\ref{thm:ogp-planted}, we thus arrive at the following rather intriguing conclusion: there exists partitions $\bs$ of substantially large $H(\bs)$ separating the planted partition $\bs^*$ from partitions $\bs$ with `intermediate' $H(\bs)$. Namely, the planted partition is `isolated' in a certain sense, suggesting that finding it is likely to be algorithmically hard. 

\paragraph{Parameter Scaling} We now highlight the scaling of parameters. Our analysis reveals that if $\epsilon>0$, then any $\beta$ with $h_b(\beta)<\epsilon$ works. Observe that this is consistent with Theorem~\ref{thm:zeta-rho}. When $d=\Omega(\log n)$, it suffices to take $d=cE/\log_2(n/E)$, for a suitable $c\in(0,1)$. As a sanity check, when $E=\Theta(n)$, we indeed have $d=\Theta(E\cdot\log^{-1}(n/E)) = \Theta(n)$. Importantly, the probability guarantee in part ${\rm (b)}$ is no longer exponential in $n$, but rather exponential in $E$ itself. This is an artifact of our proof technique via the first moment method.

In light of Theorem~\ref{thm:ogp-planted}, finding the planted partition $\bs^*$ is likely challenging, as it is isolated with `energy barriers'. Notice though that the p-\texttt{NPP} is an optimization problem, thus one can focus on a more modest algorithmic goal of finding a $\bs$, not necessarily equal to $\pm \bs^*$, for which $H(\bs)$ is small. Our next focus is on the tractability of this problem via landscape analysis.

\subsection{Multi Overlap Gap Property for p-\texttt{NPP}}\label{sec:m-ogp}\label{sec:ogp}
In this section, we focus on the algorithmic problem of efficiently finding a partition $\bs$ with a small Hamiltonian $H(\bs)$, where $\bs$ is not necessarily equal to  $\pm\bs^*$. In order to understand the tractability of this algorithmic problem, we investigate the landscape of p-\texttt{NPP}. That is, we study the `geometry' of the set of near-optima of p-\texttt{NPP}. To that end, define 
\begin{equation}\label{eq:hamiltonian-with-disorder}
    H(\bs,X) \triangleq \frac{1}{\sqrt{n}}\left|\ip{\bs}{X}\right|,
\end{equation}
where the dependence of Hamiltonian $H$ to $X$ is made explicit. We first formalize the set of near-optimal tuples we investigate.
\begin{definition}\label{def:overlap-set}
   Fix a $\bs^*\in\Sigma_n$, an $m\in\mathbb{N}$, $0<\eta<\beta<1$, an $E:\N\to\R^+$, and a $\mathcal{I}\subset [0,\pi/2]$. Let $X_i\sim \cN(0,I_n)$, $0\le i\le m$ be i.i.d.\,conditional on the event that $\max_{0\le i\le m} H(\bs^*,X_i)\le 3^{-n}$. Denote by $\mathcal{S}(m,\beta,\eta,E,\mathcal{I})$ the set of all $\bs_1,\dots,\bs_m\in\Sigma_n\setminus\{\pm \bs^*\}$ satisfying the following. 
   \begin{itemize}
       \item {\bf (Pairwise Overlap Condition)} For any $1\le i<j\le m$, 
       \[
       \beta-\eta\le n^{-1}\ip{\bs_i}{\bs_j}\le \beta.
       \]
       \item {\bf (Near-Optimality)} There exists $\tau_i\in \mathcal{I}$, $1\le i\le m$, such that
       \[
       \max_{1\le i\le m}H\bigl(\bs_i,Y_i(\tau_i)\bigr)\le 2^{-E},\quad\text{where}\quad Y_i(\tau) = \cos(\tau)X_0 + \sin(\tau)X_i\in\R^n. 
       \]
       \end{itemize}
        Furthermore, let $\mathcal{S}^*(m,\beta,\eta,E,\mathcal{I})$ denotes the same set where $\bs_i\in \Sigma_n$ ($\bs_i = \pm\bs^*$ allowed).
\end{definition}
Several remarks are now in order. Definition~\ref{def:overlap-set} regards $m$-tuples $\bs_1,\dots,\bs_m$ such that (a) $\bs_i$ are near-optimal, i.e.\,$H(\bs_i)$ is at most $2^{-E}$ and (b) overlaps between $\bs_i$ and $\bs_j$ are constrained to the interval $[\beta-\eta,\beta]$. In what follows, we employ Definition~\ref{def:overlap-set} with $\beta\gg\eta$: $(\bs_1,\dots,\bs_m)$ is a nearly equidistant $m$-tuple of near-optima with pairwise Hamming distances around $n\frac{1-\beta}{2}$. Moreover, $\bs_i$ need not be near-optimal with respect to the same instance as $Y_i(\tau_i)$ are correlated. The set $\mathcal{I}$ is needed for describing the correlated instances. Lastly, we trivially have $\mathcal{S}(m,\beta,\eta,E,\mathcal{I})\subseteq \mathcal{S}^*(m,\beta,\eta,E,\mathcal{I})$.

Equipped with these, we establish that p-\texttt{NPP} exhibits the Ensemble multi Overlap Gap Property (Ensemble $m$-OGP). 

\begin{theorem}\label{thm:m-ogp-planted-npp}
Fix any $\epsilon>0$ and any $\delta<\epsilon$. There exists a $c>0$, an $m\in\mathbb{N}$, and $0<\eta<\beta<1$ such that the following holds. For $X_0,\dots,X_m$ generated as in Definition~\ref{def:overlap-set} and any $\mathcal{I}\subset [0,\pi/2]$ with $|\mathcal{I}|\le 2^{cn}$ and $\min_{x\in\mathcal{I}\setminus\{0\}}x\ge 2^{-\delta n}$, we have
\[
\mathbb{P}_{\rm pl}\Bigl[\mathcal{S}^*(m,\beta,\eta,\epsilon n,\mathcal{I}) = \varnothing\Bigr]\ge 1-2^{-\Omega(n)},\quad\text{and}\quad \mathbb{P}_{\rm pl}\Bigl[\mathcal{S}(m,\beta,\eta,\epsilon n,\mathcal{I}) = \varnothing\Bigr]\ge 1-2^{-\Omega(n)}.
\]
 \end{theorem}
We prove Theorem~\ref{thm:m-ogp-planted-npp} in Section~\ref{sec:proof-m-ogp}. Certain remarks are now in order.

Theorem~\ref{thm:m-ogp-planted-npp} is the first $m$-OGP result for a planted model, it establishes that for any $\epsilon>0$, there exists an $m\in\N$ and $0<\eta<\beta<1$ such that (w.h.p.) there do not exist an $m$-tuple $\bs_1,\dots,\bs_m$ each with Hamiltonian at most $2^{-\epsilon n}$ and pairwise overlaps $n^{-1}\ip{\bs_i}{\bs_j}\in[\beta-\eta,\beta]$. Our analysis reveals $\beta\gg \eta$, so Theorem~\ref{thm:m-ogp-planted-npp} rules out nearly equidistant $m$-tuples. Moreover, $\bs_i$ need not be near-optimal w.r.t.\,the same instance of disorder: Definition~\ref{def:overlap-set} and Theorem~\ref{thm:m-ogp-planted-npp} allow correlated instances. This is known as the \emph{Ensemble} $m$-OGP; it will be instrumental in ruling out stable algorithms later in Theorem~\ref{thm:stable-hardness}. 

Our next remark regards the technical assumptions $|\mathcal{I}|\le 2^{cn}$ and $\min_{x\in\mathcal{I},x\ne 0}x \ge 2^{-\delta n}$. In order to rule out algorithms, we employ Theorem~\ref{thm:m-ogp-planted-npp} with $\mathcal{I} = \{\tau_0,\dots,\tau_Q\}$, where $0=\tau_0<\tau_1<\cdots<\tau_Q=\pi/2$ are equispaced and $Q$ is a large constant, see~\eqref{eq:some-params}. In particular, $|\mathcal{I}|=Q=O(1)$ and $\min_{0\ne x\in\mathcal{I}} x = \tau_1 = \Theta(Q^{-1}) = O(1)$, so the assumptions are satisfied automatically. 

\paragraph{Parameter Scaling} An inspection of our analysis reveals that $m=\Theta_\epsilon(1/\epsilon)$. In particular, it remains constant in $n$. In other words, by inspecting $m$-tuples, one can `rule out' energy levels $2^{-E}$ for $E=\Theta(n/m)$.  It is conceivable that by considering $m$-tuples with growing $m$, $m=\omega_n(1)$, one can rule out a more broad range of exponents. For the unplanted \texttt{NPP}, this was done in~\cite{gamarnik2021algorithmic}; we however do not pursue this improvement in the present paper for simplicity. 

\paragraph{Proof Technique} The proof of Theorem~\ref{thm:m-ogp-planted-npp} is based on the first moment method. Specifically, we fix an $\epsilon>0$ and show that $\mathbb{E}_{\rm pl}|\mathcal{S}^*(m,\beta,\eta,\epsilon n,\mathcal{I})| = o(1)$ for suitable $m,\beta$ and $\eta$. Markov's inequality then yields Theorem~\ref{thm:m-ogp-planted-npp}. However, the presence of planting and the need for studying correlated instances make it challenging to control the first moment. We bypass these obstacles by using a careful conditioning argument, see the proof in Section~\ref{sec:proof-m-ogp} for further details.

\section{Algorithmic Barriers in p-\texttt{NPP}}\label{sec:hardness}
In this section, we leverage our $m$-OGP result, Theorem~\ref{thm:m-ogp-planted-npp}, to establish a hardness result against the class of \emph{stable algorithms}. We begin by detailing the algorithmic setting we consider.
\subsection{Algorithmic Setting}
The setting we consider is almost identical to~\cite{gamarnik2021algorithmic}. An algorithm $\A$ is a map from $\R^n$ to $\Sigma_n$, where randomization is allowed. That is, we assume there exists a probability space $(\Omega,\mathbb{P}_\omega)$ such that $\A:\R^n\times \Omega\to\Sigma_n$, and that for $\omega\in\Omega$ and $X\in\R^n$, $\A$ returns a partition $\bs_{\rm ALG}\triangleq\A(X,\omega)\in\Sigma_n$. 

We now formalize the class of \emph{stable algorithms} that we focus. 
\begin{definition}\label{def:stable-algs}
Fix an $E:\N\to\R^+$. An algorithm $\A:\R^n\times \Omega\to\Sigma_n$ for the p-\texttt{NPP} is called 
\[
\bigl(E,p_f,p_{\rm st},p_\ell,\rho,f,L\bigr)\text{-stable}
\]
 if there exists an $N_0\in \N$ such that $\A$ satisfies the following for all $n\ge N_0$.
\begin{itemize}
    \item {\bf (Success)} Let $X\sim \cN(0,I_n)$. Then,
    \[
    \min_{C\in\{1,\sqrt{2}\}}\mathbb{P}_{X,\omega}\Bigl[H\bigl(\A(X,\omega),X\bigr)\le 2^{-E}\Big\lvert H(\bs^*,X)\le C\cdot 3^{-n}\Bigr]\ge 1-p_f.
    \]
    \item {\bf (Stability)} Let $X,Y\distr \cN(0,I_n)$ with $\mathbb{E}\bigl[XY^T\bigr] = \rho I_n$.\footnote{We treat $X$ and $Y$ as column vectors.} Then,
    \[
   \min_{C\in\{1,\sqrt{2}\}}\mathbb{P}_{X,Y,\omega}\Bigl[d_H\bigl(\A(X,\omega),\A(Y,\omega)\bigr)\le f+L\|X-Y\|_2^2\Big\lvert H(\bs^*,X),H(\bs^*,Y)\le C\cdot 3^{-n}\Bigr]\ge 1-p_{\rm st}.
    \]
    \item {\bf (Anti-Concentration)} Let $X\distr \cN(0,I_n)$. Then,
    \[
\mathbb{P}_{X,\omega}\bigl[\A(X,\omega)\ne \pm\bs^*\big\lvert H(\bs^*,X)\le 3^{-n}\bigr]\ge 1-p_\ell.
    \]
\end{itemize}
\end{definition}
Definition~\ref{def:stable-algs} is similar to those 
considered in~\cite[Definition~3.1]{gamarnik2022algorithms} and~\cite[Definition~3.1]{gamarnik2021algorithmic}; it applies also to deterministic algorithms. In that case, the sole source of randomness is the input $X$ and its `perturbation' $Y$, hence the probability terms depend only on $X,Y$. In what follows, we drop $\omega$ and simply refer to an $\A:\R^n\to \Sigma_n$  as a randomized algorithm. 

 We next elaborate on parameters in Definition~\ref{def:stable-algs}. Parameter $E$ is the target exponent and $p_f$ is the failure probability: w.p.\,at least $1-p_f$ under the planted measure, the algorithm returns a $\bs_{\rm ALG}\triangleq\A(X)\in\Sigma_n$ with $H(\bs_{\rm ALG},X)\le 2^{-E}$. Parameters $p_{\rm st},\rho,f$ and $L$ collectively control the stability guarantee. When $X$ and $Y$ are Gaussian with correlation $\rho$, Hamming distance between $\A(X)$ and $\A(Y)$ is controlled by an explicit quantity, with probability at least $1-p_{\rm st}$. The parameters $f$ and $L$ quantify the sensitivity of the output of $\A$ to the changes in its input. Lastly, we require our algorithm to satisfy a certain `anti-concentration' probability: $\A(X)=\pm \bs^*$ with probability at most $p_\ell$. Our result below rules out $\A$ with $p_\ell=O(1)$. As we elaborate later, this is a benign assumption.

The assumption that $\A$ operates conditional on the objective values both $3^{-n}$ and $\sqrt{2}\cdot 3^{-n}$ is required for technical reasons, and again quite benig since these values are of the same asymptotic order. That is, any $\A$ succeeding under one value would likely succeed under the other.
\subsection{Multi Overlap Gap Property Implies Failure of Stable Algorithms}\label{sec:m-ogp-to-hardness}
We now show that stable algorithms fail to return a $\bs$ with $H(\bs,X)\le 2^{-\Theta(n)}$ for the p-\texttt{NPP}. 
\begin{theorem}\label{thm:stable-hardness}
    Fix any $\epsilon>0$ and $L>0$. Let $m\in\mathbb{N}$, $0<\eta<\beta<1$ be the $m$-OGP parameters prescribed by Theorem~\ref{thm:m-ogp-planted-npp}. Set
    \begin{equation}\label{eq:some-params}
        C=\frac{\eta^2}{1600}, \quad Q=\frac{40C_2\pi\sqrt{L}}{\eta} ,\quad\text{and}\quad T=\exp_2\left(2^{4mQ\log_2 Q}\right),
    \end{equation}
    where $C_2>0$ is a sufficiently large constant. Then, there exists an $n_0\in\mathbb{N}$ such that the following holds. For every $n\ge n_0$, there exists no randomized algorithm $\A:\R^n\times \Omega\to\Sigma_n$ that is
    \[
    \left(\epsilon n,\frac{3^{-n}}{108QT\sqrt{\pi}},\frac{\pi}{648Q^2 T},\frac{1}{54T},\cos\left(\frac{\pi}{2Q}\right),Cn,L\right)-\text{stable}
    \]
    for the p-\texttt{NPP} in the sense of Definition~\ref{def:stable-algs}.
\end{theorem}
We prove Theorem~\ref{thm:stable-hardness} in Section~\ref{sec:stable-hardness}. Several remarks are now in order.

First, we note that there is no assumption on the running time of $\A$: Theorem~\ref{thm:stable-hardness} rules out any $\A$ that is stable in the sense of Definition~\ref{def:stable-algs} with suitable parameters. 

We next highlight the scaling of parameters. An inspection of Theorem~\ref{thm:m-ogp-planted-npp} reveals that as $n\to\infty$, $m,\beta$ and $\eta$ remain constant in $n$: $m,\beta,\eta=O(1)$. Consequently, $Q,T=O(1)$ and therefore $p_\ell,p_{\rm st}=O(1)$. On the other hand, $p_f=e^{-\Theta(n)}$: Theorem~\ref{thm:stable-hardness} rules out algorithms that succeed with exponentially large probability. This is in stark contrast with prior work~\cite{gamarnik2021algorithmic,gamarnik2022algorithms} where $m$-OGP with constant $m$ yields hardness guarantees against stable algorithms succeeding with constant probability for certain unplanted models. This is an artifact of our proof technique. Due to presence of planting, one needs a certain `change of measure' argument, see below for details. An interesting open question is whether the $m$-OGP can rule out stable algorithms with constant probability of success for planted models. 

Next, the algorithms we rule out satisfy 
\[
d_H\bigl(\A(X),\A(Y)\bigr)\le Cn + L\|X-Y\|_2^2,
\]
where unconditionally $X,Y\sim \cN(0,I_n)$ with $\mathbb{E}[XY^T]=\cos\frac{\pi}{2Q} I_n$. As $C=O(1)$ and $L\|X-Y\|_2^2=\Theta(n)$, $\A$ is still allowed to make $\Theta(n)$ flips even when the inputs $X$ and $Y$ are close.

\subsection{Success Probability of Ruled Out Algorithms}
The proof of Theorem~\ref{thm:stable-hardness} is similar to those of~\cite[Theorem~3.2]{gamarnik2021algorithmic} and~\cite[Theorem~3.2]{gamarnik2022algorithms}; it is in particular based on an interpolation argument to leverage $m$-OGP and Ramsey Theory from extremal combinatorics. On the other hand, additional technical care is necessary due to planting, on which we now elaborate. A key step in the proof is to construct many interpolation paths $Y_i(\tau)$, $\tau \in \mathcal{I}$, as in Definition~\ref{def:overlap-set} for a suitable finite $\mathcal{I}$, and to argue that $\A$ is successful and stable along each path. For unplanted models, this is done via a crude union bound, where one can leverage the $m$-OGP with constant $m$ 
and subsequently rule out stable algorithms succeeding with constant probability. For planted models, however, such a union bound argument is no longer valid. In order to leverage the success/stability properties of $\A$ per Definition~\ref{def:stable-algs} along the interpolation paths, it appears necessary to rely on a certain change of measure argument, see in particular Lemmas~\ref{lemma:inter-oracle} and~\ref{lemma:success}. The change of measure argument brings additional terms for the stability guarantee, all of which are $O(1)$. For the success guarantee, however, the argument works at the expense of a $3^n$ factor,  see Lemma~\ref{lemma:success}. For this reason, our argument falls short of ruling out stable algorithms that succeed with constant probability for p-\texttt{NPP}; we leave this as an interesting open problem for future work.
\subsection{Anti-Concentration Condition} 
The algorithms that we rule out satisfy a certain anti-concentration property per Definition~\ref{def:stable-algs}: $\A(X)=\pm \bs^*$ with probability at most $1/(54T)$, which is $O(1)$ as $T=O(1)$. This is required for technical reasons: in order to leverage $m$-OGP, one has to establish that certain pairs $\bs_i,\bs_j\in\Sigma_n$ outputted by $\A$ on conditionally independent inputs are nearly orthogonal. This property is known as \emph{chaos}, see Lemma~\ref{lemma:chaos} for further details. If $\A$ outputs the planted partition $\pm \bs^*$ with overwhelming probability, it appears not possible to establish chaos since all pairwise overlaps (between outputs of $\A$) are trivially unity w.h.p. For this reason, such an anti-concentration property is indeed necessary for establishing algorithmic hardness using our techniques. On the other hand, we now argue that this is a benign requirement. Fix any $\epsilon\in(0,1)$, let $N_\epsilon$ be the number of partitions $\bs$ with $H(\bs,X)\le 2^{-\epsilon n}$. Remark~\ref{remark:exp-many} now yields that $N_\epsilon\sim e^{\Theta(n)}$ w.h.p. In particular, the number of `correct' outputs $\A$ can return is exponential. Hence, it is indeed plausible that any `natural' $\A$ returns any particular partition not only with at most a constant probability, but in fact with at most a certain vanishing probability. For these reasons, the anti-concentration assumption is indeed mild. An interesting open question is whether an $m$-OGP argument can rule out algorithms that return the planted solution itself w.h.p.
\section{Proofs}\label{sec:proofs}
\subsection{Notation}
 For any $n\in\mathbb{N}$, $[n]\triangleq \{1,2,\dots,n\}$. For any set $A$, $|A|$ denotes its cardinality. Given any proposition $E$, $\ind\{E\}$ denotes its indicator. For $n\in\mathbb{N}$, $\Sigma_n\triangleq\{-1,1\}^n$. For any $v=(v(i):1\le i\le n)\in\R^n$ and $p>0$, $\|v\|_p=(\sum_{i\le n}|v(i)|^p)^{1/p}$ and $\|v\|_\infty  = \max_{1\le i\le n}|v(i)|$. For $v,v'\in\R^n$, $\ip{v}{v'}\triangleq \sum_{1\le i\le n}v(i)v'(i)$. For any $\bs,\bs'\in\Sigma_n$, $d_H(\bs,\bs') = \sum_{1\le i\le n}\ind\{\bs(i)\ne \bs'(i)\}$. Ffor any $r>0$, $\log_r(\cdot)$ and $\exp_r(\cdot)$ denote, respectively, the logarithm and exponential functions base $r$; when $r=e$, we omit the subscript. For $p\in[0,1]$, $h_b(p) = -p\log_2 p -(1-p)\log_2(1-p)$. Denote by $I_k$ the $k\times k$ identity matrix, and by $\boldsymbol{1}$ the vector of all ones whose dimension will be clear from the context. Given $\boldsymbol{\mu}\in\R^k$ and $\Sigma\in\R^{k\times k}$, denote by $\cN(\boldsymbol{\mu},\Sigma)$ the multivariate normal distribution on $\R^k$ with mean $\boldsymbol{\mu}$ and covariance $\Sigma$. For $X\sim \cN(0,I_n)$ and a fixed $\bs^*\in\Sigma_n$, we denote the planted measure by $\mathbb{P}_{\rm pl}$ where $\mathbb{P}_{\rm pl}[\mathcal{E}] = \mathbb{P}_{X\sim \cN(0,I_n)}[\mathcal{E}\mid \left|\ip{\bs^*}{X}\right|\le \sqrt{n}3^{-n}]$ for any event $\mathcal{E}$. We denote by $\mathbb{E}_{\rm pl}$ the expectation under the planted measure. Given a matrix $\M$; $\|\M\|_F$, $\|\M\|_2$, $\sigma(\M)$, $\sigma_{\min}(\M)$, $\sigma_{\max}(\M)$, $|\M|$ and ${\rm trace}(\M)$ denote, respectively, its Frobenius norm, spectral norm, spectrum (that is, the set of its eigenvalues), smallest singular value, largest singular value, determinant, and trace. 
 
 A graph $\mathbb{G}=(V,E)$ is a collection of vertices $V$ together with some edges $(v,v')\in E$ between $v,v'\in V$. Throughout, we consider simple graphs only, that is undirected graphs with no loops. A graph $\mathbb{G}=(V,E)$ is called a \emph{clique} if for every distinct $v,v'\in V$, $(v,v')\in E$. We denote by $K_m$ the clique on $m$ vertices. Given $\mathbb{G}=(V,E)$, a subset $S\subset V$ of its vertices is called an \emph{independent set} if for every distinct $v,v'\in V$, $(v,v')\notin E$. The largest cardinality of such a set is called the \emph{independence number} of $\mathbb{G}$, denoted $\alpha(\mathbb{G})$. A $q-$coloring of a graph $\mathbb{G}=(V,E)$ is a function $\varphi:E\to\{1,2,\dots,q\}$ assigning to each $e\in E$ one of $q$ available colors. 

We employ the standard Bachmann-Landau asymptotic notation throughout, e.g.\,$\Theta(\cdot),O(\cdot),o(\cdot)$, and $\Omega(\cdot)$, where the underlying asymptotics will often be clear from the context.  In certain cases
where a confusion is possible, we reflect the underlying asymptotics as a subscript.  All
floor/ceiling operators are omitted for the sake of simplicity.
\subsection{Auxiliary Results}
In this section, we collect several auxiliary results. We begin with several standard estimates on binomial coefficients.
\begin{lemma}\label{lemma:binom-entropy}
Denote by $h_b(\cdot)$ the binary entropy function logarithm base 2. For any $1\le k\le n-1$, 
    \[
    \sqrt{\frac{n}{8k(n-k)}}2^{nh_b(k/n)}\le \binom{n}{k}\le \sqrt{\frac{n}{\pi k(n-k)}}2^{nh_b(k/n)}.
    \]
    In particular, if $k=\rho n$ for $\rho=O(1)$, then
    \[
    \binom{n}{\rho n} =\exp_2\bigl(nh_b(\rho)+O(\log n)\bigr).
    \]
\end{lemma} \noindent See~\cite[Section~17.5]{coverthomasbook} for a proof of Lemma~\ref{lemma:binom-entropy}.
\begin{lemma}\label{lemma:sum-of-bin}
    Fix $\alpha\le 1/2$. For all $n$,
    \[
    \sum_{i\le \alpha n}\binom{n}{i}\le 2^{nh_b(\alpha)}.
    \]
\end{lemma}
\noindent See~\cite[Theorem~3.1]{galvin2014three} for the proof of Lemma~\ref{lemma:sum-of-bin}.

The next bound controls the binomial coefficient $\binom{n}{d}$ when $d=o(n)$. 
\begin{lemma}\label{lemma:binom-sublinear}
    Let $n,d\in\mathbb{N}$ with $d=o(n)$. Then, 
    \[
    \log_2 \binom{n}{d} = \bigl(1+o_n(1)\bigr)d\log_2 \frac{n}{d}.
    \]
\end{lemma}
Lemma~\ref{lemma:binom-sublinear} is folklore, see e.g.~\cite[Lemma~6.1]{gamarnik2021algorithmic} for a proof.

Our next auxiliary result lower bounds a certain determinant that occurs frequently in our calculations.
\begin{lemma}\label{lemma:det-bd}
    Let $\bs_1,\bs_2,\bs_3\in\Sigma_n$ with $\bs_i\ne \pm \bs_j$ for $1\le i<j\le 3$. Denote by $\Sigma\in \R^{3\times 3}$ the matrix with entries $\Sigma_{ij} = n^{-1}\ip{\bs_i}{\bs_j}$ for $1\le i,j\le 3$. Then,
    \[
    |\Sigma|\ge n^{-3}.
    \]
\end{lemma}
\begin{proof}[Proof of Lemma~\ref{lemma:det-bd}]
    Let $\Xi \in \{-1,1\}^{3\times n}$ with rows $\bs_1,\bs_2,\bs_3$. Then, $\Sigma = \frac1n\Xi \Xi^T$.  We first show ${\rm rank}(\Xi)=3$. To that end, observe that $\bs_1$ and $\bs_2$ are linearly independent. Indeed, otherwise $\bs_2 = \lambda \bs_1$ for some $\lambda\in\mathbb{R}$. But then $\lambda\in\{-1,1\}$ since $\bs_2(1)=\lambda \bs_1(1)$ and $\bs_1(1),\bs_2(1)\in\{-1,1\}$. Since $\bs_2\ne \pm \bs_1$, this is a contradiction. 

    Next, we show $\bs_1,\bs_2,\bs_3$ are linearly independent, which will then yield ${\rm rank}(\Xi)=3$. Assume the contrary. Then there exists $\lambda_1,\lambda_2\in\R$ such that
    \[
    \bs_3 = \lambda_1 \bs_1 + \lambda_2 \bs_2.
    \]
    As $\bs_2\ne \pm \bs_1$, it is the case that there exists $1\le i,j\le n$ such that $\bs_1(i)=\bs_2(i)$ and $\bs_1(j)=-\bs_2(j)$. So, 
    \begin{align*}
        \bs_3(i) = \lambda_1 \bs_1(i) + \lambda_2 \bs_2(i) &\implies \lambda_1+\lambda_2\in\{-1,1\}. \\
        \bs_3(j) = \lambda_1 \bs_1(j) + \lambda_2 \bs_2(j) &\implies \lambda_1-\lambda_2\in\{-1,1\}.
    \end{align*}
    Hence, either $\lambda_1-\lambda_2 = \lambda_1+\lambda_2$ or $\lambda_1-\lambda_2 = -(\lambda_1+\lambda_2)$. The former case implies $\lambda_2=0$, so $\bs_3 = \pm \bs_1$, a contradiction. Likewise, if $\lambda_1-\lambda_2 = -\lambda_1-\lambda_2$, then $\lambda_1=0$, yielding $\bs_3=\pm \bs_2$, which is again a contradiction. Hence, $\bs_1,\bs_2,\bs_3$ are linearly independent, yielding ${\rm rank}(\Xi)=3$. In particular, ${\rm rank}(\Sigma)=3$. Now that $n\Sigma\in\Z^{3\times 3}$ is positive semidefinite, full rank, and has integer-valued entries, we obtain $|n\Sigma|\ge 1$. Hence, $|\Sigma|\ge n^{-3}$, as claimed.
\end{proof}
We next record several Gaussian inequalities. The first is the celebrated \emph{Gaussian Correlation Inequality} due to Royen~\cite{royen2014simple}.
\begin{theorem}\label{thm:royen}
    Let $K,L$ be symmetric\footnote{A set $K$ is symmetric if $x\in K\iff -x\in K$.} convex sets in $\R^d$ and $\mu$ be any centered Gaussian measure on $\R^d$. Then,
    \[
    \mu(K\cap L)\ge \mu(K)\mu(L).
    \]
    In particular, if $n=n_1+n_2$, $X$ is an $n$-dimensional multivariate Gaussian random vector, and $t_1,\dots,t_n\in\R^+$ are arbitrary then
    \begin{equation}
        \mathbb{P}[|X_1|<t_1,\dots,|X_n|\le t_n]\ge \mathbb{P}[|X_i|<t_i,1\le i\le n_1]\mathbb{P}[|X_{j}|<t_{j},n_1+1\le j\le n].\label{eq:Gaussian-corr}
        \end{equation}
\end{theorem}
Theorem~\ref{thm:royen} is established in a breakthrough paper by Royen~\cite{royen2014simple}. The result recorded in~\eqref{eq:Gaussian-corr} is equivalent to the original version, but rather easier to work with, see~\cite{latala2017royen} for details.

The next result pertains certain Gaussian probabilities, partly reproduced from~\cite{turner2020balancing}. 
\begin{lemma}\label{lemma:gaussproblemma}
Let $Z\sim \cN(0,1)$, $\rho\in[0,1)$, and $Z_\rho \sim \cN(0,1)$ with $\mathbb{E}[ZZ_\rho]=\rho$. 
\begin{itemize}
    \item[(a)] There exists an absolute constant $c>0$ such that
    \[
   \sqrt{\frac{2}{\pi}}z\ge  \mathbb{P}[|Z|\le z]\ge \sqrt{\frac{2}{\pi}}z-cz^3
    \]
    for all $0<z<1$. In particular, if $z=o_n(1)$ then we obtain
    \[
    \mathbb{P}[|Z|\le z] = \sqrt{\frac{2}{\pi}}z\bigl(1+o_n(1)\bigr).
    \]
    \item[(b)] Let $z_1,z_2>0$ with $z_1,z_2=o_n(1)$. Then, 
    \[
   \frac{2z_1z_2}{\pi \sqrt{1-\rho^2}}\left(1-\frac{z_1^2+z_2^2}{\sqrt{1-\rho^2}}\right) \le \mathbb{P}[|Z|\le z_1,|Z_\rho|\le z_2] \le \frac{2z_1z_2}{\pi\sqrt{1-\rho^2}}.
    \]
    In particular, if $(z_1^2+z_2^2)/\sqrt{1-\rho^2}=o_n(1)$ then
    \[
    \mathbb{P}[|Z|\le z_1,|Z_\rho|\le z_2]  = \frac{2z_1z_2}{\pi\sqrt{1-\rho^2}}\bigl(1+o_n(1)\bigr).
    \]
\end{itemize}
\begin{proof}[Proof of Lemma~\ref{lemma:gaussproblemma}]
    Part ${\rm (a)}$ is reproduced from~\cite[Lemma~13]{turner2020balancing}, see the proof therein. We now establish part ${\rm (b)}$. Observe that for any $\rho\in[-1,1]$ and $x,y$, 
    \[
    x^2-2\rho xy+y^2 \ge 2|xy|-2\rho xy \ge 0.
    \]
    Hence, if
    \[
    f_{\rho}(x,y) = \exp\left(-\frac{x^2-2\rho xy+y^2}{2(1-\rho^2)}\right)
    \]
    then $f_\rho(x,y)\le 1$ for any $\rho\in(-1,1)$ and $x,y\in\R$, and that
    \[
    \mathbb{P}[|Z|\le z_1,|Z_\rho|\le z_2] = \frac{1}{2\pi\sqrt{1-\rho^2}}\int_{\substack{-z_1\le x\le z_1\\-z_2\le x\le z_2}}f_{\rho}(x,y)\;dx\;dy \le \frac{2z_1z_2}{\pi\sqrt{1-\rho^2}}.
    \]
    For the lower bound, we use the well-known inequality $e^{-x}\ge 1-x$ to get
    \[
    f_{\rho}(x,y)\ge 1-\frac{x^2-2\rho xy+y^2}{2(1-\rho^2)}.
    \]
    Clearly, 
    \[
    \sup_{|x|\le z_1,|y|\le z_2}|x^2-2\rho xy+y^2|\le (1+\rho)(x^2+y^2)\le 2(z_1^2+z_2^2)
    \]
    by triangle inequality and the simple inequality $2|xy|\le x^2+y^2$. So, 
    \[
   \mathbb{P}[|Z|\le z_1,|Z_\rho|\le z_2] \ge \frac{2z_1z_2}{\pi \sqrt{1-\rho^2}}\left(1-\frac{z_1^2+z_2^2}{\sqrt{1-\rho^2}}\right).\qedhere
   \]
\end{proof}
    
\end{lemma}
The next auxiliary result is a simple Gaussian anti-concentration lemma.
\begin{lemma}\label{lemma:anti-con}
    Let $\boldsymbol{Z}=(Z_1,\dots,Z_m)\in\R^m$ be a multivariate normal random vector with covariance matrix $\Sigma$ and $V\subset \R^m$. Then,
    \[
    \mathbb{P}[\boldsymbol{Z}\in V]\le (2\pi)^{-\frac{m}{2}}|\Sigma|^{-\frac12}{\rm Vol}(V),
    \]
    where ${\rm Vol}(V)$ is the Lebesgue measure of $V$.
\end{lemma}
\begin{proof}[Proof of Lemma~\ref{lemma:anti-con}]
    We have
    \[
\mathbb{P}[\boldsymbol{Z}\in V] = (2\pi)^{-\frac{m}{2}}|\Sigma|^{-\frac12}\int_{\boldsymbol{z}\in V}\exp\left(-\frac{\boldsymbol{z}^T\Sigma^{-1}\boldsymbol{z}}{2}\right)\;d\boldsymbol{z} \le (2\pi)^{-\frac{m}{2}}|\Sigma|^{-\frac12}{\rm Vol}(V),
    \]
    as $\exp(-\frac{\boldsymbol{z}^T\Sigma^{-1}\boldsymbol{z}}{2})\le 1$ holds trivially for any positive semi-definite $\Sigma$.
\end{proof}

The next auxiliary result will be useful to bound the change of the spectrum of a matrix in presence of an additive perturbation.
\begin{theorem}\label{thm:wielandt-hoffman}{\bf (Hoffman-Wielandt)}
    Let $n\in\mathbb{N}$ and $A$ and $A+E$ be $n\times n$ symmetric matrices with eigenvalues $\lambda_1\ge \cdots \ge \lambda_n$ and $\mu_1\ge \cdots \ge \mu_n$, respectively. Then, 
    \[
    \sum_{1\le i\le n}(\lambda_i-\mu_i)^2 \le \|E\|_F^2.
    \]
\end{theorem}
Theorem~\ref{thm:wielandt-hoffman} is due to Hoffman and Wielandt. See~\cite[Corollary~6.3.8]{horn2012matrix} for a proof and~\cite{hoffman1953variation} for the original paper.

The last auxiliary results regards trigonometric functions.
\begin{lemma}\label{lemma:trigo}
    \begin{itemize}
        \item[(a)] $\sin(x)>\frac{x}{2}$ for any $x\in[0,\pi/2]$. 
        \item[(b)] For any $x,y$, 
        \[
        |\sin(x)-\sin(y)|\le |x-y|\quad\text{and}\quad |\cos(x)-\cos(y)|\le |x-y|.
        \]
    \end{itemize}
\end{lemma}
\begin{proof}[Proof of Lemma~\ref{lemma:trigo}]
    For part ${\rm (a)}$, it suffices to observe that $\sin(x)$ is concave on $[0,\pi/2]$, so the line $y=2x/\pi$ always remains below the graph, hence $\sin(x)\ge 2x/\pi>x/2$.

    For part ${\rm (b)}$, we have by using Mean Value Theorem,
    \[
    |\sin(x)-\sin(y)| = |x-y|\cdot |\cos(c)|
    \]
    for some $\min\{x,y\}<c<\max\{x,y\}$. As $|\cos(c)|\le 1$ trivially, the conclusion follows. The same argument applies also to $\cos$, concluding the proof of Lemma~\ref{lemma:trigo}.
\end{proof}
\subsection{Proof of Theorem~\ref{thm:ground-state}}\label{sec:pf-gs}
We begin by reminding the reader that
\[
H(\bs) = \frac{1}{\sqrt{n}}\left|\ip{\bs}{X}\right|.
\]
We establish that both the lower bound and the upper bounds hold individually w.h.p. The result then follows via a union bound.
\subsubsection*{Proof of Lower Bound}
Fix an arbitrary $f = o_n(1)$. We first establish that w.h.p.\,under the planted measure,
\[
\min_{\bs \in\Sigma_n\setminus\{\pm \bs^*\}} H(\bs)\ge f(n)2^{-n}.
\]
To that end, we establish the following proposition.
\begin{proposition}\label{prop:first-moment}
    Let $f:\N\to\R^+$ with $f(n)=o_n(1)$ and
    \[
    S_f=\bigl\{\bs\in\Sigma_n:\bs\ne\pm\bs^*,H(\bs)\le f(n)2^{-n}\bigr\}.
    \]
    Then, $\mathbb{E}_{\rm pl}|S_f|= o(1)$ as $n\to\infty$.
\end{proposition}
\begin{proof}[Proof of Proposition~\ref{prop:first-moment}]
    Note that for $\bs\ne \pm\bs^*$, $1\le d_H(\bs,\bs^*)\le n-1$ and if $k=d_H(\bs,\bs^*)$, then
    \begin{equation}\label{eq:overlap}
        \frac1n \ip{\bs}{\bs^*} = \frac{n-2d_H(\bs,\bs^*)}{n} = 1-\frac{2k}{n}.
    \end{equation}
    Furthermore, 
    \begin{align}
        \mathbb{P}\bigl[H(\bs)\le f(n)2^{-n}\lvert H(\bs^*)\le 3^{-n}\bigr]&=\frac{\mathbb{P}\bigl[H(\bs)\le f(n)2^{-n},H(\bs^*)\le 3^{-n}\bigr]}{\mathbb{P}\bigl[H(\bs^*)\le 3^{-n}\bigr]} \label{eq:bayes} \\
        &\le\frac{1}{\sqrt{2\pi\frac{k}{n}\left(1-\frac{k}{n}\right)}}f(n)2^{-n}(1+o_n(1))\label{eq:prob-first-mom}.
    \end{align}
    We next justify the steps above.~\eqref{eq:bayes} follows from removing the conditioning. On the other hand,
  after removing the conditioning, $\bigl(H(\bs),H(\bs^*)\bigr)$ forms a bivariate normal with parameters
    \[
    \begin{pmatrix}
        H(\bs) \\ H(\bs^*)
    \end{pmatrix} \sim \cN\left(\boldsymbol{0},\begin{pmatrix}1 & 1-\frac{2k}{n} \\1-\frac{2k}{n} &1\end{pmatrix}\right).
    \]
   Hence,
    \begin{equation}\label{eq:auxil2}
        \mathbb{P}\bigl[H(\bs)\le f(n)2^{-n},H(\bs^*)\le 3^{-n}\bigr] \le \frac{1}{2\pi \sqrt{1-\left(1-\frac{2k}{n}\right)^2}}(2\cdot f(n)2^{-n})(2\cdot 3^{-n}).
    \end{equation}
    Further, using Lemma~\ref{lemma:gaussproblemma}
    \begin{equation}\label{eq:auxil3}
        \mathbb{P}\bigl[H(\bs^*)\le 3^{-n}\bigr] = \frac{1}{\sqrt{2\pi}}(2\cdot 3^{-n})(1+o_n(1)),
        \end{equation}
        where the $o_n(1)$ term is clearly independent of $k=d_H(\bs,\bs^*)$.
        Combining~\eqref{eq:bayes},~\eqref{eq:auxil2} and~\eqref{eq:auxil3} yields~\eqref{eq:prob-first-mom}. We now estimate $\mathbb{E}_{\rm pl}|S_f|$. Fix an $\epsilon\in(0,1/2)$ and let
        \[
        I(\epsilon) = \left[1,n\left(\frac12-\epsilon\right)\right] \cup \left[n\left(\frac12+\epsilon\right),n-1\right],\quad\text{and}\quad I'(\epsilon) = \left[n\left(\frac12-\epsilon\right),n\left(\frac12+\epsilon\right)\right]
        \]
        \begin{align*}
            \mathbb{E}_{\rm pl}|S_f| &=\sum_{1\le k\le n-1}\binom{n}{k}  \mathbb{P}\bigl[H(\bs)\le f(n)2^{-n}\lvert H(\bs^*)\le 3^{-n}\bigr] \\
            &=(1+o_n(1))\sum_{1\le k\le n-1}\binom{n}{k} \frac{1}{\sqrt{2\pi \frac{k}{n}\left(1-\frac{k}{n}\right)}}f(n)2^{-n}\\
            &\le (1+o_n(1))\left(\sum_{k\in I(\epsilon) \cap \Z}\binom{n}{k} \frac{1}{\sqrt{2\pi \frac{k}{n}\left(1-\frac{k}{n}\right)}}f(n)2^{-n} + \sum_{k\in I'(\epsilon) \cap \Z}\binom{n}{k} \frac{1}{\sqrt{2\pi \frac{k}{n}\left(1-\frac{k}{n}\right)}}f(n)2^{-n}\right).
        \end{align*}
        Note that $k(n-k)\ge n-1$ for $1\le k\le n-1$ and $n\ge 2$. Hence, hence $\sqrt{(k/n)(1-k/n)}\ge 1/\sqrt{2n}$ for $n,k$ in the same range. So,
        \begin{align*}
        \sum_{k\in I(\epsilon) \cap \Z}\binom{n}{k} \frac{1}{\sqrt{2\pi \frac{k}{n}\left(1-\frac{k}{n}\right)}}f(n)2^{-n}&\le f(n)2^{-n}\sqrt{\frac{n}{\pi}}\sum_{k\in I(\epsilon) \cap \Z}\binom{n}{k} \\
        &\le O\left(f(n)\sqrt{n}2^{-n\left(1-h(1/2-\epsilon)\right)}\right)\\
        &=2^{-\Omega(n)},
            \end{align*}
            where we applied Lemma~\ref{lemma:binom-entropy} and used the fact $1-h(1/2-\epsilon)>0$. On the other hand if $k\in I'(\epsilon)$, then 
            \[
            \frac{k}{n}\left(1-\frac{k}{n}\right)\ge c_\epsilon = \frac14 -\epsilon^2>0.
            \]
            So, 
            \begin{align*}
            \sum_{k\in I'(\epsilon)\cap \mathbb{Z}}\binom{n}{k} \frac{1}{\sqrt{2\pi \frac{k}{n}\left(1-\frac{k}{n}\right)}}f(n)2^{-n}&\le \frac{f(n)2^{-n}}{\sqrt{2\pi c_\epsilon}} \sum_{k\in I'(\epsilon)\cap\mathbb{Z}}\binom{n}{k} = O\bigl(f(n)\bigr) = o_n(1),
            \end{align*}
            where we used the fact
            \[
            \sum_{k\in I'(\epsilon)\cap \mathbb{Z}}\binom{n}{k}\le 2^n.
            \]                 
    Combining these, we immediately establish
                        \[
                        \mathbb{E}_{\rm pl}|S_f|\le 2^{-\Omega(n)} + O(f(n)) = o_n(1).
                        \]
This completes the proof of Proposition~\ref{prop:first-moment}.
\end{proof}
Note that using Proposition~\ref{prop:first-moment}, we have by Markov's inequality
\[
\mathbb{P}_{\rm pl}[|S_f|\ge 1]\le \mathbb{E}_{\rm pl}|S_f| = o_n(1).
\]
So,  w.h.p.
\[
\min_{\bs \in\Sigma_n\setminus\{\pm \bs^*\}} H(\bs)\ge f(n)2^{-n},
\]
establishing the lower bound in Theorem~\ref{thm:ground-state}.
\subsubsection*{Proof of Upper Bound}
\begin{proposition}\label{prop:2nd-mom}
    Fix an $\epsilon>0$, a function $g=\omega_n(1)$ and define
    \[
    S_g(\epsilon) = \bigl\{\bs\in\Sigma_n:H(\bs)\le g2^{-n},n(1-\epsilon)/2\le d_H\bigl(\bs,\bs^*\bigr)\le n(1+\epsilon)/2\bigr\}.
    \]
    Then, the following holds.
    \begin{itemize}
        \item[(a)] $\mathbb{E}_{\rm pl}|S_g(\epsilon)|\ge \frac{2g}{\sqrt{2\pi}}(1+o_n(1))$.
        \item[(b)] $\displaystyle \frac{(\mathbb{E}_{\rm pl}|S_g(\epsilon)|)^2}{\mathbb{E}_{\rm pl}[|S_g(\epsilon)|^2]}\ge \frac{1}{o_n(1)+(1+o_n(1))(1-\epsilon\sqrt{6})^{-3/2}}$.
    \end{itemize}
\end{proposition}
We first show how Proposition~\ref{prop:2nd-mom} yields the upper bound in Theorem~\ref{thm:ground-state}. To that end, we use the \emph{second moment method}: for any non-negative integer valued random variable $N$,
\begin{equation}\label{eq:2nd-mom-met}
    \mathbb{P}[N\ge 1]\ge \frac{\mathbb{E}[N]^2}{\mathbb{E}[N^2]}.
    \end{equation}
    The proof of~\eqref{eq:2nd-mom-met} is based on a simple application of Cauchy-Schwarz inequality, which we provide for completeness. Let $I=\ind\{N\ge 1\}$ and observe that $\mathbb{E}[N]=\mathbb{E}[NI]$, since $N\ind\{N\le 0\}=0$ almost surely as $N$ takes non-negative integer values. So, Cauchy-Schwarz inequality implies \[
    \mathbb{P}[N\ge 1]\mathbb{E}[N^2]=\mathbb{E}[I^2]\mathbb{E}[N^2]\ge \mathbb{E}[N\ind\{N\ge 1\}]^2=\mathbb{E}[N]^2.
    \]
   Next, assume Proposition~\ref{prop:2nd-mom} holds and define $S_g = \{\bs\in\Sigma_n\setminus \{\pm \sigma^*\}:H(\bs)\le g2^{-n}\}$. Note that for any $\epsilon>0$, $S_g(\epsilon)\subseteq S_g$. So, 
    \[
    \mathbb{P}_{\rm pl}[|S_g|\ge 1]\ge \mathbb{P}_{\rm pl}[|S_g(\epsilon)\ge 1]\ge \frac{(\mathbb{E}_{\rm pl}|S_g(\epsilon)|)^2}{\mathbb{E}_{\rm pl}[|S_g(\epsilon)|^2]}\ge \frac{1}{o_n(1)+(1+o_n(1))(1-\epsilon\sqrt{6})^{-3/2}},
    \]
    where we applied~\eqref{eq:2nd-mom-met} and part ${\rm (b)}$ of Proposition~\ref{prop:2nd-mom}. We now fix an $\epsilon>0$ and pass to the limit as $n\to\infty$ to obtain    \begin{equation}\label{eq:liminf}
    \liminf_{n\to\infty}\mathbb{P}_{\rm pl}[|S_g|\ge 1]\ge (1-\epsilon\sqrt{6})^{3/2}.
        \end{equation}
    Note that~\eqref{eq:liminf} holds for all $\epsilon>0$ and that the left hand side does not depend on $\epsilon$. So, sending $\epsilon\to 0$, we get 
      \begin{equation*}
    \liminf_{n\to\infty}\mathbb{P}_{\rm pl}[|S_g|\ge 1]\ge 1.
        \end{equation*}
        As we trivially have $\limsup_{n\to\infty} \mathbb{P}_{\rm pl}[|S_g|\ge 1]\le 1$, we thus conclude
        \[
        \lim_{n\to\infty}\mathbb{P}_{\rm pl}[|S_g|\ge 1]=1,
        \]
        which precisely implies that $\min_{\bs \in \Sigma_n,\bs\ne \pm \sigma^*}H(\bs)\le g2^{-n}$ w.h.p. Therefore, it suffices to establish Proposition~\ref{prop:2nd-mom}.
\begin{proof}[Proof of Proposition~\ref{prop:2nd-mom}]
Set 
\begin{equation}\label{eq:J-of-eps}
    J(\epsilon) = \left[\frac{n(1-\epsilon)}{2},\frac{n(1+\epsilon)}{2}\right] \cap \mathbb{Z} \quad\text{and}\quad J'(\epsilon) = \left([1,n-1]\setminus \left[\frac{n(1-\epsilon)}{2},\frac{n(1+\epsilon)}{2}\right]\right)\cap \mathbb{Z}.
\end{equation}
   \paragraph{Part {\rm (a)}.} Note that for any $k\in\mathbb{N}$, there exists $\binom{n}{k}$ choices for $\bs$ such that $d_H(\bs,\bs^*)=k$. So,
    \begin{equation}\label{eq:first-mom}            \mathbb{E}_{\rm pl}|S_g(\epsilon)| = \sum_{k\in J(\epsilon)}\binom{n}{k} \mathbb{P}_k [H(\bs)\le g(n)2^{-n}\mid H(\bs^*)\le 3^{-n}],
    \end{equation}
    where $\mathbb{P}_k$ highlights the dependence of probability term on $k=d_H(\bs,\bs^*)$. Now,
    \begin{align}
        \mathbb{P}_k [H(\bs)\le g(n)2^{-n}\mid H(\bs^*)\le 3^{-n}]&=\frac{ \mathbb{P}_k [H(\bs)\le g(n)2^{-n},H(\bs^*)\le 3^{-n}]}{ \mathbb{P} [H(\bs^*)\le 3^{-n}]} \nonumber \\
        &\ge \mathbb{P}[H(\bs)\le g2^{-n}]\label{eq:gaussian-comparison} \\
        &=\frac{2g2^{-n}}{\sqrt{2\pi}}(1+o_n(1)) \label{eq:lemma-pax}.
    \end{align}
    Above,~\eqref{eq:gaussian-comparison} uses Theorem~\ref{thm:royen}, and~\eqref{eq:lemma-pax} uses Lemma~\ref{lemma:gaussproblemma} with $z=g2^{-n}$. Inserting~\eqref{eq:lemma-pax} into~\eqref{eq:first-mom}, we thus have
    \[
    \mathbb{E}_{\rm pl}|S_g(\epsilon)|\ge \sum_{k\in J(\epsilon)} \binom{n}{k}\frac{2g 2^{-n}}{\sqrt{2\pi}}(1+o_n(1)) \ge \frac{2g}{\sqrt{2\pi}}(1+o_n(1)),
    \]
    where we appealed to Lemma~\ref{lemma:sum-of-bin} to obtain
    \[
    \sum_{k\in J(\epsilon)} \binom{n}{k} \ge 2^n -2 - \sum_{k\in J'(\epsilon)}\binom{n}{k} \ge 2^n - 2\cdot 2^{nh_b(1/2-\epsilon/2)} = 2^n(1+o_n(1)).
    \]
\paragraph{Part {\rm (b)}.} We begin by writing
\begin{align*}
    |S_g(\epsilon)| = \sum_{\substack{\bs_1\in \Sigma_n \\ d_H(\bs_1,\bs^*)\in J(\epsilon)}}\ind \bigl\{H(\bs_1)\le g2^{-n}\bigr\} 
   \Rightarrow |S_g(\epsilon)|^2 = \sum_{\substack{\bs_1,\bs_2\in \Sigma_n \\ d_H(\bs_1,\bs^*)\in J(\epsilon) \\ d_H(\bs_2,\bs^*)\in J(\epsilon)}}\ind \bigl\{H(\bs_1)\le g2^{-n},H(\bs_2)\le g2^{-n}\bigr\}
   \end{align*}.
   Hence, we have
\begin{align}
 \mathbb{E}_{\rm pl}[|S_g(\epsilon)|^2]&=\sum_{\substack{\bs_1,\bs_2\in \Sigma_n \\ d_H(\bs_1,\bs^*)\in J(\epsilon) \\ d_H(\bs_2,\bs^*)\in J(\epsilon)}}\mathbb{P}_{\rm pl}\bigl[H(\bs_1)\le g2^{-n},H(\bs_2)\le g2^{-n}\bigr] \nonumber\\
 & = \sum_{\substack{\bs_1,\bs_2\in\Sigma_n\\d_H(\bs_1,\bs^*) \in J(\epsilon)\\ d_H(\bs_2,\bs^*)\in J(\epsilon)}} \mathbb{P}\bigl[H(\bs_1)\le g2^{-n},H(\bs_2)\le g2^{-n}\mid H(\bs^*)\le 3^{-n}\bigr].\label{eq:2nd-mom-auxil}
\end{align}
The sum appearing in~\eqref{eq:2nd-mom-auxil} runs over pairs $(\bs_1,\bs_2)\in\Sigma_n\times\Sigma_n$ with $d_H(\bs_1,\bs^*),d_H(\bs_2,\bs^*)\in J(\epsilon)$. We now split these pairs into three classes based on $d_H(\bs_1,\bs_2)$.
\begin{itemize}
    \item Let \[
    \mathcal{T}_1 = \bigl\{(\bs_1,\bs_2): \bs_2=\pm \bs_1,d_H(\bs_1,\bs^*)\in J(\epsilon),d_H(\bs_2,\bs^*)\in J(\epsilon)\bigr\}.
    \]
    Then,
    \[
    |\mathcal{T}_1| = 2\sum_{k\in J(\epsilon)}\binom{n}{k}.
    \]
    \item  Let
    \[
    \mathcal{T}_2 = \bigl\{(\bs_1,\bs_2): d_H(\bs_1,\bs_2)\in J'(\epsilon),d_H(\bs_1,\bs^*)\in J(\epsilon),d_H(\bs_2,\bs^*)\in J(\epsilon)\bigr\}.
    \]
    Note that having fixed a $\bs_1$, there exists $\sum_{k\in J'(\epsilon)}\binom{n}{k}$ choices for $\bs_2$ such that $(\bs_1,\bs_2)\in\mathcal{T}_2$. As
    \[
    \sum_{k\in J'(\epsilon)}\binom{n}{k}\le 2\exp_2\left(nh_b\left(\frac{1-\epsilon}{2}\right)\right)    \] 
    per Lemma~\ref{lemma:sum-of-bin}, we conclude
    \[
    |\mathcal{T}_2| = \exp_2\left(nh_b\left(\frac{1-\epsilon}{2}\right)+1\right)\underbrace{\sum_{k\in J(\epsilon)}\binom{n}{k}}_{\le 2^n}\le \exp_2\left(n+nh_b\left(\frac{1-\epsilon}{2}\right)+1\right).
    \]
    \item Let
    \[
    \mathcal{T}_3=\bigl\{(\bs_1,\bs_2):d_H(\bs_1,\bs_2)\in J(\epsilon),d_H(\bs_1,\bs^*)\in J(\epsilon),d_H(\bs_2,\bs^*)\in J(\epsilon)\bigr\}.
    \]
    Then we have
    \[
    |\mathcal{T}_3|\le \left(\sum_{k\in J'(\epsilon)}\binom{n}{k}\right)^2.
    \]
\end{itemize}
Our analysis below reveals that the dominant contribution to the second moment is due to $\mathcal{T}_3$: the contributions from $\mathcal{T}_1,\mathcal{T}_2$ are negligible. We now make this more precise. We have
\begin{align}
     \mathbb{E}_{\rm pl}[|S_g(\epsilon)^2] &= \underbrace{\sum_{(\bs_1,\bs_2)\in\mathcal{T}_1}\mathbb{P}\bigl[H(\bs_1)\le g2^{-n},H(\bs_2)\le g2^{-n}\mid H(\bs^*)\le 3^{-n}\bigr]}_{\Sigma_1} \nonumber \\
     &+ \underbrace{\sum_{(\bs_1,\bs_2)\in\mathcal{T}_2}\mathbb{P}\bigl[H(\bs_1)\le g2^{-n},H(\bs_2)\le g2^{-n}\mid H(\bs^*)\le 3^{-n}\bigr]}_{\Sigma_2}\nonumber \\
     &+\underbrace{\sum_{(\bs_1,\bs_2)\in\mathcal{T}_3}\mathbb{P}\bigl[H(\bs_1)\le g2^{-n},H(\bs_2)\le g2^{-n}\mid H(\bs^*)\le 3^{-n}\bigr]}_{\Sigma_3}.\label{eq:2nd-mom-auxil2}
\end{align}
\paragraph{Analysis of $\Sigma_1$ and $\Sigma_2$} Note that 
\begin{equation*}
    \Sigma_1 = 2\sum_{\bs_1:d_H(\bs_1,\bs^*)\in J(\epsilon)}\mathbb{P}\bigr[H(\bs_1)\le g2^{-n}\mid H(\bs^*)\le 3^{-n}\bigl]=2\mathbb{E}_{\rm pl}|S_g(\epsilon)|.
\end{equation*}
In particular, using Part {\rm (a)}, we get
\begin{equation}\label{eq:Sigma-1}
  \frac{\Sigma_1}{(\mathbb{E}_{\rm pl}|S_g(\epsilon)|)^2}   = o_n(1),
\end{equation}
as $g(n)=\omega_n(1)$. Next, we consider $\Sigma_2$. 
\begin{lemma}\label{lemma:prob-middle-term}
   \[\max_{(\bs_1,\bs_2)\in\mathcal{T}_2} \mathbb{P}\bigl[H(\bs_1)\le g2^{-n},H(\bs_2)\le g2^{-n}\mid H(\bs^*)\le 3^{-n}\bigr] \le (2\pi)^{-1}\cdot n^{\frac32}\cdot \left(2\cdot g2^{-n}\right)^2.
   \]
\end{lemma}
\begin{proof}[Proof of Lemma~\ref{lemma:prob-middle-term}]
Fix any $(\bs_1,\bs_2)\in\mathcal{T}_2$.
  We have
\begin{equation}\label{eq:bayes-auxil}
    \mathbb{P}\bigl[H(\bs_1)\le g2^{-n},H(\bs_2)\le g2^{-n}\mid H(\bs^*)\le 3^{-n}\bigr]=\frac{\mathbb{P}\bigl[H(\bs_1)\le g2^{-n},H(\bs_2)\le g2^{-n},H(\bs^*)\le 3^{-n}\bigr]}{\mathbb{P}\bigl[H(\bs^*)\le 3^{-n}\bigr]}. 
\end{equation}
Further, Lemma~\ref{lemma:gaussproblemma} yields $\mathbb{P}[H(\bs^*)\le 3^{-n}] = \frac{2}{\sqrt{2\pi}}3^{-n}(1+o_n(1))$. For the numerator, let $T\sim \cN(0,I_n)$, set $Z_i=n^{-\frac12}\ip{\bs_i}{T}$, $1\le i\le 3$, where $\bs_3=\bs^*$ for convenience. Note that, 
\begin{align}
   & \mathbb{P}\bigl[H(\bs_1)\le g2^{-n},H(\bs_2)\le g2^{-n},H(\bs^*)\le 3^{-n}\bigr] 
 \nonumber\\&= \mathbb{P}\bigl[(T_1,T_2,T_3)\in [-g2^{-n},g2^{-n}]\times [-g2^{-n},g2^{-n}]\times [-3^{-n},3^{-n}]\bigr].\label{eq:auxil-pb}
\end{align}
Further, $Z_1,Z_2,Z_{*}\sim \cN(0,1)$ form a multivariate normal random vector with covariance matrix $\Sigma$ where $\Sigma_{ij}=\frac1n\ip{\bs_i}{\bs_j}$  with $\bs_3 = \bs^*$. Note that $\bs_i\ne \pm \bs_j$ for $1\le i<j\le 3$, so Lemma~\ref{lemma:det-bd} is applicable. This yields
\begin{align*}
\mathbb{P}\bigl[(T_1,T_2,T_3)\in [-g2^{-n},g2^{-n}]\times [-g2^{-n},g2^{-n}]\times [-3^{-n},3^{-n}]\bigr]&\le (2\pi)^{-\frac32}|\Sigma|^{-\frac12}(2\cdot g 2^{-n})^2 \cdot (2\cdot 3^{-n}) \\ 
&\le (2\pi)^{-\frac32}n^{\frac32}(2\cdot g 2^{-n})^2 \cdot (2\cdot 3^{-n}),
\end{align*}
where we applied Lemma~\ref{lemma:anti-con} in the first line and Lemma~\ref{lemma:det-bd} in the second line. Hence, 
\[
\mathbb{P}\bigl[H(\bs_1)\le g2^{-n},H(\bs_2)\le g2^{-n}\mid H(\bs^*)\le 3^{-n}\bigr] \le (2\pi)^{-1}\cdot n^{\frac32}\cdot \left(2\cdot g2^{-n}\right)^2.\qedhere
\]
\end{proof}
\noindent Using Lemma~\ref{lemma:prob-middle-term}, we obtain
\begin{align*}
    \Sigma_2 &= \sum_{(\bs_1,\bs_2)\in\mathcal{T}_2}\mathbb{P}\bigl[H(\bs_1)\le g2^{-n},H(\bs_2)\le g2^{-n}\mid H(\bs^*)\le 3^{-n}\bigr] \\
    &\le |\mathcal{T}_2|\cdot (2\pi)^{-1}\cdot n^{\frac32}\cdot \left(2\cdot g2^{-n}\right)^2 \\  &\le \exp_2\left(n+nh_b\left(\frac{1-\epsilon}{2}\right)+1\right)\cdot (2\pi)^{-1}\cdot n^{\frac32}\cdot \left(2\cdot g2^{-n}\right)^2.
\end{align*}
Note that per Part {\rm (a)}, we have $\mathbb{E}_{\rm pl}|S_g(\epsilon)|\ge \frac{2g}{\sqrt{2\pi}}(1+o_n(1))$. So, 
\begin{equation}\label{eq:Sigma-2-auxil}
    \frac{\Sigma_2}{(\mathbb{E}_{\rm pl}|S_g(\epsilon)|)^2}\le \exp_2\left(-n+nh_b\left(\frac{1-\epsilon}{2}\right)+O(\log_2 n)\right) = 2^{-\Omega(n)}.
\end{equation}
\paragraph{Analysis of $\Sigma_3$} We now study $\Sigma_3$ appearing in~\eqref{eq:2nd-mom-auxil2}. 
\begin{lemma}\label{lemma:prob-last-term}
\[
\max_{(\bs_1,\bs_2)\in\mathcal{T}_3}  \mathbb{P}\bigl[H(\bs_1)\le g2^{-n},H(\bs_2)\le g2^{-n}|H(\bs^*)\le 3^{-n}\bigr]\le (1-\epsilon\sqrt{6})^{-\frac32}\left(\frac{2g2^{-n}}{\sqrt{2\pi}}\right)^2 (1+o_n(1)).
\]
\end{lemma}
\begin{proof}[Proof of Lemma~\ref{lemma:prob-last-term}]
To that end, fix an arbitrary $(\bs_1,\bs_2)\in\mathcal{T}_3$ and set $\bs_3=\bs^*$ for convenience. In particular, $d_H(\bs_i,\bs_j)\in 
 J(\epsilon) = \left[\frac{n(1-\epsilon)}{2},\frac{n(1+\epsilon)}{2}\right]$ for $1\le i<j\le 3$. Notice that 
 \[
 \frac{1}{n}\ip{\bs_i}{\bs_j} = \frac{n-2d_H(\bs_i,\bs_j)}{n}\in[-\epsilon,\epsilon].
 \]
 Analogous to above, set $T\sim \cN(0,I_n)$ and $Z_i=n^{-\frac12}\ip{\bs_i}{T}\sim \cN(0,I_n)$. Note that $(Z_1,Z_2,Z_3)$ forms a multivariate normal random vector with covariance $\Sigma$, where $\Sigma_{ii}=1$ for $1\le i\le 3$ and $|\Sigma_{ij}|\le \epsilon$ for $1\le i\ne j\le 3$. Now let 
 $\Sigma=I_{3\times 3}+E$ where $E\in \R^{3\times 3}$ with 
 $E_{ii}=0$ and $|E_{ij}|\le \epsilon$. In particular, $\|E\|_F\le \epsilon\sqrt{6}$. Denoting by $\lambda_1,\lambda_2,\lambda_3$ the eigenvalues of $\Sigma$, we obtain, by using Theorem~\ref{thm:wielandt-hoffman}, that 
 \[
 |\lambda_i-1|\le \sqrt{\sum_{1\le i\le }(\lambda_i-1)^2} \le \|E\|_F \le \epsilon\sqrt{6}.
 \]
 So, $\min_{1\le i\le 3}\lambda_i\ge 1-\epsilon\sqrt{6}$ and therefore, $|\Sigma| = \prod_{i\le 3}\lambda_i\ge (1-\epsilon\sqrt{6})^3$. Next, notice that~\eqref{eq:bayes-auxil} and~\eqref{eq:auxil-pb} remain the same. Proceeding similarly, 
 \begin{align*}
     \mathbb{P}\bigl[H(\bs_1)\le g2^{-n},H(\bs_2)\le g2^{-n}|H(\bs^*)\le 3^{-n}\bigr]&=\frac{ \mathbb{P}\bigl[H(\bs_1)\le g2^{-n},H(\bs_2)\le g2^{-n},H(\bs^*)\le 3^{-n}\bigr]}{ \mathbb{P}\bigl[H(\bs^*)\le 3^{-n}\bigr]} \\
     &\le \frac{(2\pi)^{-\frac32}|\Sigma|^{-\frac12}(2\cdot g\cdot 2^{-n})^2(2\cdot 3^{-n})}{\frac{2}{\sqrt{2\pi}}3^{-n}(1+o_n(1))} \\
     &\le (1-\epsilon\sqrt{6})^{-\frac32}\left(\frac{2g2^{-n}}{\sqrt{2\pi}}\right)^2 (1+o_n(1)).
 \end{align*}
 This concludes the proof of Lemma~\ref{lemma:prob-last-term}.
\end{proof}
\noindent Endowed with Lemma~\ref{lemma:prob-last-term}, we obtain
 \begin{align*}
     \Sigma_3 &=\sum_{(\bs_1,\bs_2)\in\mathcal{T}_3}\mathbb{P}\bigl[H(\bs_1)\le g2^{-n},H(\bs_2)\le g2^{-n}\mid H(\bs^*)\le 3^{-n}\bigr] \\
     &\le |\mathcal{T}_3|\cdot (1-\epsilon\sqrt{6})^{-\frac32}\left(\frac{2g2^{-n}}{\sqrt{2\pi}}\right)^2 (1+o_n(1)) \\
     &=(1-\epsilon\sqrt{6})^{-\frac32}(1+o_n(1))\left(\frac{2g}{\sqrt{2\pi}}\right)^2.
  \end{align*}
  where we used the trivial bound $|\mathcal{T}_3|\le 2^{2n}$. Using part {\rm (a)}, we thus obtain
\begin{equation}\label{eq:Sigma-3-auxil}
      \frac{\Sigma_3}{(\mathbb{E}_{\rm pl}|S_g(\epsilon)|)^2}\le (1-\epsilon\sqrt{6})^{-\frac32}(1+o_n(1)).
  \end{equation}
  \paragraph{Combining the estimates on $\Sigma_1,\Sigma_2,\Sigma_3$.}
Recall that 
\[
\mathbb{E}_{\rm pl}[|S_g(\epsilon)|^2] \le \Sigma_1+\Sigma_2+\Sigma_3
\]
per~\eqref{eq:2nd-mom-auxil2}. Combining~\eqref{eq:Sigma-1},~\eqref{eq:Sigma-2-auxil} and~\eqref{eq:Sigma-3-auxil} we obtain
\begin{align*}
   \frac{(\mathbb{E}_{\rm pl}|S_g(\epsilon)|)^2}{\mathbb{E}_{\rm pl}[|S_g(\epsilon)|^2]}&\ge \frac{(\mathbb{E}_{\rm pl}|S_g(\epsilon)|)^2}{\Sigma_1+\Sigma_2+\Sigma_3} \\
   &\ge \frac{1}{o_n(1)+2^{-\Omega(n)}+(1+o_n(1))\bigl(1-\epsilon\sqrt{6}\bigr)^{-\frac32}} \\
   &=\frac{1}{o_n(1)+(1+o_n(1))\bigl(1-\epsilon\sqrt{6}\bigr)^{-\frac32}},
\end{align*}
completing the proof of Proposition~\ref{prop:2nd-mom}{\rm (b)}.
 \end{proof}
 \subsection{Proof of Theorem~\ref{thm:zeta-rho}}\label{sec:pf-zeta-rho}
Note that $H(\bs^*)=H(-\bs^*)$ from symmetry. Further, if $d_H(\bs,\bs^*) = \rho n$ then $d_H(\bs,-\bs^*)=(1-\rho)n$. For these reasons,  we assume without loss of generality that $\rho\le \frac12$. To that end, set 
\begin{equation}\label{eq:bar-rho}
    \bar{\rho} = 1-2\rho.
\end{equation}
Further, for any $h:\N\to\R^+$, define
\begin{equation}\label{eq:S-f-rho}
 S_{h,\rho} \triangleq \bigl\{\bs\in\Sigma_n:d_H(\bs,\bs^*)=\rho n,H(\bs)\le h(n)\sqrt{n}2^{-nh_b(\rho)}\bigr\}.
 \end{equation}
 Similar to the proof of Theorem~\ref{thm:ground-state}, we establish the lower bound and upper bound hold individually w.h.p.\,under the planted measure. The result then follows via a union bound.
 \subsubsection*{Proof of Lower Bound}
 We first fix an $f=o_n(1)$ and establish that
 \[
 \zeta(\rho) \ge f(n)\sqrt{n}2^{-nh_b(\rho)}
 \]
 w.h.p., where $\zeta(\rho)$ is defined in~\eqref{eq:zeta-rho}. We start with an auxiliary result.
\begin{proposition}\label{prop:rho-low}
We have
\[
\mathbb{E}_{\rm pl}|S_{f,\rho}| = \Theta\left(\frac{f(n)}{\sqrt{(1-\bar{\rho}^2)}}\right). 
\]
In particular, if $f=o_n(1)$ then
 $\mathbb{E}_{\rm pl}|S_{f,\rho}|=o_n(1)$.
 \end{proposition}
 \begin{proof}[Proof of Proposition~\ref{prop:rho-low}]
     Note that
     \begin{align*}
             \mathbb{E}_{\rm pl}|S_{f,\rho}| &= \sum_{\bs:d_H(\bs,\bs^*)=\rho n}\mathbb{P}\bigl[H(\bs)\le f(n)\sqrt{n}2^{-nh_b(\rho)}\mid H(\bs^*)\le 3^{-n}\bigr] \\
             &=\binom{n}{\rho n}\frac{\mathbb{P}\bigl[H(\bs)\le f(n)\sqrt{n}2^{-nh_b(\rho)}, H(\bs^*)\le 3^{-n}\bigr]}{\mathbb{P}\bigl[H(\bs^*)\le 3^{-n}\bigr]}.
          \end{align*}    
    As $d_H(\bs,\bs^*)=\rho n$, we have $\frac1n\ip{\bs}{\bs^*} =\bar{\rho}$. In particular, if $(T,T_{\bar{\rho}})$ is a centered bivariate normal random vector with covariance matrix $\begin{pmatrix}
        1&\bar{\rho} \\\bar{\rho}&1
    \end{pmatrix}$, then
    \begin{align*}
    \mathbb{P}_\rho\bigl[H(\bs)\le f(n)\sqrt{n}2^{-nh_b(\rho)}, H(\bs^*)\le 3^{-n}\bigr]&=\mathbb{P}\bigl[(T,T_{\bar{\rho}})\in [-f(n)\sqrt{n}2^{-nh_b(\rho)},f(n)\sqrt{n}2^{-nh_b(\rho)}]\times [-3^{-n},3^{-n}] \bigr]\\
    &=\frac{1}{2\pi\sqrt{1-\bar{\rho}^2}}\bigl(2f(n)\sqrt{n}2^{-nh_b(\rho)}\bigr)\cdot\bigl(2\cdot 3^{-n}\bigr)(1+o_n(1)),
        \end{align*}
        using Lemma~\ref{lemma:gaussproblemma}(b). 
    Furthermore, $\mathbb{P}[H(\bs^*)\le 3^{-n}] =\frac{1}{\sqrt{2\pi}}(2\cdot 3^{-n})(1+o_n(1))$ per Lemma~\ref{lemma:gaussproblemma}. Moreover, 
        \[
        \binom{n}{\rho n} = \Theta\left(\frac{2^{nh_b(\rho)}}{\sqrt{n}}\right)
        \]
        using Lemma~\ref{lemma:binom-entropy}.
        Combining these estimates, we find that
        \[
        \mathbb{E}_{\rm pl}|S_{f,\rho}| = \Theta\left(\frac{f(n)}{\sqrt{(1-\bar{\rho}^2)}}\right),
        \]
        which is indeed $o_n(1)$ as $\rho=O(1)$ and $f(n)=o_n(1)$. This establishes Proposition~\ref{prop:rho-low}.
 \end{proof}
 Using Markov's inequality, i.e. 
 \[
 \mathbb{P}_{\rm pl}[|S_{f,\rho}|\ge 1]\le \mathbb{E}_{\rm pl}|S_{f,\rho}| = o_n(1),
 \]
 we conclude that $S_{f,\rho}=\varnothing$ w.h.p.\,under the planted model. So,  w.h.p.
 \[
 \zeta(\rho)\ge f(n)\sqrt{n}2^{-nh_b(\rho)}.
 \]
 \subsubsection*{Proof of Upper Bound}
 We first establish an auxilary result to manipulate a certain sum of binomial coefficients. 
\begin{lemma}\label{lemma:sum-binom}
    Fix $p\le \frac12$ and  $\epsilon>0$ sufficiently small. Then, for all large enough $n$,
    \[
    \sum_{k=(p(1-p)-\epsilon)n}^{(p(1-p)+\epsilon)n}\binom{pn}{k}\binom{(1-p)n}{k} = \binom{n}{pn}\left(1-2^{-n\Theta_\epsilon(\epsilon^2)}\right).
    \]
    Here, $\Theta_\epsilon(\epsilon^2)$ is a positive term that depends only on $\epsilon$ and $p$ and remains constant as $n\to\infty$.
\end{lemma}
\begin{proof}[Proof of Lemma~\ref{lemma:sum-binom}]
    First, we show that
\begin{equation}\label{eq:auxil1}
            \sum_{0\le k\le p n}\binom{pn}{k} \binom{(1-p)n}{k} = \binom{n}{pn}.
        \end{equation}
To see this, notice first that 
\[
\sum_{0\le k\le p n}\binom{p n}{k}\binom{(1-p)n}{p n-k} = 
\sum_{0\le k\le p n}\binom{p n}{p n-k}\binom{(1-p)n}{p n-k} = \sum_{0\le k\le p n} \binom{p n}{k}\binom{(1-p)n}{k}.
\]
So, it suffices to establish
\begin{equation}\label{eq:establish-THIS}
    \sum_{0\le k\le p n}\binom{p n}{k}\binom{(1-p)n}{p n-k} = \binom{n}{p n}.
\end{equation}
To show~\eqref{eq:establish-THIS}, consider the number of ways of choosing a group $G$ of $pn$ people out of $n$ people.  Splitting $n$ people into arbitrary sets $S_1,S_2$ with $|S_1|=pn$, we have $|G\cap S_1|=k$, where $0\le k\le pn$. So, a double counting argument yields~\eqref{eq:establish-THIS}.

Next, recalling Lemma~\ref{lemma:binom-entropy}, it thus suffices to show
\begin{equation}\label{eq:to-prove-auxil}
    \sum_{k\le (p(1-p)-\epsilon)n}\binom{pn}{k}\binom{(1-p)n}{k} +\sum_{k\le (p(1-p)+\epsilon)n}\binom{pn}{k} \binom{(1-p)n}{k}\le \exp_2\left(nh_b(p)- n\Theta_\epsilon(\epsilon^2)\right).
\end{equation}
Note that if $1\le k\le pn-1$ then $k(pn-k)\ge pn-1$ trivially. Further, for $1\le pn$, we have $k((1-p)n-k)\ge (1-p)n-1$. Now, suppose 
\[
n\ge \frac{\pi}{p(\pi-1)} \quad\text{and}\quad n\ge \frac{\pi}{(1-p)(\pi-1)}.
\]
Applying Lemma~\ref{lemma:binom-entropy}, we thus obtain that if $1\le k\le pn-1$ then
\[
\binom{p n}{k} \le \sqrt{\frac{pn}{\pi k(n-k)}}\exp_2\left(pn h_b\left(\frac{k}{p n}\right)\right)\le \sqrt{\frac{p n}{\pi(pn-1)}}\exp_2\left(pn h_b\left(\frac{k}{p n}\right)\right)\le \exp_2\left(pn h_b\left(\frac{k}{p n}\right)\right).
\]
Similarly, 
\[
\binom{(1-p)n}{k}\le \sqrt{\frac{(1-p)n}{\pi ((1-p)n-1)}}\exp_2\left((1-p)nh_b\left(\frac{k}{(1-p)n}\right)\right)\le \exp_2\left((1-p)nh_b\left(\frac{k}{(1-p)n}\right)\right).
\]
So, letting $k=\alpha n$, we arrive at
\[
\binom{pn}{k}\binom{(1-p)n}{k}\le \exp_2\left(pnh\left(\frac{\alpha}{p}\right)+(1-p)nh\left(\frac{\alpha}{1-p}\right)\right).
\]
Set
\[
\varphi(\alpha) \triangleq ph\left(\frac{\alpha}{p}\right)+(1-p)h\left(\frac{\alpha}{1-p}\right),
\]
and recall that $\frac{d}{dp}h(p) = \log_2 e \cdot \ln \frac{1-p}{p}$. Using this,
\begin{align*}
    \varphi'(\alpha) &= \log_2 e \left(\ln\frac{1-\alpha/p}{\alpha/p}+\ln\frac{1-\alpha/(1-p)}{\alpha/(1-p)}\right)\\
    &=\log_2 e\cdot \ln\left(\frac{p(1-p)}{\alpha^2}-\frac1\alpha+1\right).
\end{align*}
From here, we immediately obtain that $\varphi'(\alpha)>0$ for $\alpha<p(1-p)$, $\varphi'(p(1-p))=0$ and $\varphi'(\alpha)<0$ for $\alpha>p(1-p)$. Hence, 
\begin{align}
\max_{k\le p(1-p)n-\epsilon n}\binom{pn}{k}\binom{(1-p)n}{k} &= \binom{pn}{(p(1-p)-\epsilon)n} \binom{(1-p)n}{(p(1-p)-\epsilon)n},\label{eq:importante1} \\
\max_{pn\ge k\ge p(1-p)n+\epsilon n} \binom{pn}{k}\binom{(1-p)n}{k} &= \binom{pn}{(p(1-p)+\epsilon)n} \binom{(1-p)n}{(p(1-p)+\epsilon)n}.\label{eq:importante2}
\end{align}
Equipped with~\eqref{eq:importante1} and~\eqref{eq:importante2}, we obtain
\begin{align}
    & \sum_{k\le (p(1-p)-\epsilon)n}\binom{pn}{k}\binom{(1-p)n}{k} +\sum_{k\le (p(1-p)+\epsilon)n}\binom{pn}{k} \binom{(1-p)n}{k} \nonumber\\ 
    &\le  n^{O(1)}\binom{pn}{(p(1-p)-\epsilon)n} \binom{(1-p)n}{(p(1-p)-\epsilon)n} + n^{O(1)}\binom{pn}{(p(1-p)+\epsilon)n} \binom{(1-p)n}{(p(1-p)+\epsilon)n} \nonumber\\
    &=\exp_2\left(pnh_b\left(1-p-\frac{\epsilon}{p}\right)+(1-p)nh_b\left(p-\frac{\epsilon}{1-p}\right)+O(\log_2 n)\right) \nonumber\\
    &+\exp_2\left(pnh_b\left(1-p+\frac{\epsilon}{p}\right)+(1-p)n h_b\left(p+\frac{\epsilon}{1-p}\right)+O(\log_2 n)\right),\label{eq:penultimate}
\end{align}
where we used once again Lemma~\ref{lemma:binom-entropy}.

Next, note that $h'(p) = \log_2 e\cdot \ln \frac{1-p}{p}$ and $h''(p) = -\frac{\log_2 e}{p(1-p)}$. So, by a Taylor expansion around $p$,
\begin{align*}
   ph_b\left(1-p-\frac{\epsilon}{p}\right)+(1-p)h_b\left(p-\frac{\epsilon}{1-p}\right)&=ph_b\left(p+\frac{\epsilon}{p}\right)+(1-p)h_b\left(p-\frac{\epsilon}{1-p}\right) \\&=h_b(p) + \sum_{n\ge 1}\frac{h^{(n)}(p)}{n!}\left(p\left(\frac{\epsilon}{p}\right)^n +(1-p)\left(-\frac{\epsilon}{1-p}\right)^n\right) \\
    &=h_b(p)-\frac{\log_2 e}{2p^2(1-p)^2}\epsilon^2 +O_\epsilon(\epsilon^3).
\end{align*}
Similarly, 
\begin{align*}
    ph_b\left(1-p+\frac{\epsilon}{p}\right)+(1-p)h_b\left(p+\frac{\epsilon}{1-p}\right)&=ph_b\left(p-\frac{\epsilon}{p}\right)+(1-p)h_b\left(p+\frac{\epsilon}{1-p}\right) \\
    &=h_b(p)-\frac{\log_2 e}{2p^2(1-p)^2}\epsilon^2 +O_\epsilon(\epsilon^3).
\end{align*}
Returning to~\eqref{eq:penultimate}, we thus conclude
\[
\sum_{k\le (p(1-p)-\epsilon)n}\binom{pn}{k}\binom{(1-p)n}{k} +\sum_{k\le (p(1-p)+\epsilon)n}\binom{pn}{k} \binom{(1-p)n}{k} \le \exp_2\left(nh_b(p)-n\Theta_\epsilon(\epsilon^2)\right),
\]
where $\Theta_\epsilon(\epsilon^2)$ term is positive and depends only on $p$ and $\epsilon$. This establishes Lemma~\ref{lemma:sum-binom}.
\end{proof}
We next establish the following proposition.
\begin{proposition}\label{prop:rho-2nd-mom}
     Let $g=\omega(1)$ and set
     \[
     \lambda(\rho)\triangleq \frac{\bar{\rho}^2+2-\bar{\rho}\sqrt{\bar{\rho}^2+8}}{2},
     \]
     where we recall that $\bar{\rho}=1-2\rho$ and that $\rho\le \frac12$. Then, the following holds for any $\epsilon>0$ small enough.
     \begin{itemize}
         \item[(a)] \begin{align*}
   \mathbb{E}_{\rm pl}|S_{g,\rho}|&=\binom{n}{\rho n} \frac{1+o_n(1)}{\sqrt{2\pi(1-\bar{\rho}^2)}}\Bigl(2g(n)\sqrt{n}2^{-nh_b(\rho)}\Bigr) =\Omega\left(\frac{g}{\sqrt{1-\bar{\rho}^2}}\right)=\omega(1).
                  \end{align*}

         \item[(b)] \[\frac{(\mathbb{E}_{\rm pl}|S_{g,\rho}|)^2}{\mathbb{E}_{\rm pl}[|S_{g,\rho}|^2]}\ge \frac{1}{o_n(1)+(1+o_n(1))\left(1-\epsilon/\lambda(\rho)\right)^{-\frac32}}.
         \]
     \end{itemize}
 \end{proposition}
 Before proving Proposition~\ref{prop:rho-2nd-mom}, we first show how it yields the upper bound in Theorem~\ref{thm:zeta-rho}. Using \emph{second moment method},~\eqref{eq:2nd-mom-met}, we have
 \[
 \mathbb{P}_{\rm pl}[|S_{g,\rho}|\ge 1]\ge \frac{(\mathbb{E}_{\rm pl}|S_{g,\rho}|)^2}{\mathbb{E}_{\rm pl}[|S_{g,\rho}|^2]}\ge \frac{1}{o_n(1)+(1+o_n(1))\left(1-\epsilon/\lambda(\rho)\right)^{-\frac32}}.
 \]
 So, 
 \[
\liminf_{n\to\infty}\mathbb{P}_{\rm pl}[|S_{g,\rho}|\ge 1]\ge  \left(1-\epsilon/\lambda(\rho)\right)^{\frac32}.
 \]
 Noting that the LHS is independent of $\epsilon$ and the inequality above is valid for all small enough $\epsilon>0$, we send $\epsilon\to 0$ to obtain
 \[
\liminf_{n\to\infty}\mathbb{P}_{\rm pl}[|S_{g,\rho}|\ge 1]\ge 1.
 \]
 Since we trivially have $\limsup_{n\to\infty}\mathbb{P}_{\rm pl}[|S_{g,\rho}|\ge 1]\le 1$, we conclude
 \[
 \lim_{n\to\infty}\mathbb{P}_{\rm pl}[|S_{g,\rho}|\ge 1] = 1,
 \]
 i.e. 
 \[
 \zeta(\rho) = \min_{\bs:d_H(\bs,\bs^*)=\rho n}H(\bs)\le g(n)\sqrt{n}2^{-nh_b(\rho)}
 \]
 w.h.p. Hence, it suffices to prove Proposition~\ref{prop:rho-2nd-mom}.
 \begin{proof}[Proof of Proposition~\ref{prop:rho-2nd-mom}]
     Part {\rm (a)} follows immediately by modifying the proof of Proposition~\ref{prop:rho-low}. So, it suffices to prove part {\rm (b)}. We begin by writing
    \begin{equation}\label{eq:2nd-mom-rho}
              \mathbb{E}_{\rm pl}\bigl[\bigl|S_{g,\rho}\bigr|^2\bigr] = \sum_{\substack{\bs_1,\bs_2\in\Sigma_n \\ d_H(\bs_1,\bs^*)=\rho n\\ d_H(\bs_2,\bs^*)=\rho n}} \mathbb{P}\bigl[H(\bs_1)\le g(n)\sqrt{n}2^{-nh_b(\rho)},H(\bs_2)\le g(n)\sqrt{n}2^{-nh_b(\rho)}\mid H(\bs^*)\le 3^{-n}\bigr].
          \end{equation}
     Note that there are $\binom{n}{\rho n}$ choices for $\bs_1$ with $d_H(\bs_1,\bs^*)=\rho n$. Fix any such $\bs_1$, let $I=\{i\in[n]:\bs_1(i)\ne \bs^*(i)\}$ and $I^c=[n]\setminus I$. Then $|I|=\rho n$ and $|I^c|=(1-\rho)n$. We now bound the number of ways of choosing a $\bs_2$. Suppose $k=\{i\in I:\bs_1(i)\ne \bs_2(i)\}$ and $\ell = \{i\in I^c:\bs_1(i)\ne \bs_2(i)\}$. Observe that
     \[
    \rho n =  d_H(\bs_2,\bs^*) = (\rho n -k)+\ell  \implies k=\ell.
     \]
     So, the number of ways of choosing any such $\bs_2$ is
     \begin{equation}\label{eq:coun-for-bs2}
       \binom{n}{\rho n} =   \sum_{0\le k\le \rho n}\binom{\rho n}{k}\binom{(1-\rho)n}{k},
     \end{equation}
     where $d_H(\bs_1,\bs_2)=2k$. Note that as $\rho \le \frac12$, $\binom{(1-\rho)n}{k}>0$ for any $0\le k\le \rho n$. 
     
     Similar to the proof of Proposition~\ref{prop:2nd-mom}, we now split the pairs $(\bs_1,\bs_2)$ appearing in~\eqref{eq:2nd-mom-rho} into three disjoint classes. To that end, fix an $\epsilon>0$ small enough.
     \begin{itemize}
    \item Let 
        \[
        \mathcal{T}_1 = \bigl\{(\bs_1,\bs_2):d_H(\bs_1,\bs^*)=d_H(\bs_2,\bs^*)=\rho n,\bs_2=\pm \bs^*\bigr\}.
        \]
        In particular, $|\mathcal{T}_1|\le 2 \binom{n}{\rho n}$.
        \item For
        \[
        I(\rho,\epsilon) = [2,2(\rho(1-\rho)-\epsilon)n]\cup [2(\rho(1-\rho)+\epsilon)n,n-1],
        \]
        let
        \[
         \mathcal{T}_2 = \bigl\{(\bs_1,\bs_2):d_H(\bs_1,\bs^*)=d_H(\bs_2,\bs^*)=\rho n,d_H(\bs_1,\bs_2)\in I(\rho,\epsilon)\bigr\}.
        \]
        In particular, using Lemma~\ref{lemma:sum-binom} and the identity in~\eqref{eq:coun-for-bs2}, we have 
        \begin{align}
                |\mathcal{T}_2|&\le \binom{n}{\rho n}\left(\sum_{0\le k\le n(\rho(1-\rho)-\epsilon)}\binom{\rho n}{k}\binom{(1-\rho)n}{k} + \sum_{n(\rho(1-\rho)+\epsilon)\le k\le n}\binom{\rho n}{k}\binom{(1-\rho)n}{k}\right)\nonumber \\
                &\le \binom{n}{\rho n}^2 2^{-n\Theta_\epsilon(\epsilon^2)}.\label{eq:T-2-card}
                \end{align}
                \item Finally, let
                \[
                J(\rho,\epsilon) = \bigl[2\bigl(\rho(1-\rho)-\epsilon\bigr)n,2\bigl(\rho(1-\rho)+\epsilon\bigr)n\bigr],
                \]
                and set
                \[
                 \mathcal{T}_3= \bigl\{(\bs_1,\bs_2):d_H(\bs_1,\bs^*)=d_H(\bs_2,\bs^*)=\rho n,d_H(\bs_1,\bs_2)\in J(\rho,\epsilon)\bigr\}.
                \]
                Note that 
                \[
                |\mathcal{T}_3|\le \binom{n}{\rho n}^2
                \]
                trivially.
     \end{itemize}
     Next, we decompose the sum in~\eqref{eq:2nd-mom-rho} as follows:
     \begin{align}
         \mathbb{E}_{\rm pl}\bigl[\bigl|S_{g,\rho}\bigr|^2\bigr]&= \underbrace{\sum_{(\bs_1,\bs_2)\in\mathcal{T}_1} \mathbb{P}\bigl[H(\bs_1)\le g(n)\sqrt{n}2^{-nh_b(\rho)},H(\bs_2)\le g(n)\sqrt{n}2^{-nh_b(\rho)}\mid H(\bs^*)\le 3^{-n}\bigr]}_{\Sigma_1}\nonumber\\ 
         &+\underbrace{\sum_{(\bs_1,\bs_2)\in\mathcal{T}_2} \mathbb{P}\bigl[H(\bs_1)\le g(n)\sqrt{n}2^{-nh_b(\rho)},H(\bs_2)\le g(n)\sqrt{n}2^{-nh_b(\rho)}\mid H(\bs^*)\le 3^{-n}\bigr]}_{\Sigma_2}\nonumber \\
&+\underbrace{\sum_{(\bs_1,\bs_2)\in\mathcal{T}_3} \mathbb{P}\bigl[H(\bs_1)\le g(n)\sqrt{n}2^{-nh_b(\rho)},H(\bs_2)\le g(n)\sqrt{n}2^{-nh_b(\rho)}\mid H(\bs^*)\le 3^{-n}\bigr]}_{\Sigma_3}\label{eq:2nd-mom-rho-up}.
     \end{align}
     \paragraph{Analysis of $\Sigma_1$ and $\Sigma_2$.} We immediately have
     \[
\Sigma_1=2\sum_{\bs_1:d_H(\bs_1,\bs^*)=\rho n}\mathbb{P}\bigl[H(\bs_1)\le g(n)\sqrt{n}2^{-nh_b(\rho)}\mid H(\bs^*)\le 3^{-n}\bigr]= 2\mathbb{E}_{\rm pl}|S_{g,\rho}|.
     \]
     In particular, using part {\rm (a)} of Proposition~\ref{prop:rho-2nd-mom}, we get
     \begin{equation}
         \label{eq:sigma-1-rho-up}
         \frac{\Sigma_1}{(\mathbb{E}_{\rm pl}|S_{g,\rho}|)^2} = o_n(1).
     \end{equation}
     We now analyze $\Sigma_2$. To that end, fix any $(\bs_1,\bs_2)\in \mathcal{T}_2$. Note that from the definition of $I(\rho,\epsilon)$, $\bs_2\ne \pm \bs_1$. Further, as $\rho\ne 0,1$, we also have $\bs^*\ne \pm \bs_1$ and $\bs^*\ne \pm \bs_2$. Hence, a straightforward modification of Lemma~\ref{lemma:prob-middle-term} immediately yields
     \begin{align}
     &\mathbb{P}\bigl[H(\bs_1)\le g(n)\sqrt{n}2^{-nh_b(\rho)},H(\bs_2)\le g(n)\sqrt{n}2^{-nh_b(\rho)}\mid H(\bs^*)\le 3^{-n}\bigr] \nonumber\\
     &\le (2\pi)^{-1}\cdot n^{\frac32}\cdot \left(2\cdot g(n)\sqrt{n}2^{-nh_b(\rho)}\right)^2.\label{eq:prob-middle-term}
          \end{align}
           Further, $\binom{n}{\rho n} = \Theta(n^{-\frac12}2^{nh_b(\rho)})$ per Lemma~\ref{lemma:binom-entropy}. Combining~\eqref{eq:T-2-card} and~\eqref{eq:prob-middle-term}, we obtain
\begin{align*}
    \Sigma_2&\le |\mathcal{T}_2|\cdot (2\pi)^{-1}\cdot n^{\frac32}\cdot \left(2\cdot g(n)\sqrt{n}2^{-nh_b(\rho)}\right)^2 \\
    &\le 2^{-n\Theta_\epsilon(\epsilon^2)}n^{\frac32}\cdot \left(\frac{2g(n)}{\sqrt{2\pi}}\right)^2.
\end{align*}
As $\mathbb{E}_{\rm pl}|S_{g,\rho}| = \frac{2g(n)}{\sqrt{2\pi(1-\bar{\rho}^2)}}(1+o_n(1))$ per part {\rm (a)} and $\bar{\rho}=1-2\rho = O(1)$, we obtain
\begin{equation}\label{eq:Sigma2-for-rho}
    \frac{\Sigma_2}{(\mathbb{E}_{\rm pl}|S_{g,\rho}|)^2}\le \exp_2\bigl(-n\Theta_\epsilon(\epsilon^2)+O(\log_2 n)\bigr) = 2^{-\Theta(n)}.
\end{equation}
\paragraph{Analysis of $\Sigma_3$.}
Fix any $(\bs_1,\bs_2)\in\mathcal{T}_3$. We first estimate the probability term. 
\begin{lemma}\label{lemma:pho-Sigma3-pb}
    Set
    \[
    \lambda(\rho) = \frac{\bar{\rho}^2 +2 - \bar{\rho}\sqrt{\bar{\rho}^2+8}}{2},\quad\text{where}\quad \bar{\rho}=1-2\rho.
    \]
    Then, 
    \begin{align*}
    &\mathbb{P}\bigl[H(\bs_1)\le g(n)\sqrt{n}2^{-nh_b(\rho)},H(\bs_1)\le g(n)\sqrt{n}2^{-nh_b(\rho)}\mid H(\bs^*)\le 3^{-n}\bigr]\\
    &\le (2\pi)^{-1}(1-\bar{\rho}^2)^{-1}\left(1-\frac{4\epsilon\sqrt{2}}{\lambda(\rho)}\right)^{-\frac32}\Bigl(2g(n)\sqrt{n}2^{-nh_b(\rho)}\Bigr)^2.
    \end{align*}
\end{lemma}
\begin{proof}[Proof of Lemma~\ref{lemma:pho-Sigma3-pb}]
    Similar to the proof of Lemma~\ref{lemma:prob-middle-term}, introduce a $T\sim \cN(0,I_n)$ and set $Z_i=n^{-\frac12}\ip{\bs_i}{T}$ where $\bs_3 = \bs^*$ for convenience. Further, denote by $\Sigma$ the covariance matrix of multivariate normal random vector $(Z_1,Z_2,Z_3)$. Applying  Lemma~\ref{lemma:gaussproblemma} as in the proof of Lemma~\ref{lemma:prob-middle-term}, we get
    \begin{align}
    &\mathbb{P}\bigl[H(\bs_1)\le g(n)\sqrt{n}2^{-nh_b(\rho)},H(\bs_1)\le g(n)\sqrt{n}2^{-nh_b(\rho)}\mid H(\bs^*)\le 3^{-n}\bigr] \nonumber \\ &\le (2\pi)^{-1}|\Sigma|^{-\frac12}\Bigl(2g(n)\sqrt{n}2^{-nh_b(\rho)}\Bigr)^2 \label{eq:temp1}.
        \end{align}
        We now study $|\Sigma|$. To that end, we have $\Sigma_{ii}=1$ for $1\le i\le 3$. Moreover, as $d_H(\bs_1,\bs^*)=d_H(\bs_2,\bs^*)=\rho n$, we have
        \[
    \frac{1}{n}\ip{\bs_1}{\bs^*} = 
    \frac{1}{n}\ip{\bs_1}{\bs^*} = 1-2\rho = \bar{\rho}.
        \]
        Now, as $(\bs_1,\bs_2)\in\mathcal{T}_3$, we have
        \[
d_H(\bs_1,\bs_2)\in\bigl[2n\bigl(\rho(1-\rho)-\epsilon\bigr),2n\bigl(\rho(1-\rho)+\epsilon\bigr)\bigr].
        \]
        Hence, 
        \[
        \frac1n\ip{\bs_1}{\bs_2} =\bar{\rho}^2 +\Delta,\quad\text{where}\quad -4\epsilon\le \Delta\le 4\epsilon.
        \]
        Set $\Sigma = \Sigma'+E$ where
        \[
        \Sigma' = \begin{pmatrix}
            1 &\bar{\rho}^2 & \bar{\rho} \\
            \bar{\rho}^2 & 1 & \bar{\rho} \\
            \bar{\rho} & \bar{\rho} & 1
        \end{pmatrix}\quad\text{and}\quad E = \begin{pmatrix}
            0 & \Delta & 0 \\ \Delta & 0 & 0 \\ 0 &0 & 0
        \end{pmatrix}.
        \]
        We now find the eigenvalues of $\Sigma'$. First,
        \[
        |\lambda I_{3\times 3}-\Sigma'| = |-\lambda^3 + 3\lambda^2 -3\lambda + (\lambda+1) \bar{\rho}^4 +2\bar{\rho}^2(\lambda-1)+1|.
        \]
        Solving $|\lambda I-\Sigma'|=0$, we find that the eigenvalues of $\Sigma'$ are 
        \[
        \left\{1-\bar{\rho}^2,\frac{\bar{\rho}^2 +2 +\bar{\rho}\sqrt{\bar{\rho}^2+8}}{2},\frac{\bar{\rho}^2 +2 - \bar{\rho}\sqrt{\bar{\rho}^2+8}}{2}\right\}.
        \]
        It is easy to verify, via a simple algebra, that
        \[
        1-\bar{\rho}^2 \ge \frac{\bar{\rho}^2+2 -\bar{\rho}\sqrt{\bar{\rho}^2+8}}{2}
        \]
         as $\bar{\rho}\le 1$. 
        In particular, $\lambda(\rho)$ is the smallest eigenvalue of $\Sigma'$ and $|\Sigma'| = (1-\bar{\rho}^2)^2$. Now, denote the eigenvalues of $\Sigma'$ by $\lambda_1\ge \lambda_2\ge \lambda_3= \lambda(\rho)$ and that of $\Sigma$ by $\mu_1\ge \mu_2\ge \mu_3$.        Using Hoffman-Wieland Theorem, Theorem~\ref{thm:wielandt-hoffman}, we obtain
        \[
        |\mu_i-\lambda_i|\le \sqrt{\sum_{i\le 3}(\mu_i-\lambda_i)^2}\le \|E\|_F\le 4\epsilon\sqrt{2}.
        \]
        So, $\mu_i\ge \lambda_i-4\epsilon\sqrt{2}>0$, provided $\epsilon>0$ is small enough. Hence,
        \begin{align}
                |\Sigma| = \mu_1\mu_2\mu_3 &\ge \prod_{i\le 3}(\lambda_i - 4\epsilon\sqrt{2})\nonumber \\
                &\ge \lambda_1\lambda_2\lambda_3 \prod_{i\le 3}\left(1-\frac{4\epsilon\sqrt{2}}{\lambda_i}\right) \nonumber\\
                &\ge (1-\bar{\rho}^2)^2\left(1-\frac{4\epsilon\sqrt{2}}{\lambda(\rho)}\right)^3,\label{eq:det-auxil}
                \end{align}
using $|\Sigma'|=\lambda_1\lambda_2\lambda_3=(1-\bar{\rho}^2)^2$. Inserting~\eqref{eq:det-auxil} into~\eqref{eq:temp1} immediately yields Lemma~\ref{lemma:pho-Sigma3-pb}.
\end{proof}
Using Lemma~\ref{lemma:pho-Sigma3-pb}, the fact $|\mathcal{T}_3|\le \binom{n}{\rho n}^2$ and part {\rm (a)}, we conclude 
\begin{align*}
\Sigma_3 &\le \underbrace{\binom{n}{\rho n}^2 (2\pi)^{-1}(1-\bar{\rho}^2)^{-1}\Bigl(2g(n)\sqrt{n}2^{-nh_b(\rho)}\Bigr)^2 }_{=(\mathbb{E}_{\rm pl}|S_{g,\rho}|)^2(1+o_n(1))}\left(1-\frac{4\epsilon\sqrt{2}}{\lambda(\rho)}\right)^{-\frac32}.
\end{align*}
Hence,
\begin{equation}\label{eq:rho-Sigma3}
    \frac{\Sigma_3}{(\mathbb{E}_{\rm pl}|S_{g,\rho}|)^2}\le (1+o_n(1))\left(1-\frac{4\epsilon\sqrt{2}}{\lambda(\rho)}\right)^{-\frac32}.
\end{equation}
Combining~\eqref{eq:2nd-mom-rho-up},~\eqref{eq:sigma-1-rho-up},~\eqref{eq:Sigma2-for-rho} and~\eqref{eq:rho-Sigma3}, we immediately conclude 
\[
\frac{(\mathbb{E}_{\rm pl}|S_{g,\rho}|)^2}{\mathbb{E}_{\rm pl}[|S_{g,\rho}|]^2}\ge \frac{1}{o_n(1) + (1+o_n(1))\left(1-\frac{4\epsilon\sqrt{2}}{\lambda(\rho)}\right)^{-3/2}},
\]
yielding Proposition~\ref{prop:rho-2nd-mom}.
 \end{proof}
 \subsection{Proof of Theorem~\ref{thm:ogp-planted}}\label{sec:ogp-planted}
 As we mentioned, the proof is based on a simple application of the first moment method.
 \subsubsection*{Part ${\rm (a)}$}
Fix any $\epsilon>0$. We show $\mathbb{E}_{\rm pl}|S^*(\beta n,\epsilon n)| = 2^{-\Theta(n)}$ for a suitable $\beta\in(0,1/2]$. The conclusion then follows via Markov's inequality. To that end, we have
\begin{align*}
    \mathbb{E}_{\rm pl}|S^*(\beta n,\epsilon n)| &= \sum_{\bs:0<d_H(\bs,\bs^*)\le \beta n} \mathbb{P}[H(\bs)\le 2^{-\epsilon n}\mid H(\bs^*)\le 3^{-n}].
\end{align*}
For any $\bs\ne \pm \bs^*$, $\rho\triangleq n^{-1}\ip{\bs}{\bs^*}\le 1-\frac{2}{n}$. So, 
\begin{align}
    \mathbb{P}[H(\bs)\le 2^{-\epsilon n}\mid H(\bs^*)\le 3^{-n}]&=\frac{\mathbb{P}[H(\bs)\le 2^{-\epsilon n}, H(\bs^*)\le 3^{-n}]}{\mathbb{P}[H(\bs^*)\le 3^{-n}]}\nonumber\\
    &\le \frac{\sqrt{n}}{\sqrt{2\pi}}2^{-\epsilon n}(1+o_n(1))\label{eq:use-later},
\end{align}
using bivariate normal probability and Lemma~\ref{lemma:gaussproblemma}. Moreover, using Lemma~\ref{lemma:sum-of-bin}
\[
\bigl|\bigl\{\bs:0<d_H(\bs,\bs^*)\le \beta n\bigr\}\bigr|\le \sum_{k\le \beta n}\binom{n}{k}\le 2^{nh_b(\beta)}.
\]
So, 
\[
\mathbb{E}_{\rm pl}|S^*(\beta n,\epsilon n)| \le \exp_2\bigl(-n\bigl(\epsilon-h_b(\beta)\bigr) + O(\log n)\bigr),
\]
which is indeed $\exp_2(-\Theta(n))$ for any $\beta$ with $h_b(\beta)<\epsilon$.
\subsubsection*{Part ${\rm (b)}$}
The case $E=\Theta(n)$ is covered by Part ${\rm (a)}$. So, suppose
\[
\Omega(\log n)\le E \le o(n).
\]
In particular, we assume that there is a constant $C>\frac32$ and an $n_0\in\mathbb{N}$ such that $E\ge C\log_2 n$ for all $n\ge n_0$. We now show that for a suitable $d$, $\mathbb{E}_{\rm pl}|S^*(d,E)| = 2^{-\Theta(E)}$ and finish by applying Markov's inequality.

To that end, the argument in Part ${\rm (a)}$ immediately yields
\[
\mathbb{P}\bigl[H(\bs)\le 2^{-E}\mid H(\bs^*)\le 3^{-n}\bigr] \le  \frac{\sqrt{n}}{\sqrt{2\pi}}2^{-E}(1+o_n(1)).
\]
Furthermore, if $d=o(n)$, then using Lemma~\ref{lemma:binom-sublinear} as well as the fact that $k\mapsto\binom{n}{k}$ is increasing on $0\le k\le n/2$, we obtain that
\[
\log_2\bigl|\bigl\{\bs:0<d_H(\bs,\bs^*)\le d\bigr\}\bigr|\le \log_2\left(\sum_{1\le k\le d}\binom{n}{k} \right)\le (1+o_n(1))d\log_2\frac{n}{d} + \log_2 d.
\]
As $d\le n$, we trivially have $\log_2 d\le \log_2 n$. So, we obtain
\[
\mathbb{E}_{\rm pl}|S^*(d,E)|\le \exp_2\left(-E + (1+o_n(1))d\log_2\frac{n}{d} +\frac32\log_2 n +O(1)\right).
\]
Now, fix a $c<1$ and let
\[
d=\frac{cE}{\log_2(n/E)}.
\]
Observe that as $E=o(n)$, we have that $n/E=\omega(1)$. So,
\begin{align*}
    d\log_2\frac{n}{d} &= \frac{cE}{\log_2(n/E)}\log_2\left(\frac{n}{cE}\log_2\left(\frac{n}{E}\right)\right) \\
    &=\frac{cE}{\log_2(n/E)}\left(\log_2\left(\frac{n}{E}\right)+\log_2\log_2\left(\frac{n}{E}\right) + \log_2\frac1c\right) \\
    &=cE(1+o_n(1)),
\end{align*}
where we used the fact 
\[
\log_2\log_2\left(\frac{n}{E}\right)+\log_2\frac1c = o\left(\log_2\left(\frac{n}{E}\right)\right).
\]
We choose $c>0$ such that $C(1-c)>\frac32$. Such a $c>0$ indeed exists as $C>\frac32$. With this choice, we indeed have
\[
\mathbb{E}_{\rm pl}|S^*(d,E)| = \exp_2\bigl(-\Theta(E)\bigr),
\]
completing the proof.
\subsection{Proof of Theorem~\ref{thm:m-ogp-planted-npp}}\label{sec:proof-m-ogp}
Our proof is based on the first moment method. Fix an $\epsilon>0$ and an $\mathcal{I}\subset [0,\pi/2]$ such that:
\begin{itemize}
    \item $|\mathcal{I}|\le 2^{cn}$ for some $c>0$ to be tuned.
    \item $\min_{\tau \in \mathcal{I},\tau\ne 0}\tau\ge 2^{-\delta n}$ for some $\delta<\epsilon$, again to be tuned.
\end{itemize}
For some $0<\eta<\beta<1$ and $m\in\N$ to be tuned, define
\[
\mathcal{F}(m,\beta,\eta) = \bigl\{(\bs_1,\dots,\bs_m):\bs_i\ne \pm\bs^*,\beta-\eta\le n^{-1}\ip{\bs_i}{\bs_j}\le \beta,1\le i<j\le m\bigr\}.
\]
Then
\begin{equation}\label{eq:main-RV}
    \bigl|\mathcal{S}(m,\beta,\eta,\epsilon n,\mathcal{I})\bigr| = \sum_{(\bs_1,\dots,\bs_m)\in\mathcal{F}(m,\beta,\eta)}\ind\Bigl\{\exists \tau_i\in\mathcal{I}:H(\bs_i,Y_i(\tau_i))\le 2^{-\epsilon n}\Bigr\},
\end{equation}
where
\begin{equation}\label{eq:Y_i-tau}
    Y_i(\tau_i) = \cos(\tau_i)X_0 + \sin(\tau_i)X_i\in \R^n,\quad 1\le i\le m.
\end{equation}
We establish that w.h.p.\,under the planted measure,
\[
\mathcal{S}(m,\beta,\eta,\epsilon n,\mathcal{I})=\varnothing.
\]
\paragraph{Counting Estimate} We first control $|\mathcal{F}(m,\beta,\eta)|$. To that end, there exists $2^n-2$ choices for $\bs_1\ne \pm \bs^*$. Next, for any fixed $\bs_1$, there exists $\binom{n}{n\frac{1-\rho}{2}}$ choices for $\bs_i$ with $n^{-1}\ip{\bs_1}{\bs_i}=\rho$. Hence, there are 
\[
\sum_{\substack{\rho : \rho n\in\mathbb{N} \\ \beta-\eta\le\rho\le \beta}}\binom{n}{n\frac{1-\rho}{2}}
\]
choices for $\bs_i$ with $n^{-1}\ip{\bs_1}{\bs_i}\in[\beta-\eta,\beta]$. So, 
\begin{align}
    |\mathcal{F}(m,\beta,\eta)|&\le (2^n-2)\left(\sum_{\substack{\rho : \rho n\in\mathbb{N} \\ \beta-\eta\le\rho\le \beta}}\binom{n}{n\frac{1-\rho}{2}}\right)^{m-1} \nonumber \\
    &\le 2^n\left(n^{O(1)}\binom{n}{n\frac{1-\beta+\eta}{2}}\right)^m\nonumber \\
    &\le \exp_2\left(n+mnh_b\left(\frac{1-\beta+\eta}{2}\right)+O(m\log n)\right) \label{eq:stirr},
\end{align}
where we applied Lemma~\ref{lemma:binom-entropy} to obtain~\eqref{eq:stirr}.
\paragraph{Probability Estimate} We first highlight the planted measure. Let $X_0,\dots,X_m\sim \cN(0,I_n)$ be i.i.d. For $\bs^*\in\Sigma_n$, and any event $\mathcal{E}$, define the planted measure
\[
\mathbb{P}_{\rm pl}[\mathcal{E}] =\mathbb{P}\left[\mathcal{E}\Big\lvert \max_{0\le j\le m}H(\bs^*,X_j)\le 3^{-n}\right].
\]
We next control the probability term, 
\[
\mathbb{P}_{\rm pl}\left[\exists \tau_1,\dots,\tau_m\in\mathcal{I}:\max_{1\le i\le m}H\bigl(\bs_i,Y_i(\tau_i)\bigr)\le 2^{-\epsilon n}\right],
\]
where $\bs_i\ne \pm\bs^*$, $1\le i\le m$.
\begin{lemma}\label{lemma:prob-bd-ensemble-ogp}
Let $\mathcal{I}\subset[0,\pi/2]$ with $|\mathcal{I}|\le 2^{cn}$ and $\min_{0\ne \tau\in\mathcal{I}}\tau \ge 2^{-\delta n}$ for some $\delta<\epsilon$. Then,
   \[
\max_{\tau_1,\dots,\tau_m\in\mathcal{I}}\mathbb{P}_{\rm pl}\left[\max_{1\le i\le m}H\bigl(\bs_i,Y_i(\tau_i)\bigr)\le 2^{-\epsilon n}\right]\le \exp_2\Bigl(-n(\epsilon-\delta)m +n\log_2 3+O(m\log n)\Bigr).
   \] 
   In particular,
   \begin{align*}
       &\mathbb{P}_{\rm pl}\left[\exists \tau_1,\dots,\tau_m\in\mathcal{I}:\max_{1\le i\le m}H\bigl(\bs_i,Y_i(\tau_i)\bigr)\le 2^{-\epsilon n}\right]\\
       &\le \exp_2\Bigl(-n(\epsilon-\delta)m+n\log_2 3 +cnm+O(m\log n)\Bigr).
          \end{align*}

\end{lemma}
\begin{proof}[Proof of Lemma~\ref{lemma:prob-bd-ensemble-ogp}]
    Fix an $\tau_1,\dots,\tau_m\in\mathcal{I}$, let $I = \{i\in [m]:\tau_i=0\}$ and $I^c=[m]\setminus I$, where $[m]=\{1,2,\dots,m\}$. We begin by writing
   \begin{align}\label{eq:cond-bd}
       \mathbb{P}_{\rm pl}\left[\max_{1\le i\le m}H\bigl(\bs_i,Y_i(\tau_i)\bigr)\le 2^{-\epsilon n}\right]&=\frac{ \mathbb{P}\bigl[\max_{1\le i\le m}H\bigl(\bs_i,Y_i(\tau_i)\bigr)\le 2^{-\epsilon n},\max_{0\le j\le m}H(\bs^*,X_j)\le 3^{-n}\bigr]}{\mathbb{P}\bigl[\max_{0\le j\le m}H(\bs^*,X_j)\le 3^{-n}\bigr]}
   \end{align}
   and proceed to control the numerator and the denominator. Note that $\mathbb{P}[|Z|\le 3^{-n}]\le (1/\sqrt{2\pi})(2\cdot 3^{-n})(1+o_n(1))$ per Lemma~\ref{lemma:gaussproblemma}. So, we can bound denominator by
   \begin{equation}\label{eq:denom}
          \mathbb{P}\left[\max_{0\le j\le m}H(\bs^*,X_j)\le 3^{-n}\right] = \mathbb{P}\bigl[H(\bs^*,X_0)\le 3^{-n}\bigr]^{m+1} = \exp_2\bigl(-n(m+1)\log_2 3 +O(m)\bigr),
   \end{equation}
   where we used the fact that $X_i$'s are i.i.d.\,unconditionally. The numerator, on the other hand, is more involved. Introduce the events 
   \begin{align*}
       E &= \left\{\max_{i\in I}H\bigl(\bs_i,Y_i(\tau_i)\bigr)\le 2^{-\epsilon n}\right\} \\
              E_i &= \Bigl\{H\bigl(\bs^*,X_i\bigr)\le 3^{-n}\Bigr\},\quad i\in I \\
       F_i &= \Bigl\{H\bigl(\bs_i,Y_i(\tau_i)\bigr)\le 2^{-\epsilon n},H\bigl(\bs^*,X_i\bigr)\le 3^{-n}\Bigr\},\quad i\in I^c.
   \end{align*}
   Further, define the associated indicator random variables $I_E=\ind\{E\}$, $I_i=\ind\{E_i\}$, $i\in I$, and $\bar{I}_i = \ind\{F_i\}$, $i\in I^c$. Observe that $I_E$ is a function of $X_0$ only, $I_i$ is a function of $X_i$ only ($i\in I$), and $\bar{I}_i$ is a function of $X_0$ and $X_i$ ($i\in I^c$). The key observation is, 
   \begin{align}
       &\mathbb{P}\left[\max_{1\le i\le m}H\bigl(\bs_i,Y_i(\tau_i)\bigr)\le 2^{-\epsilon n},\max_{0\le j\le m}H(\bs^*,X_j)\le 3^{-n}\right] \nonumber\\
       &\le \mathbb{P}\left[\max_{1\le i\le m}H\bigl(\bs_i,Y_i(\tau_i)\bigr)\le 2^{-\epsilon n},\max_{1\le j\le m}H(\bs^*,X_j)\le 3^{-n}\right]\nonumber \\
       &=\mathbb{E}\left[I_E\cdot \prod_{i\in I}I_i \cdot \prod_{i\in I^c}\bar{I}_i\right] \label{eq:num1}.
      \end{align}
      
Equipped with this, we have
\begin{align}
    \mathbb{E}\left[I_E\cdot \prod_{i\in I}I_i \cdot \prod_{i\in I^c}\bar{I}_i\right]&=\prod_{i\in I}\mathbb{P}[E_i] \cdot \mathbb{E}\left[I_E\cdot \prod_{i\in I^c}\bar{I}_i\right] \label{eq:used-indep} \\
    &=\prod_{i\in I}\mathbb{P}[E_i]\cdot \mathbb{E}\left[\mathbb{E}\left[I_E\cdot \prod_{i\in I^c}\bar{I}_i\Big\lvert X_0\right]\right] \label{eq:tower-prop} \\
    &=\prod_{i\in I}\mathbb{P}[E_i]\cdot \mathbb{E}_{X_0}\left[I_E\cdot \prod_{i\in I^c}\mathbb{E}_{X_i}\left[\bar{I}_i\Big\lvert X_0\right]\right] .\label{eq:take-away}
\end{align}
   We justify these lines as follows. Note that $I_i,i\in i$ are independent of $I_E$ and $\bar{I}_i,i\in I^c$, yielding~\eqref{eq:used-indep}. Equation~\eqref{eq:tower-prop} follows from the tower property of conditional expectations, whereas~\eqref{eq:take-away} uses the facts that (a) $I_E$ is a function of $X_0$ only and (b) conditional on $X_0$, $\bar{I}_i$, $i\in I^c$, are independent.

   Using Lemma~\ref{lemma:gaussproblemma} as above, we immediately get
   \begin{equation}\label{eq:num2}
       \prod_{i\in I}\mathbb{P}[E_i] = \exp_2\bigl(-n|I|\log_2 3+O(m)\bigr).
   \end{equation}
   Next, we fix an $i\in I^c$ and proceed to control $\mathbb{E}[\bar{I}_i\mid X_0]$, where the expectation is taken with respect to $X_i$ only, where $i\in I^c$. Set $Z_i = n^{-\frac12}\ip{\bs_i}{X_i}$ and $Z_i^* = n^{-\frac12}\ip{\bs^*}{X_i}$. Note that $(Z_i,Z_i^*)$ is a centered bivariate normal random vector with parameter $\rho =\frac1n\ip{\bs_i}{\bs^*}\le \frac{n-2}{n}$. With these, we have
       \begin{align}
   \mathbb{E}[\bar{I}_i\mid X_0]  &= \mathbb{P}\bigl[H(\bs_i,Y_i(\tau_i))\le 2^{-\epsilon n},H(\bs^*,X_i)\le 3^{-n}\mid X_0\bigr] \label{eq:def}\\
&=\mathbb{P}\left[Z_i \in \left[\frac{-2^{-\epsilon n}-\cos(\tau_i)\ip{\bs_i}{X_0}}{\sin(\tau_i)},\frac{2^{-\epsilon n}-\cos(\tau_i)\ip{\bs_i}{X_0}}{\sin(\tau_i)}\right], Z_i^*\in[-3^{-n},3^{-n}]\right]\label{eq:def2} \\
&\le \frac{\sqrt{n}(1+o_n(1))}{4\pi}\cdot \left(\frac{2\cdot 2^{-\epsilon n}}{\sin(\tau_i)}\right)\cdot\left(2\cdot 3^{-n}\right)\label{eq:bivar-norm}\\
&\le \frac{4\sqrt{n}2^{-(\epsilon-\delta) n}3^{-n}}{\pi} \label{eq:sine-identity}.
   \end{align}
  We now justify these lines. \eqref{eq:def} and~\eqref{eq:def2} are due to definitions of indicator $\mathcal{I}_i$ and $Y_i(\tau)$ per~\eqref{eq:Y_i-tau}.~\eqref{eq:bivar-norm} follows from applying Lemma~\ref{lemma:anti-con} to the bivariate normal $(Z_i,Z_i^*)$ with parameter $\rho\le 1-\frac2n$, so that $1-\rho^2 = (1-\rho)(1+\rho)\ge \frac2n$. Finally~\eqref{eq:sine-identity} holds via  Lemma~\ref{lemma:trigo}${\rm (a)}$: we have $\sin(\tau_i)\ge \tau_i/2\ge 2^{-\delta n-1}$, $\min_{0\ne \tau \in \mathcal{I}}\tau\ge 2^{-\delta n}$.
 Note that the right hand side of~\eqref{eq:sine-identity} is independent of $X_i$. So, for any $i$, 
  \begin{equation}\label{eq:num3}
      \prod_{i\in I^c}\mathbb{E}[\bar{I}_i\mid X_0] \le \exp_2\Bigl(-n(\epsilon -\delta)|I^c| -n\log_2 3 |I^c| +O(\log n)\Bigr),
  \end{equation}
  almost surely. We lastly control $\mathbb{E}[I_E]$. Set $\mathcal{Z}_i = n^{-\frac12}\ip{\bs_i}{X_0}$, $i\in I$. Note that $(\mathcal{Z}_i:i\in I)$ is a centered multivariate normal random vector, denote by  $\Sigma\in\R^{|I|\times |I|}$ its covariance, where
  \[
  \Sigma_{ij} = \frac1n\ip{\bs_i}{\bs_j}\in[\beta-\eta,\beta],\quad i,j\in I.
  \]
  Next, applying Lemma~\ref{lemma:anti-con} with $V=[-2^{-\epsilon n},2^{-\epsilon n}]^m$, we get
 \[
 \mathbb{E}[I_E] = \mathbb{P}\left[\max_{i\in I}|\mathcal{Z}_i|\le 2^{-\epsilon n}\right]\le (2\pi)^{-\frac{|I|}{2}}|\Sigma|^{-\frac12} (2\cdot 2^{-\epsilon n})^{|I|}.
 \] 
  Write $\Sigma=\Sigma'+E$ where $\Sigma'_{ii}=1$ and $\Sigma'_{ij}=\Sigma'_{ji}=\beta$ for $i\ne j$. In particular, (a) $\Sigma'=(1-\beta)I+\beta \boldsymbol{1}_{|I|\times |I|}$ and (b) $E$ is such that $E_{ii}=0$ and $-\eta\le E_{ij}=E_{ji}\le 0$ for $i<j$, so that $\|E\|_F< |I|\eta\le \eta m$. Further, the spectrum of $\Sigma'$ consists of the eigenvalue $1-\beta$ with multiplicity $|I|-1$ and the eigenvalue $1-\beta+\beta|I|$ with multiplicity one. So, Theorem~\ref{thm:wielandt-hoffman} yields that provided $\|E\|_2<1-\beta$, $\Sigma$ is invertible. As $\|E\|_2\le\|E\|_F<m\eta$, it suffices to have $\eta<\frac{1-\beta}{m}$ for $\Sigma$ to be invertible. This will be satisfied by our eventual choice of $m,\beta,\eta$. Under these, $\Sigma\succ 0$. Moreover, $n\Sigma$ is integer-valued, so that $|n\Sigma|\ge 1$. Thus, $|\Sigma|\ge n^{-|I|}\ge n^{-m}$ as $|I|\le m$.  Combining these facts, we conclude
  \begin{equation}\label{eq:num4}
      \mathbb{E}[I_E] \le \exp_2\bigl(-\epsilon n|I|+O(m\log n)\bigr).
  \end{equation}

  We are now ready to bound the numerator appearing in~\eqref{eq:cond-bd}. Combining~\eqref{eq:num1},~\eqref{eq:num2},~\eqref{eq:num3} and~\eqref{eq:num4} and using the fact $|I|+|I^c|=m$, we obtain
  \begin{align}
     & \mathbb{P}\left[\max_{1\le i\le m}H\bigl(\bs_i,Y_i(\tau_i)\bigr)\le 2^{-\epsilon n},\max_{0\le j\le m}H(\bs^*,X_j)\le 3^{-n}\right] \nonumber \\
     &\le \exp_2\Bigl(-n(\epsilon-\delta)|I|^c -n\epsilon|I|-nm\log_2 3+O(m\log n)\Bigr).\label{eq:num5}
  \end{align}
  Finally, combining~\eqref{eq:cond-bd},~\eqref{eq:denom} and~\eqref{eq:num5}, and recalling $\tau_i$ are arbitrary, we obtain
  \begin{align*}
      \mathbb{P}_{\rm pl}\left[\max_{1\le i\le m}H\bigl(\bs_i,Y_i(\tau_i)\bigr)\le 2^{-\epsilon n}\right]&\le\exp_2\Bigl(-n(\epsilon-\delta)|I|^c -n\epsilon|I|+n\log_2 3+O(m\log n)\Bigr) \\
      &\le \exp_2\Bigl(-n(\epsilon-\delta)(|I|+|I^c|)+n\log_2 3 +O(m\log n)\Bigr)\\
      &=\exp_2\bigl(-n(\epsilon-\delta)m +n\log_2 3+O(m\log n)\bigr).
    \end{align*}
    From here, Lemma~\ref{lemma:prob-bd-ensemble-ogp} follows directly by taking a union bound over all $\tau_1,\dots,\tau_m\in\mathcal{I}$.
  \end{proof}
Equipped with Lemma~\ref{lemma:prob-bd-ensemble-ogp}, we are ready to estimate the expectation.
\paragraph{Estimating the Expectation} Returning to~\eqref{eq:main-RV}, we have
\begin{align*}
    \mathbb{E}_{\rm pl}\bigl|\mathcal{S}(m,\beta,\eta,\epsilon n,\mathcal{I})\bigr|&=\sum_{(\bs_1,\dots,\bs_m)\in\mathcal{F}(m,\beta,\eta)}\mathbb{P}_{\rm pl}\Bigl[\exists \tau_1,\dots,\tau_m\in \mathcal{I}:H\bigl(\bs_i,Y_i(\tau_i)\bigr)\le 2^{-\epsilon n}\Bigr] \\
    &\le \exp_2\left(n(1+\log_2 3)+mnh_b\left(\frac{1-\beta+\eta}{2}\right)-n(\epsilon-\delta)m + cnm+ O(m\log n)\right)
\end{align*}
using the counting bound~\eqref{eq:stirr} and Lemma~\ref{lemma:prob-bd-ensemble-ogp}. We now prescribe the parameters. Set
\begin{equation}\label{eq:m-and-c}
m = \frac{8(1+\log_2 3)}{\epsilon-\delta}\quad\text{and}\quad c=\frac{\epsilon-\delta}{2}.
\end{equation}
We then choose $\beta,\eta$ such that
\begin{equation}\label{eq:BETA_ETA_1}
    \eta = \frac{1-\beta}{2m}
\end{equation}
and 
\begin{equation}\label{eq:BETA_ETA}
    h_b\left(\frac{1-\beta+\eta}{2}\right)=\frac{\epsilon-\delta}{4}<\frac{\epsilon}{4}
\end{equation}
Note that such a choice (for $\beta$) is indeed possible as 
\[
\varphi(\beta):\beta\mapsto h_b\left(\frac{1-\beta}{2} + \frac{1-\beta}{4m}\right)
\]
is continuous with $\lim_{\beta\to 1}\varphi(\beta)=0$. With these choices of parameters, we immediately get
\[
\mathbb{E}_{\rm pl}|\mathcal{S}(m,\beta,\eta,\epsilon n,\mathcal{I}) \le \exp_2\bigl(-\Theta(n)\bigr),
\]
so that applying Markov's inequality we have
\[
\mathbb{P}_{\rm pl}\bigl[\bigl|\mathcal{S}(m,\beta,\eta,\epsilon n,\mathcal{I})\bigr|\ge 1 \bigr]\le \mathbb{E}_{\rm pl}|\mathcal{S}(m,\beta,\eta,\epsilon n,\mathcal{I}) \le \exp_2\bigl(-\Theta(n)\bigr),
\]
establishing that
\[
\mathbb{P}_{\rm pl}\bigl[\mathcal{S}(m,\beta,\eta,\epsilon n,\mathcal{I}) = \varnothing\bigr]\ge 1-2^{-\Theta(n)}.
\]
We now complete the proof of Theorem~\ref{thm:m-ogp-planted-npp} by establishing that w.h.p.
\[
\mathcal{S}^*(m,\beta,\eta,\epsilon n,\mathcal{I})=\varnothing,
\]
where we remind the reader that for $(\bs_1,\dots,\bs_m)\in \mathcal{S}^*(m,\beta,\eta,\epsilon n,\mathcal{I})$, $\bs_i=\pm \bs^*$ is allowed.
\begin{lemma}\label{lemma:rest-of-set}
    We have
    \[
    \mathbb{P}_{\rm pl}\Bigl[\mathcal{S}^*(m,\beta,\eta,\epsilon n,\mathcal{I})\setminus \mathcal{S}(m,\beta,\eta,\epsilon n,\mathcal{I})\ne\varnothing\Bigr]\le 2^{-\Theta(n)}.
    \]
\end{lemma}
Note that assuming Lemma~\ref{lemma:rest-of-set}, we are indeed done as
\begin{align*}
&\mathbb{P}_{\rm pl}\Bigl[\mathcal{S}^*(m,\beta,\eta,\epsilon n,\mathcal{I})\ne\varnothing\Bigr] \\
&\le \mathbb{P}_{\rm pl}\Bigl[\mathcal{S}(m,\beta,\eta,\epsilon n,\mathcal{I})\ne\varnothing\Bigr] + \mathbb{P}_{\rm pl}\Bigl[\mathcal{S}^*(m,\beta,\eta,\epsilon n,\mathcal{I})\setminus \mathcal{S}(m,\beta,\eta,\epsilon n,\mathcal{I})\ne\varnothing\Bigr]\\
&\le 2^{-\Theta(n)}.
\end{align*}
So, it suffices to establish Lemma~\ref{lemma:rest-of-set}.
\begin{proof}[Proof of Lemma~\ref{lemma:rest-of-set}]
Observe that
\begin{align*} 
&\Bigl\{\mathcal{S}^*(m,\beta,\eta,\epsilon n,\mathcal{I})\setminus \mathcal{S}(m,\beta,\eta,\epsilon n,\mathcal{I})\ne\varnothing\Bigr\}\\
&\subseteq \Bigl\{\exists i\in[m],\bs_i\in\Sigma_n, \tau_i\in\mathcal{I}:H(\bs_i,Y_i(\tau_i))\le 2^{-\epsilon n},\beta-\eta\le n^{-1}\ip{\bs_i}{\bs^*}\le \beta\Bigr\}.
\end{align*}
Next, set
\begin{equation}\label{eq:Set-s-hat}\widehat{\mathcal{S}}\triangleq\sum_{\substack{\bs\in\Sigma_n\\ \beta-\eta \le n^{-1}\ip{\bs}{\bs^*}\le \beta}}\ind\Bigl\{\exists \tau\in\mathcal{I}:H(\bs,Y(\tau))\le 2^{-\epsilon n}\Bigr\},
\end{equation}
where $Y=\cos(\tau)X + \sin(\tau)X'$ for $X,X'\sim \cN(0,I_n)$ i.i.d. Clearly,
\[
\mathbb{P}_{\rm pl}\Bigl[\exists i\in[m],\bs_i\in\Sigma_n, \tau_i\in\mathcal{I}:H(\bs_i,Y_i(\tau_i))\le 2^{-\epsilon n},\beta-\eta\le n^{-1}\ip{\bs}{\bs^*}\le \beta\Bigr]\le \mathbb{P}_{\rm pl}\Bigl[\widehat{\mathcal{S}}\ne\varnothing\Bigr].
\]
So, it suffices to verify that
\begin{equation}\label{eq:to-prove}
    \mathbb{P}_{\rm pl}\Bigl[\widehat{\mathcal{S}}\ne\varnothing\Bigr] = \mathbb{P}\Bigl[\widehat{\mathcal{S}}\ne\varnothing\Big\lvert H(\bs^*,X)\le 3^{-n},H(\bs^*,X')\le 3^{-n}\Bigr] \le 2^{-\Theta(n)}.
\end{equation}
Once again, we show $\mathbb{E}_{\rm pl}\bigl|\widehat{\mathcal{S}}\bigr|\le 2^{-\Theta(n)}$ and conclude by applying Markov's inequality. 
\paragraph{Counting term} With $\beta$ and $\eta$ chosen as in~\eqref{eq:BETA_ETA_1} and~\eqref{eq:BETA_ETA}, we have
\begin{equation}\label{eq:cnt}
\Bigl|\Bigl\{\bs\in\Sigma_n:\beta-\eta\le n^{-1}\ip{\bs}{\bs^*}\le \beta\Bigr\}\Bigr|\le n^{O(1)}\binom{n}{n\frac{1-\beta+\eta}{2}} < \exp_2\left(\frac{(\epsilon-\delta)n}{4}+O(\log n)\right).
\end{equation}
\paragraph{Probability term} Fix any $\bs\in\Sigma_n$ with 
\[
\Overlap\triangleq \frac1n\ip{\bs}{\bs^*}\in[\beta-\eta,\beta],
\]
and a $\tau\in\mathcal{I}$. Further, set
\[
Z_1(\Overlap) = \frac{1}{\sqrt{n}}\ip{\bs}{Y(\tau)}, \quad Z_2(\Overlap)=\frac{1}{\sqrt{n}}\ip{\bs^*}{X},\quad\text{and}\quad Z_3(\Overlap)=\frac{1}{\sqrt{n}}\ip{\bs^*}{X'}.
\]
Note that unconditionally, $Z_i(\Overlap)\sim \cN(0,1)$, $1\le i\le 3$, and that $(Z_i(\Overlap):1\le i\le 3)$ is a multivariate normal random vector with covariance matrix
\[
\Sigma_\Overlap =\begin{pmatrix}
    1 & \cos(\tau)\Overlap & \sin(\tau)\Overlap \\
    \cos(\tau)\Overlap & 1 & 0 \\
    \sin(\tau)\Overlap & 0 & 1
\end{pmatrix}.
\]
Moreover,
\[
|\Sigma_\Overlap| = 1-\Overlap^2>0
\]
as $\beta<1$. So,
\begin{align*}
&\mathbb{P}\Bigl[H\bigl(\bs,Y(\tau)\bigr)\le 2^{-\epsilon n}\Big\lvert H(\bs^*,X)\le 3^{-n},H(\bs^*,X')\le 3^{-n}\Bigr]\\&= \frac{\mathbb{P}\Bigl[H\bigl(\bs,Y(\tau)\bigr)\le 2^{-\epsilon n},H(\bs^*,X)\le 3^{-n},H(\bs^*,X')\le 3^{-n}\Bigr]}{\mathbb{P}\Bigl[ H(\bs^*,X)\le 3^{-n},H(\bs^*,X')\le 3^{-n}\Bigr]}\\
&\le \frac{(2\pi)^{-\frac32}|\Sigma_\Overlap|^{-\frac12}(2\cdot 3^{-n})^2 \cdot (2\cdot 2^{-\epsilon n})}{(2\pi)^{-1}(2\cdot 3^{-n})^2 (1+o_n(1))}\\
&=\frac{2+o_n(1)}{\sqrt{2\pi(1-\Overlap^2)}}2^{-\epsilon n} = \exp_2\bigl(-\epsilon n +O(1)\bigr).
\end{align*}
As $\tau$ is arbitrary, we obtain by taking a union bound over $\tau\in\mathcal{I}$ that
\begin{equation}\label{eq:cnt2}
    \max_{\bs:\beta-\eta\le n^{-1}\ip{\bs}{\bs^*}\le \beta} \mathbb{P}_{\rm pl}\Bigl[\exists \tau\in\mathcal{I}:H\bigl(\bs,Y(\tau)\bigr)\le 2^{-\epsilon n}\Bigr]\le \exp_2\bigl(-\epsilon n + cn\bigr) = \exp_2\left(-\frac{\epsilon+\delta}{2}n\right),
\end{equation}
where we recalled $|\mathcal{I}|\le 2^{cn}$ for $c$ defined in~\eqref{eq:m-and-c}. Combining~\eqref{eq:cnt} and~\eqref{eq:cnt2}, we get
\begin{align*}
        \mathbb{E}_{\rm pl}\bigl|\widehat{\mathcal{S}}\bigr|&\le \exp_2\left(\frac{\epsilon -\delta}{4}n -\frac{\epsilon+\delta}{2}n +O(\log n)\right) \\
        &=\exp_2\bigl(-\Theta(n)\bigr).
\end{align*}
So, by Markov's inequality,
\[
\mathbb{P}_{\rm pl}\bigl[\bigl|\widehat{\mathcal{S}}\bigr|\ge 1\bigr]\le \mathbb{E}_{\rm pl}\bigl|\widehat{\mathcal{S}}\bigr|\le 2^{-\Theta(n)}.
\]
This completes the proof of Lemma~\ref{lemma:rest-of-set}.
\end{proof}
\subsection{Proof of Theorem~\ref{thm:stable-hardness}}\label{sec:stable-hardness}
Our proof follows an outline similar to those of~\cite[Theorem~3.2]{gamarnik2021algorithmic} and~\cite[Theorem~3.2]{gamarnik2022algorithms}. It is based, in particular, on the Ramsey Theory from extremal combinatorics and a contradiction argument. 
\subsubsection*{Auxiliary Results from Ramsey Theory}
We record two auxiliary results from Ramsey Theory. The first one is an upper bound on the ``two-color Ramsey numbers".
\begin{theorem}\label{thm:2color-ramsey}
    Fix $k,\ell\in\mathbb{N}$, and denote by $R(k,\ell)$ the smallest $n\in\mathbb{N}$ such that any red/blue coloring of the edges of $K_n$ necessarily contains either a red $K_k$ or a blue $K_\ell$. Then,
    \[
    R(k,\ell)\le \binom{k+\ell-2}{k-1} = \binom{k+\ell-2}{\ell-1}.
    \]
    In particular when $k=\ell=M$ for some $M\in\mathbb{N}$, 
    \[
    R(M,M) \le \binom{2M-2}{M-1}<4^M.
    \]
\end{theorem}
Theorem~\ref{thm:2color-ramsey} is a folklore result in Ramsey Theory, see e.g.~\cite[Theorem~6.6]{gamarnik2021algorithmic} for a proof. 

The second auxiliary result is an upper bound on the ``multicolor Ramsey numbers". 
\begin{theorem}\label{thm:Ramsey-multicolor}
    For $q,m\in\mathbb{N}$, denote by $R_q(m)$ the smallest positive integer $n$ such that any $q$-coloring of the edges of $K_n$ necessarily contains a monochromatic $K_m$. Then,
    \[
    R_q(m)\le q^{qm}.
    \]
\end{theorem}
Theorem~\ref{thm:Ramsey-multicolor} can be established by modifying the \emph{neighborhood-chasing} argument of Erd{\"o}s and Szekeres~\cite{erdos1935combinatorial}, see~\cite[Page~6]{conlon2015recent} for further discussion.

Equipped with Theorems~\ref{thm:2color-ramsey}-\ref{thm:Ramsey-multicolor}, we are ready to start proving Theorem~\ref{thm:stable-hardness}. 
\subsubsection*{Proof of Theorem~\ref{thm:stable-hardness}}
We begin by reminding the reader the notation
\[
H(\bs,X) = \frac{1}{\sqrt{n}}\bigl|\ip{\bs}{X}\bigr|.
\]
Fix an $\epsilon>0$. In what follows, $E=\epsilon n$. We establish that there exists no randomized algorithm $\A:\R^n\times \Omega\to\Sigma_n$ which is
\[
\bigl(E,p_f,p_{\rm st},p_\ell,\rho,L\bigr)-\text{stable}
\]
for the p-\texttt{NPP} in the sense of Definition~\ref{def:stable-algs}. We argue by contradiction. Suppose that such an $\A$ exists.
\paragraph{Parameter Choice} For $E=\epsilon n$, let $m\in\mathbb{N}$, $0<\eta<\beta<1$ be the parameters prescribed by the $m$-OGP result, Theorem~\ref{thm:m-ogp-planted-npp}. In particular, an inspection of the proof of Theorem~\ref{thm:m-ogp-planted-npp} reveals that
\begin{equation}\label{eq:beta-eta-relation}
    \eta = \frac{1-\beta}{2m}\quad\text{and}\quad h_b\left(\frac{1-\beta+\eta}{2}\right) <\frac{\epsilon}{4}.
\end{equation}
We first set
\begin{equation}\label{eq:param-f}
    f=C_1n \quad \text{where}\quad C_1 = \frac{\eta^2}{1600}.
\end{equation}
We then define auxiliary parameters $Q$ and $T$, where
\begin{equation}\label{eq:param-Q-and-T}
    Q=\frac{40C_2\pi\sqrt{L}}{\eta}\quad\text{and}\quad T = \exp_2\left(2^{4mQ\log_2 Q}\right)
\end{equation}
with  $C_2$ being a sufficiently large constant (see Lemma~\ref{lemma:planted-con} below). We finally set
\begin{equation}\label{eq:probs-and-rho}
p_f = \frac{3^{-n}}{108QT\sqrt{\pi}},\quad p_{\rm st} = \frac{\pi}{648Q^2T},\quad p_\ell = \frac{1}{54T},\quad\text{and}\quad \rho = \cos\left(\frac{\pi}{2Q}\right).
\end{equation}
\paragraph{Reduction to Deterministic Algorithms} We first reduce the proof to the case $\A$ is deterministic. That is, we establish that randomness does not improve the performance of stable algorithms.
\begin{lemma}\label{lemma:not-importante}
    Let $\A:\R^n\times \Omega\to \Sigma_n$ be a randomized algorithm that is $(E,p_f,p_{\rm st},p_\ell,\rho,f,L)$-stable for the p-\texttt{NPP} in the sense of Definition~\ref{def:stable-algs}. Then, there exists a deterministic algorithm $\A^*:\R^n\to\Sigma_n$ that is $(E,6p_f,6p_{\rm st},6p_\ell,\rho,f,L)$-stable.
\end{lemma}
\begin{proof}[Proof of Lemma~\ref{lemma:not-importante}]
    Let $C=1$. For any $\omega\in \Omega$, define the event
    \[
    \mathcal{E}_1(\omega)\triangleq \Bigl\{H\bigl(\A(X,\omega),X\bigr)\le 2^{-E}\Bigr\},\quad\text{where}\quad X\in\R^n\quad\text{and}\quad \A(X)\in\Sigma_n.
    \]
    Observe that 
\begin{align*}
    \mathbb{P}_{\rm pl,\omega}\Bigl[H\bigl(\A(X,\omega),X\bigr)> 2^{-E}\Bigr]&=\mathbb{E}_\omega\Bigl[\mathbb{P}_{\rm pl}\bigl[\mathcal{E}_1(\omega)^c\bigr]\Bigr]\le p_f,
\end{align*}
where we perceive $\mathbb{P}_{\rm pl}[\mathcal{E}_1^c(\omega)]$ as a random variable with sole source of randomness $\omega\in\Omega$ with 
\[
\mathbb{P}_{\rm pl}[\mathcal{E}] = \mathbb{P}\bigl[\mathcal{E} \big\lvert H(\bs^*,X)\le 3^{-n}\bigr].
\]
for any event $\mathcal{E}$. Using Markov's inequality, 
\[
\mathbb{P}_\omega\bigl[\mathbb{P}_{\rm pl}\bigl[\mathcal{E}_1(\omega)^c\bigr]\ge 6p_f\bigr]\le \frac{\mathbb{E}_\omega\bigl[\mathbb{P}_{\rm pl}\bigl[\mathcal{E}_1(\omega)^c \bigr]\bigr]}{6p_f}\le \frac16.
\]
So, $\mathbb{P}[\Omega_1]\ge \frac56$, where
\[
\Omega_1=\Bigl\{\omega\in\Omega:\mathbb{P}\bigl[\mathcal{E}_1(\omega)^c\big\lvert H(\bs^*,X)\le 3^{-n}\bigr]<6p_f\Bigr\}.
\]
Next, define 
\[
\mathcal{E}_2(\omega)=\Bigl\{d_H\bigl(\A(X,\omega),\A(Y,\omega)\bigr)\le f+L\|X-Y\|_2\Bigr\},
\]
where $X,Y\sim \cN(0,I_n)$ with $\mathbb{E}[X_iY_i] = \rho$ for $1\le i\le n$ and $\mathbb{E}[X_iY_j]=0$ for $i\ne j$. 
Then the exact same argument as above yields $\mathbb{P}[\Omega_2]\ge \frac56$, where
\[
\Omega_2 = \Bigl\{\omega\in\Omega:\mathbb{P}\bigl[\mathcal{E}_2(\omega)^c\big\lvert H(\bs^*,X)\le 3^{-n},H(\bs^*,Y)\le 3^{-n}\bigr]<6p_{\rm st}\Bigr\}.
\]
Also letting 
\[
\mathcal{E}_3(\omega) = \bigl\{\A(X,\omega)\in\{\pm \bs^*\}\bigr\}
\]
we obtain $\mathbb{P}[
\widetilde{\Omega}]\ge \frac56$, where
\[
\widetilde{\Omega}=\Bigl\{\omega\in\Omega:\mathbb{P}\bigl[\mathcal{E}_3(\omega)^c\big\lvert H(\bs^*,X)\le 3^{-n}\bigr]<6p_\ell\Bigr\}.
\]
Next, set $C=\sqrt{2}$ and analogously define
\begin{align*}
    \overline{\Omega}_1&=\Bigl\{\omega\in\Omega:\mathbb{P}\bigl[\mathcal{E}_1(\omega)^c\big\lvert H(\bs^*,X)\le \sqrt{2}\cdot 3^{-n}\bigr]<6p_f\Bigr\},\\
    \overline{\Omega}_2&=\Bigl\{\omega\in\Omega:\mathbb{P}\bigl[\mathcal{E}_2(\omega)^c\big\lvert H(\bs^*,X)\le \sqrt{2}\cdot 3^{-n},H(\bs^*,Y)\le \sqrt{2}\cdot 3^{-n}\bigr]<6p_{\rm st}\Bigr\}.
\end{align*}
Then the same reasoning yields $\mathbb{P}_\omega[\overline{\Omega}_1]\ge \frac56$ and $\mathbb{P}_\omega[\overline{\Omega}_2]\ge\frac56$.

Now, we have
\[
\mathbb{P}_\omega[\Omega_1\cap \Omega_2]=\mathbb{P}_\omega[\Omega_1]+\mathbb{P}_\omega[\Omega_2] - \mathbb{P}_\omega[\Omega_1\cup\Omega_2]\ge 2\cdot \frac56-1 = \frac23,
\]
where we used $\mathbb{P}_\omega[\Omega_1\cup \Omega_2]\le 1$. Similarly, $\mathbb{P}_\omega[\overline{\Omega}_1\cap \overline{\Omega}_2]\ge \frac23$. Applying the same idea, this time to $\Omega_1\cap \Omega_2$ and $\overline{\Omega}_1\cap \overline{\Omega}_2$, we find that
\[
\mathbb{P}_\omega\bigl[\Omega_1\cap \Omega_2\cap \overline{\Omega}_1\cap \overline{\Omega}_2\bigr]\ge 2\cdot \frac23-1 = \frac13.
\]
Finally, we apply the same logic to the sets $\Omega_1\cap \Omega_2\cap \overline{\Omega}_1\cap \overline{\Omega}_2$ and $\widetilde{\Omega}$ to find 
\[
\mathbb{P}_\omega\bigl[\Omega_1\cap \Omega_2\cap \overline{\Omega}_1\cap \overline{\Omega}_2\cap \widetilde{\Omega}\bigr]\ge \frac13+\frac56 - 1 = \frac16>0.
\]
So, \[
\Omega_1\cap\Omega_2\cap \overline{\Omega}_1\cap \overline{\Omega}_2\cap \widetilde{\Omega}\ne\varnothing.
\]
Now, take any `seed'
\[
\omega^*\in \Omega_1\cap\Omega_2\cap \overline{\Omega}_1\cap \overline{\Omega}_2\cap\widetilde{\Omega}
\]
and set $\A^*(\cdot)=\A(\cdot,\omega^*)$. It is then clear that $\A^*$ is $(E,6p_f,6p_{\rm st},6p_\ell,\rho,L)$-stable, establishing Lemma~\ref{lemma:not-importante}.
\end{proof}
In the remainder of the proof, we will focus on the deterministic algorithm $\A^*$ appearing in Lemma~\ref{lemma:not-importante} that is  $(E,6p_f,6p_{\rm st},6p_\ell,\rho,L)$-stable.
\paragraph{Chaos Event.} The next auxiliary result is the \emph{chaos event}, where we focus on $m$-tuples $\bs_1,\dots,\bs_m\in\Sigma_n\setminus\{\pm \bs^*\}$ with $\max_{1\le i\le m} H(\bs_i,X_i)\le 2^{-\epsilon n}$. 
\begin{lemma}\label{lemma:chaos}
Let $X_i\sim\cN(0,I_n)$, $1\le i\le m$ be i.i.d.\,Fix an $\epsilon>0$ and a $\bs^*\in\Sigma_n$. Let $\eta^*$ be such that $h_b(\eta^*/2)=\epsilon/2$. Then for all sufficiently large  $m$ and $n$,
    \[
\mathbb{P}\Bigl[\mathcal{S}\bigl(m,1,\eta^*,\epsilon n,\{\pi/2\}\bigr)\ne\varnothing\Big\lvert \max_{1\le j\le m}H(\bs^*,X_j)\le 3^{-n}\Bigr]\le 2^{-\Omega(n)},
    \]
    where $\mathcal{S}$ is the set appearing in Definition~\ref{def:overlap-set}.
\end{lemma}
\begin{proof}[Proof of Lemma~\ref{lemma:chaos}]
    The proof is very similar to, and in fact much easier than, that of Theorem~\ref{thm:m-ogp-planted-npp}. So, using the exact same notation, we only highlight the necessary changes. Once again, the argument is based on Markov's inequality. Fix any $\bs_1,\dots,\bs_m\in\Sigma_n\setminus\{\pm\bs^*\}$ with
    \[
    1-\eta^*\le \frac1n\ip{\bs_i}{\bs_j}\le 1,\quad 1\le i<j\le m,
    \]
    where $\eta^*$ is such that
    \[
    h_b\left(\frac{\eta^*}{2}\right)=\frac{\epsilon}{2}.
    \] Using the bound leading to~\eqref{eq:stirr} we find that
    \[
    |\mathcal{F}(m,1,\eta^*)|\le \exp_2\left(n+mnh_b\left(\frac{\eta^*}{2}\right)+O(\log n)\right).
    \]
   Next, for any fixed $\bs_1,\dots,\bs_m\in\Sigma_n\setminus \{\bs^*\}$, we find 
   \begin{align*}
          \mathbb{P}\Bigl[\max_{1\le j\le m}H(\bs_j,X_j)\le 2^{-\epsilon n}\Big\lvert \max_{1\le j\le m}H(\bs^*,X_j)\le 3^{-n}\Bigr] &= \prod_{1\le j\le m} \mathbb{P}\bigl[H(\bs_j,X_j)\le 2^{-\epsilon n}\big\lvert H(\bs^*,X_j)\le 3^{-n}\bigr] \\
          &\le \exp_2\bigl(-\epsilon mn+O(m\log n)\bigr),
      \end{align*}
      where we used~\eqref{eq:use-later}. So, 
      \begin{align*}
\mathbb{E}\left[\bigl|\mathcal{S}\bigl(1,\eta^*,\epsilon n,\{\pi/2\}\bigr)\bigr|\,\,\Big\lvert \max_{1\le j\le m}H(\bs^*,X_j)\le 3^{-n}\right]&\le \exp_2\left(n+mnh_b\left(\frac{\eta^*}{2}\right)-\epsilon mn +O(m\log n)\right)\\
&=\exp_2\left(n-\frac{\epsilon mn}{2}+O(m\log n)\right)\\
&=\exp_2\bigl(-\Theta(n)\bigr)
            \end{align*}
           for all sufficiently large $m$. As
           \begin{align*}
&\mathbb{P}\left[\bigl|\mathcal{S}\bigl(1,\eta^*,\epsilon n,\{\pi/2\}\bigr)\bigr|\ge 1\Big\lvert \max_{1\le j\le m}H(\bs^*,X_j)\le 3^{-n}\right]\\
&\le \mathbb{E}\left[\bigl|\mathcal{S}\bigl(1,\eta^*,\epsilon n,\{\pi/2\}\bigr)\bigr|\,\,\Big\lvert \max_{1\le j\le m}H(\bs^*,X_j)\le 3^{-n}\right]=\exp_2\bigl(-\Theta(n)\bigr)
           \end{align*}
           by Markov's inequality, we complete the proof of Lemma~\ref{lemma:chaos}. 
        \end{proof}
An inspection of the parameters in Theorem~\ref{thm:m-ogp-planted-npp} reveals that $\eta<\beta$ and
\[
h_b\left(\frac{1-\beta+\eta}{2}\right)<\frac{\epsilon}{4}<\frac{\epsilon}{2} = h_b\left(\frac{\eta^*}{2}\right),
\]
so that $    1-\eta^*<\beta-\eta
$. In particular, w.h.p.\,for any $\bs_1,\dots,\bs_m\in\Sigma_n\setminus\{\pm \bs^*\}$ with $\max_{1\le i \le m}H(\bs_i,X_i)\le 2^{-\epsilon n}$, 
there exists $1\le i<j\le m$ such that
\begin{equation}\label{eq:chaos}
\frac1n\ip{\bs_i}{\bs_j}<1-\eta^*<\beta-\eta.
\end{equation}
\paragraph{Planted Measure} To set the stage, fix a $\bs^*\in\Sigma_n$. We now describe the \emph{planted measure}. Let $X_i\sim \cN(0,I_n)$, $0\le i\le T$ be i.i.d.\,random vectors. For any event $\mathcal{E}$,  set
\begin{equation}\label{eq:planted-measure}
\mathbb{P}_{\rm pl}[\mathcal{E}] = \mathbb{P}_{X_0,\dots,X_T}\left[\mathcal{E}\Big\lvert \max_{0\le j\le T}H(\bs^*,X_j)\le 3^{-n}\right].
\end{equation}
In what follows, probabilities are taken with respect to the planted measure~\eqref{eq:planted-measure}. 

\paragraph{Construction of Interpolation Paths} Our proof uses an \emph{interpolation} and a \emph{discretization} argument. Construct the `interpolation paths' per~\eqref{eq:Y_i-tau}, repeated below for convenience:
\begin{equation}\label{eq:interpol-path}
    Y_i(\tau) = \cos(\tau)X_0 + \sin(\tau)X_i \in \R^n,\quad 1\le i\le T,\quad \tau\in[0,\pi/2].
\end{equation}
We next discretize $[0,\pi/2]$. Set $\tau_i = \frac{\pi}{2Q}i$, $0\le i\le Q$. Then, 
\[
0=\tau_0<\tau_1<\cdots<\tau_Q=\frac{\pi}{2}
\]
partition $[0,\pi/2]$ into $Q$ intervals $[\tau_i,\tau_{i+1}]$, $0\le i\le Q-1$, each of length $\frac{\pi}{2Q}$.  Apply $\A^*$ to each $Y_i(\tau_k)$ to obtain partitions
\begin{equation}\label{eq:sigma-i-tau}
\bs_i(\tau_k)=\A^*\bigl(Y_i(\tau_k)\bigr)\in\Sigma_n,\quad 1\le i\le T,\quad 0\le k\le Q,
\end{equation}
and denote their pairwise overlaps by $\Overlap^{(ij)}(\tau_k)$, where
\begin{equation}\label{eq:overlap-i-tau}
    \Overlap^{(ij)}(\tau_k) \triangleq \frac1n \ip{\bs_i(\tau_k)}{\bs_j(\tau_k)}\in[-1,1],\quad 1\le i<j\le T,\quad 0\le k\le Q.
\end{equation}
Observe that 
\begin{equation}\label{eq:overlaps-initial}
    \Overlap^{(ij)}(\tau_0)=1,\quad 1\le i<j\le T.
\end{equation}
\paragraph{Output of Algorithm} Next, using the anti-concentration property of $\mathcal{A}^*$ appearing in Definition~\ref{def:stable-algs}, we establish the following.
\begin{lemma}\label{lemma:alg-anticon}
    \[
    \mathbb{P}_{\rm pl}\bigl[\mathcal{E}_o\bigr]\ge 1-6Tp_\ell, \quad\text{where}\quad \mathcal{E}_o = \Bigl\{\A^*\bigl(\bs_i(\tau_Q)\bigr)\ne\pm \bs^*,\forall i\in[T]\Bigr\}
    \]
\end{lemma}
\begin{proof}[Proof of Lemma~\ref{lemma:alg-anticon}]
    Fix any $i\in [T]$. Note that per Lemma~\ref{lemma:not-importante}, $\mathbb{P}_{\rm pl}\bigl[\A^*\bigl(\bs_i(\tau_Q)\bigr)\in\{\pm \bs^*\}\bigr]\le 6p_\ell$. So, taking a union bound over $1\le i\le T$, we obtain
    \[
    \mathbb{P}_{\rm pl}\Bigl[\exists i\in[T]: \A^*\bigl(\bs_i(\tau_Q)\bigr)\in\{\pm \bs^*\}\Bigr]\le 6Tp_\ell,
    \]
    establishing Lemma~\ref{lemma:alg-anticon}.
\end{proof}
\paragraph{Stability of Overlaps} The most crucial component of our proof is the stability of overlaps: we show that
\[
\Bigl|\Overlap^{(ij)}(\tau_k) - \Overlap^{(ij)}(\tau_{k+1})\Bigr|
\]
is small for all $1\le i<j\le T$ and all $0\le k\le Q-1$. 

To that end, we establish several auxiliary results. The first is a concentration result. 
\begin{lemma}\label{lemma:planted-con}
    For any sufficiently large $C_2>0$,
    \[
    \mathbb{P}_{\rm pl}[\mathcal{E}_{\rm con}]\ge 1-(T+1)e^{-\Omega(n)}, \quad\text{where}\quad \mathcal{E}_{\rm con} = \left\{\max_{0\le i\le T}\|X_i\|_2 \le C_2\sqrt{n}\right\}
    \]
\end{lemma}
\begin{proof}[Proof of Lemma~\ref{lemma:planted-con}]
      Fix any $0\le i\le T$. Applying Bernstein's inequality as in~\cite[Theorem~3.1.1]{vershynin2010introduction}, we obtain that for some absolute $c>0$ and any $t\ge 0$,
    \begin{equation}\label{eq:berns}
\mathbb{P}\left[\left|\frac1n\sum_{1\le u\le n}X_i(u)^2-1\right|\ge t\right]\le \exp\left(-cn\min\{t,t^2\}\right).
        \end{equation}
Note that the probability is taken with respect to $X_i\sim \cN(0,I_n)$ with no conditioning. Now, 
    \begin{align}
       \mathbb{P}_{\rm pl}\bigl[\|X_i\|_2>C_2\sqrt{n}\bigr]&=\mathbb{P}\Bigl[\|X_i\|_2>C_2\sqrt{n}\Big\lvert H(\bs^*,X_i)\le 3^{-n}\Bigr] \label{eq:cond-independence} \\
&\le \frac{\mathbb{P}\bigl[\|X_i\|_2>C_2\sqrt{n}\bigr]}{\mathbb{P}[H(\bs^*,X_i)\le 3^{-n}]} \label{eq:up-bd} \\
&\le \exp_2\left(-cn\left(C_2^2-1\right)+n\log_2 3 +O(1)\right)\label{eq:pax-bern}\\
&=e^{-\Theta(n)}\label{eq:C_2-large}
    \end{align}
    provided $C_2>0$ is sufficiently large. Here,~\eqref{eq:cond-independence} follows from the independence of $X_i$ from $X_j$, $j\ne i$,~\eqref{eq:up-bd} follows from trivial upper bound \[
\mathbb{P}\bigl[\|X_i\|_2>C_2\sqrt{n},H(\bs^*,X_i)\le 3^{-n}\bigr]\le \mathbb{P}\bigl[\|X_i\|_2>C_2\sqrt{n}\bigr],
\]
\eqref{eq:pax-bern} follows by combining~\eqref{eq:berns} with Lemma~\ref{lemma:gaussproblemma}, and finally~\eqref{eq:C_2-large} holds if $C_2>0$ is large enough. In particular any  
\[
C_2>\sqrt{\frac{\log_2 3}{c}+1}
\]
works. Thus,~\eqref{eq:C_2-large} together with a union bound over $0\le i\le T$ yields Lemma~\ref{lemma:planted-con}.
\end{proof}
The next result extends the stability guarantee of $\A^*$ to interpolated instances.
\begin{lemma}\label{lemma:inter-oracle}
    Fix any $1\le i\le T$ and $0\le k\le Q-1$ and let
    \[
    \mathcal{E}_{ik}\triangleq \Bigl\{d_H\bigl(\bs_i(\tau_k),\bs_i(\tau_{k+1})\bigr)\le C_1 n+L\|Y_i(\tau_k)-Y_i(\tau_{k+1})\|_2^2\Bigr\}.
    \]
    Then, for all large enough $n$,
    \[
    \mathbb{P}_{\rm pl}\bigl[\mathcal{E}_h\bigr]\ge 1-\frac{72Q^2 Tp_{\rm st}}{\pi}, \quad\text{where}\quad  \mathcal{E}_h = \bigcap_{1\le i\le T}\bigcap_{0\le k\le Q-1}\mathcal{E}_{ik}.
    \]
\end{lemma}
\begin{proof}[Proof of Lemma~\ref{lemma:inter-oracle}]
    The proof is based on a change of measure argument. First, fix any $1\le i\le T$ and $0\le k\le Q-1$ and recall $Y_i(\tau)$ from~\eqref{eq:interpol-path}. In particular, treating $Y_i$ as column vectors, we have
    \[
\mathbb{E}\bigl[Y_i(\tau_k)Y_i(\tau_{k+1})^T\bigr] = \cos(\tau_k-\tau_{k+1})I_n = \cos\left(\frac{\pi}{2Q}\right)I_n.
    \]
    Next, note that as $X_0$ and $X_i$ are unconditionally independent,
    \begin{equation}\label{eq:com-1}
    \mathbb{P}\left[H(\bs^*,X_0)\le 3^{-n},H(\bs^*,X_i)\le 3^{-n}\right] = \frac{1}{2\pi}\left(2\cdot 3^{-n}\right)^2(1+o_n(1)).
    \end{equation}
    Further, unconditionally
    \begin{equation}\label{eq:com-2}
\mathbb{P}\left[H(\bs^*,Y_i(\tau_k))\le \sqrt{2}\cdot 3^{-n},H(\bs^*,Y_i(\tau_{k+1}))\le \sqrt{2}\cdot 3^{-n}\right] = \frac{1}{2\pi \sin\frac{\pi}{2Q}}\left(2\sqrt{2}\cdot 3^{-n}\right)^2 (1+o_n(1)).
    \end{equation}
    This follows by observing that
     $\frac{1}{\sqrt{n}}\ip{\bs^*}{Y_i(\tau_k)}\sim \cN(0,1),\frac{1}{\sqrt{n}}\ip{\bs^*}{Y_i(\tau_{k+1})}\sim \cN(0,1)$ form a bivariate normal with parameter $\rho = \cos\left(\frac{\pi}{2Q}\right)$ and applying Lemma~\ref{lemma:gaussproblemma}(b). So, we obtain the change of measure
    \begin{equation}\label{eq:com3}
\frac{\mathbb{P}\left[H(\bs^*,Y_i(\tau_k))\le \sqrt{2}\cdot 3^{-n},H(\bs^*,Y_i(\tau_{k+1}))\le \sqrt{2}\cdot 3^{-n}\right]}{ \mathbb{P}\left[H(\bs^*,X_0)\le 3^{-n},H(\bs^*,X_i)\le 3^{-n}\right]} = 
\frac{2+o_n(1)}{\sin\left(\frac{\pi}{2Q}\right)}
    \end{equation}
    
    We are now ready to apply change of measure argument.
    \begin{align}
        \mathbb{P}_{\rm pl}\bigl[\mathcal{E}_{ik}^c\bigr] &= \mathbb{P}\left[\mathcal{E}_{ik}^c\big\lvert H(\bs^*,X_0)\le 3^{-n},H(\bs^*,X_i)\le 3^{-n}\right] \label{eq:c-ind} \\
&=\frac{\mathbb{P}\bigl[\mathcal{E}_{ik}^c,H(\bs^*,X_0)\le 3^{-n},H(\bs^*,X_i)\le 3^{-n}\bigr]}{\mathbb{P}\bigl[H(\bs^*,X_0)\le 3^{-n},H(\bs^*,X_i)\le 3^{-n}\bigr]}\nonumber \\
&\le \frac{\mathbb{P}\bigl[\mathcal{E}_{ik}^c,H(\bs^*,Y_i(\tau_k))\le \sqrt{2}\cdot 3^{-n},H(\bs^*,Y_i(\tau_{k+1}))\le \sqrt{2}\cdot 3^{-n}\bigr]}{\mathbb{P}\bigl[H(\bs^*,Y_i(\tau_k))\le \sqrt{2}\cdot 3^{-n},H(\bs^*,Y_i(\tau_{k+1}))\le \sqrt{2}\cdot 3^{-n}\bigr]}\cdot \frac{2+o_n(1)}{\sin\left(\frac{\pi}{2Q}\right)} \label{eq:com} \\
&=\mathbb{P}\Bigl[\mathcal{E}_{ik}^c\Big\lvert H(\bs^*,Y_i(\tau_k))\le \sqrt{2}\cdot 3^{-n},H(\bs^*,Y_i(\tau_{k+1}))\le \sqrt{2}\cdot 3^{-n}\Bigr]\frac{2+o_n(1)}{\sin\left(\frac{\pi}{2Q}\right)}\nonumber \\
&\le 6p_{\rm st}\cdot \frac{3}{\pi/(4Q)} = \frac{72Q p_{\rm st}}{\pi},\label{eq:sine-ineq}
    \end{align}
    for every large enough $n$.    Here,~\eqref{eq:c-ind} follows from the fact $\mathcal{E}_{ik}$ depends only on $X_0,X_i$ which are independent from $X_j,j\notin \{0,i\}$. Next, if $H(\bs^*,X_0)\le 3^{-n}$ and $H(\bs^*,X_i)\le 3^{-n}$ then for any $\tau\in[0,\pi/2]$,
    \[
    H(\bs^*,Y_i(\tau)) \le \cos(\tau)H(\bs^*,X_0)+\sin(\tau)H(\bs^*,X_i)\le 3^{-n}\left(\cos(\tau)+\sin(\tau)\right)\le \sqrt{2}\cdot 3^{-n},
    \]
    using triangle inequality and Cauchy-Schwarz inequality, $\cos(\tau)+\sin(\tau)\le \sqrt{2(\cos^2(\tau)+\sin^2(\tau))}=\sqrt{2}$. This fact, together with~\eqref{eq:com3} yields~\eqref{eq:com}. Lastly, invoking the fact that $\A^*$ is $(E,6p_f,6p_{\rm st},6p_\ell,\rho,f,L)$-stable per Lemma~\ref{lemma:not-importante} and Definition~\ref{def:stable-algs}, we obtain
    \[
\mathbb{P}\Bigl[\mathcal{E}_{ik}^c\Big\lvert H(\bs^*,Y_i(\tau_k))\le \sqrt{2}\cdot 3^{-n},H(\bs^*,Y_i(\tau_{k+1}))\le \sqrt{2}\cdot 3^{-n}\Bigr]\le 6p_{\rm st}.
    \]
    This, together with the fact $\sin\frac{\pi}{2Q} \ge \frac{\pi}{4Q}$ per Lemma~\ref{lemma:trigo}${\rm (b)}$ yields~\eqref{eq:sine-ineq}.

    Finally, using~\eqref{eq:sine-ineq} and a union bound over $1\le i\le T$ and $0\le k\le Q-1$, we get
    \[
    \mathbb{P}_{\rm pl}\bigl[\mathcal{E}_h^c\bigr]\le \frac{72Q^2Tp_{\rm st}}{\pi},
    \]
    which completes the proof of Lemma~\ref{lemma:inter-oracle}.
\end{proof}

\begin{lemma}\label{lemma:overlaps-stable}
    We have 
    \[
    \mathbb{P}_{\rm pl}[\mathcal{E}_{\rm st}]\ge 1-(T+1)e^{-\Omega(n)} - \frac{72Q^2 T p_{\rm st}}{\pi}
    \]
    where
    \[
    \mathcal{E}_{\rm st} = \left\{\max_{1\le i<j\le T}\max_{0\le k\le Q-1} \Bigl|\Overlap^{(ij)}(\tau_k) - \Overlap^{(ij)}(\tau_{k+1})\Bigr|\le 4\sqrt{C_1}+\frac{4C_2\pi\sqrt{L}}{Q}\right\}.
    \]
\end{lemma}
\begin{proof}[Proof of Lemma~\ref{lemma:overlaps-stable}]
For $\mathcal{E}_{\rm con}$ per Lemma~\ref{lemma:planted-con} and $\mathcal{E}_h$ per Lemma~\ref{lemma:inter-oracle}, we will establish that 
\[
\mathcal{E}_{\rm con} \cap \mathcal{E}_h \subset \mathcal{E}_{\rm st},
\]
from which Lemma~\ref{lemma:overlaps-stable} follows via a union bound. In the remainder of the proof of Lemma~\ref{lemma:overlaps-stable}, assume we operate on $\mathcal{E}_{\rm con}\cap \mathcal{E}_h$. Fix any $1\le i\le T$ and $1\le k\le Q-1$. We have
\begin{align}
    \bigl\|Y_i(\tau_k)-Y_i(\tau_{k+1})\bigr\|_2 &\le \bigl|\cos(\tau_k)-\cos(\tau_{k+1})\bigr|\|X_0\|_2 + \bigl|\sin(\tau_k)-\sin(\tau_{k+1})\bigr|\|X_i\|_2\label{eq:triangleq} \\
    &\le |\tau_{k+1}-\tau_k|\bigl(\|X_0\|_2+\|X_i\|_2\bigr) \label{eq:trigo} \\
    &\le \frac{\pi C_2}{Q}\sqrt{n}\label{eq:last},
\end{align}
where~\eqref{eq:triangleq} follows from triangle inequality applied on interpolation paths per~\eqref{eq:interpol-path},~\eqref{eq:trigo} follows from Lemma~\ref{lemma:trigo}${\rm (b)}$ and the fact $|\tau_{k+1}-\tau_k| =\frac{\pi}{2Q}$ and finally~\eqref{eq:last} follows from the fact that on $\mathcal{E}_{\rm con}$, $\max_{0\le i\le T}\|X_i\|_2\le C_2\sqrt{n}$.

Next, for any $\bs,\bs'\in\Sigma_n$, observe that $\|\bs-\bs'\|_2 = 2\sqrt{d_H(\bs,\bs')}$. So, for any $1\le i\le T$ and $0\le k\le Q-1$, 
\begin{align}
    \bigl\|\bs_i(\tau_k)-\bs_i(\tau_{k+1})\bigr\|_2 & = 2\sqrt{d_H\bigl(\bs_i(\tau_k),\bs_i(\tau_{k+1})\bigr)}\nonumber \\
    &\le 2\sqrt{C_1 n+L\|Y_i(\tau_k)-Y_i(\tau_{k+1})\|_2^2}\label{eq:def-alg} \\
    &\le \sqrt{n}\left(2\sqrt{C_1 }+\frac{2C_2\pi \sqrt{L}}{Q}\right)\label{eq:triv-bd},
\end{align}
where~\eqref{eq:def-alg} follows from~\eqref{eq:last} and the fact that on $\mathcal{E}_h$, 
\[
d_H\bigl(\bs_i(\tau_k),\bs_i(\tau_{k+1})\bigr)\le C_1n + L\|Y_i(\tau_k)-Y_i(\tau_{k+1})\|_2^2,
\]
and~\eqref{eq:triv-bd} follows from the trivial inequality $\sqrt{u+v}\le \sqrt{u}+\sqrt{v}$ valid for all $u,v\ge 0$. 

We are now ready to show $\mathcal{E}_{\rm con}\cap \mathcal{E}_h\subset \mathcal{E}_{\rm st}$. We have
\begin{align}
    \Bigl|\Overlap^{(ij)}(\tau_k) - \Overlap^{(ij)}(\tau_{k+1})\Bigr|&=\frac1n\Bigl|\ip{\bs_i(\tau_k)}{\bs_j(\tau_{k})} - \ip{\bs_i(\tau_{k+1})}{\bs_j(\tau_{k+1})}\Bigr| \nonumber \\
    &\le \frac1n\Bigl(\Bigl|\ip{\bs_i(\tau_k)-\bs_i(\tau_{k+1})}{\bs_j(\tau_{k})}\Bigr|+\Bigl|\ip{\bs_i(\tau_{k+1})}{\bs_j(\tau_k)-\bs_j(\tau_{k+1})}\Bigr|\Bigr)\label{eq:this-is-triangleq} \\
    &\le \frac{1}{\sqrt{n}}\Bigl(\bigl\|\bs_i(\tau_k)-\bs_i(\tau_{k+1})\bigr\|_2+\bigl\|\bs_j(\tau_k)-\bs_j(\tau_{k+1})\bigr\|_2\Bigr) \label{eq:CSE}\\
    &\le 4\sqrt{C_1}+\frac{4C_2\pi\sqrt{L}}{Q}.\label{eq:lasttt}
\end{align}
Here,~\eqref{eq:this-is-triangleq} follows from triangle inequality,~\eqref{eq:CSE} follows from Cauchy-Schwarz inequality and the fact $\|\bs\|_2=\sqrt{n}$ for any $\bs\in\Sigma_n$, and lastly~\eqref{eq:lasttt} follows from~\eqref{eq:triv-bd}. As $1\le i<j\le T$ and $0\le k\le Q-1$ are arbitrary, we thus conclude that on $\mathcal{E}_{\rm con}\cap \mathcal{E}_h$, we indeed have
\[
\max_{1\le i<j\le T}\max_{0\le k\le Q-1}\Bigl|\Overlap^{(ij)}(\tau_k) - \Overlap^{(ij)}(\tau_{k+1})\Bigr|\le 4\sqrt{C_1}+\frac{4C_2\pi\sqrt{L}}{Q},
\]
so that $\mathcal{E}_{\rm con}\cap \mathcal{E}_h\subset \mathcal{E}_{\rm st}$. This completes the proof of Lemma~\ref{lemma:overlaps-stable}.
\end{proof}
\paragraph{Success along Each Path} We next study the event $\A^*$ is successful along each path. 
\begin{lemma}\label{lemma:success}
    Fix any $1\le i\le T$ and $0\le k\le Q$ and let
    \[
    \mathcal{E}_{\rm s,ik}\triangleq \Bigl\{H\bigl(\bs_i(\tau_k),Y_i(\tau_k)\bigr)\le 2^{-\epsilon n}\Bigr\}.
    \]
    Then, 
    \[
    \mathbb{P}_{\rm pl}\bigl[\mathcal{E}_{\rm suc}\bigr]\ge 1-12\cdot 3^n\sqrt{\pi}QTp_f \quad\text{where}\quad \mathcal{E}_{\rm suc} = \bigcap_{1\le i\le T}\bigcap_{0\le k\le Q-1}\mathcal{E}_{\rm s,ik}.
    \]
\end{lemma}
\begin{proof}[Proof of Lemma~\ref{lemma:success}]
Our proof is, agiann, based on a change of measure argument similar to Lemma~\ref{lemma:inter-oracle}. 

First, fix any $1\le i\le T$ and $0\le k\le Q-1$. Then
   \begin{equation}\label{eq:com5}
\frac{\mathbb{P}\bigl[H(\bs^*,Y_i(\tau_k))\le \sqrt{2}\cdot 3^{-n}\bigr]}{\mathbb{P}\bigl[H(\bs^*,X_0)\le 3^{-n},H(\bs^*,X_i)\le 3^{-n}\bigr]} = 3^n\sqrt{\pi}(1+o_n(1)),
     \end{equation} 
     via Lemma~\ref{lemma:gaussproblemma}. Equipped with this, we have for all sufficiently large $n$,
     \begin{align}
         \mathbb{P}_{\rm pl}[\mathcal{E}_{s,ik}^c]&=\mathbb{P}\bigl[\mathcal{E}_{s,ik}^c\big\lvert H(\bs^*,X_0)\le 3^{-n},H(\bs^*,X_i)\le 3^{-n}\bigr] \label{eq:indeppp}\\
         &=\frac{\mathbb{P}\bigl[\mathcal{E}_{s,ik}^c, H(\bs^*,X_0)\le 3^{-n},H(\bs^*,X_i)\le 3^{-n}\bigr]}{\mathbb{P}\bigl[H(\bs^*,X_0)\le 3^{-n},H(\bs^*,X_i)\le 3^{-n}\bigr]}\nonumber \\
         &\le \frac{\mathbb{P}\bigl[\mathcal{E}_{s,ik}^c, H(\bs^*,Y_i(\tau_k)\le\sqrt{2}\cdot 3^{-n}\bigr]}{\mathbb{P}\bigl[H(\bs^*,X_0)\le 3^{-n},H(\bs^*,X_i)\le 3^{-n}\bigr]}\label{eq:trian} \\
         &\le 3^n\sqrt{\pi}(1+o_n(1))\mathbb{P}\bigl[\mathcal{E}_{s,ik}^c\big\lvert H(\bs^*,Y_i(\tau_k))\le \sqrt{2}\cdot 3^{-n}\bigr]\label{eq:succc1} \\
         &\le 3^n\sqrt{\pi}12p_f\label{eq:succ}.
     \end{align}
     Here,~\eqref{eq:indeppp} follows from the fact $\mathcal{E}_{s,ik}^c$ depends only on $X_0$ and $X_i$ and is in particular independent of $X_j$, $j\notin\{0,i\}$; \eqref{eq:trian} follows from the fact that if $H(\bs^*,X_0)\le 3^{-n}$ and $H(\bs^*,X_i)\le 3^{-n}$ then by triangle inequality (as in the step~\eqref{eq:com} in the proof of Lemma~\ref{lemma:inter-oracle}) $H(\bs^*,Y_i(\tau_k))\le \sqrt{2}\cdot 3^{-n}$; \eqref{eq:succc1} follows from~\eqref{eq:com5}; and finally,~\eqref{eq:succ} follows from the fact that $Y_i(\tau_k)\sim \cN(0,I_n)$ (unconditionally) and the fact $\A^*$ is $(E,6p_f,6p_{\rm st},6p_\ell,\rho,f,L)$-stable. 

Finally, taking a union bound over all $1\le i\le T$ and $0\le k\le Q$ via~\eqref{eq:succ}, we establish Lemma~\ref{lemma:success}.
\end{proof}
The rest of the proof is (nearly) identical to that of~\cite[Theorem~3.2]{gamarnik2022algorithms}; it is included herein for completeness.
\paragraph{Putting Everything Together} We now fix any $A\subset [T]$ with $|A|=m$. Define by $S_A$ the set of all $(\bs_i:i\in A)$ such that $\bs_i\ne \pm \bs^*$, 
\begin{itemize}
    \item $\displaystyle \max_{i\in A}H\bigl(\bs_i,Y_i(\pi/2)\bigr)\le 2^{-\epsilon n}$,
    \item $1-\eta^*\le n^{-1}\ip{\bs_i}{\bs_j}\le 1$, for every $i,j\in A$, $i\ne j$.
\end{itemize}
Set $\mathcal{E}_A = \{S_A\ne \varnothing\}$. Namely, $\mathcal{E}_A$ is nothing but the chaos event per Lemma~\ref{lemma:chaos} with indices restricted to $A$. Applying Lemma~\ref{lemma:chaos}, we immediately obtain $\mathbb{P}_{\rm pl}[\mathcal{E}_A] \le \exp_2\bigl(-\Theta(n)\bigr)$. So,
\begin{equation}\label{eq:chaos-event}
    \mathbb{P}_{\rm pl}[\mathcal{E}_{\rm ch}]\ge 1-\binom{T}{m}2^{-\Theta(n)}=1-2^{-\Theta(n)},\quad\text{where}\quad \mathcal{E}_{\rm ch} = \bigcap_{\substack{A:A\subset T, |A|=m}}\mathcal{E}_A^c.
\end{equation}
Above, we used the fact that $T,m=O(1)$ as $n\to\infty$ (per~\eqref{eq:param-Q-and-T}), so that $\binom{T}{m}=O(1)$ as well. Next define
\begin{equation}\label{eq:MAIN-EVENT}
   \mathcal{E}_{\rm main} = \mathcal{E}_o \cap \mathcal{E}_{\rm st}\cap \mathcal{E}_{\rm suc} \cap \mathcal{E}_{\rm ch},
\end{equation}
where $\mathcal{E}_o$ is defined per Lemma~\ref{lemma:alg-anticon}, $\mathcal{E}_{\rm st}$ is defined per Lemma~\ref{lemma:overlaps-stable}, $\mathcal{E}_{\rm suc}$ is defined per Lemma~\ref{lemma:success}, and $\mathcal{E}_{\rm ch}$ is defined per~\eqref{eq:chaos-event}. So, 
\begin{align}
    \mathbb{P}_{\rm pl}\bigl[\mathcal{E}_{\rm main}\bigr]&\ge 1-\mathbb{P}_{\rm pl}\bigl[\mathcal{E}_o^c\bigr]-\mathbb{P}_{\rm pl}\bigl[\mathcal{E}_{\rm st}^c\bigr]-\mathbb{P}_{\rm pl}\bigl[\mathcal{E}_{\rm suc}^c\bigr]-\mathbb{P}_{\rm pl}\bigl[\mathcal{E}_{\rm ch}^c\bigr] \label{eq:this-is-ubd}\\
    &\ge 1-6Tp_\ell - (T+1)e^{-\Omega(n)} - \frac{72Q^2 Tp_{\rm st}}{\pi}-12\cdot 3^n\sqrt{\pi}QTp_f -2^{-\Theta(n)} \label{eq:this-is-lemmas}\\
    &\ge \frac23 - 2^{-\Theta(n)}\label{eq:penul},
\end{align}
where~\eqref{eq:this-is-ubd} follows by taking a union bound in~\eqref{eq:MAIN-EVENT},~\eqref{eq:this-is-lemmas} follows by combining Lemmas~\ref{lemma:alg-anticon},~\ref{lemma:overlaps-stable},~\ref{lemma:success}, and~\eqref{eq:chaos-event}, and~\eqref{eq:penul} follows from inserting the values for $p_\ell,p_{\rm st}$ and $p_f$ per~\eqref{eq:probs-and-rho}. In the remainder of the proof, we operate on the event $\mathcal{E}_{\rm main}$.

Now, inserting the values of $C_1$ per~\eqref{eq:param-f} and $Q$ per~\eqref{eq:param-Q-and-T} while recalling $\mathcal{E}_{\rm main}\subset \mathcal{E}_{\rm st}$, we obtain
\begin{equation}\label{eq:overlap-eta-over-5}
    \max_{1\le i<j\le T}\max_{0\le k\le Q-1}\Bigl|\Overlap^{(ij)}(\tau_k)-\Overlap^{(ij)}(\tau_{k+1})\Bigr|\le \frac{\eta}{5}.
\end{equation}
We next fix any $A\subset[T]$ with $|A|=m$ and establish the following, verbatim from~\cite[Proposition~6.15]{gamarnik2022algorithms}
\begin{proposition}\label{prop:trap}
    For any $A\subset[T]$ with $|A|=m$, there exists distinct $i_A,j_A\in A$ and a $\tau_A\in\{\tau_1,\dots,\tau_Q\}$ such that for $\bar{\eta} = \frac{\eta}{100}$, 
    \[
    \Overlap^{(i_A,j_A)}(\tau_A) \in \bigl(\beta-\eta+3\bar{\eta},\beta-3\bar{\eta}\bigr)\subsetneq (\beta-\eta,\beta).
    \]
    \end{proposition}
    \begin{proof}[Proof of Proposition~\ref{prop:trap}]
        The proof is nearly identical to that of~\cite[Proposition~6.15]{gamarnik2022algorithms}, with an extra step. Recall that
$\mathcal{E}_{\rm main}\subset \mathcal{E}_{\rm ch}\cap \mathcal{E}_o$. Using $\mathcal{E}_{\rm main}\subset \mathcal{E}_o$, we have that $\bs_i(\tau_Q)\ne \pm\bs^*$, $1\le i\le T$. Now, applying $\mathcal{E}_{\rm ch}$ to the $m$-tuple $\bs_1(\tau_Q),\dots,\bs_m(\tau_Q)$, we obtain that there exists distinct $i_A,j_A\in A$ such that
        \[
        \Overlap^{(i_A,j_A)}(\tau_Q) = \frac1n\ip{\bs_{i_A}(\tau_Q)}{\bs_{j_A}(\tau_Q)} \le 1-\eta^*<\beta-\eta,
        \]
        where we utilized~\eqref{eq:chaos}. We now show that there exists $k'\in\{1,2,\dots,Q\}$ such that
        \[
        \Overlap^{(i_A,j_A)}(\tau_{k'})\in (\beta-\eta+3\bar{\eta},\beta-3\bar{\eta}), \quad\text{where}\quad \bar{\eta}= \frac{\eta}{100}.
        \]
        Let $K_0$ be the last index for which $\Overlap^{(i_A,j_A)}(\tau_{K_0})\ge \beta-3\bar{\eta}$. Note that such a $K_0$ indeed exists as (a) $\Overlap^{(i_A,j_A)}(\tau_0)=1$ per~\eqref{eq:overlaps-initial} and (b) $\Overlap^{(i_A,j_A)}(\tau_Q)<\beta-\eta$. We claim \[
        \Overlap^{(i_A,j_A)}(\tau_{K_0+1})\in(\beta-\eta+3\bar{\eta},\beta-3\bar{\eta}).
        \]
        Assume the contrary. Then, we necessarily have
        \[
        \Overlap^{(i_A,j_A)}(\tau_{K_0+1})\le \beta-\eta+3\bar{\eta},
        \]
        so that
        \[
        \Overlap^{(i_A,j_A)}(\tau_{K_0}) -  \Overlap^{(i_A,j_A)}(\tau_{K_0+1})\ge \eta - 6\bar{\eta} = \frac{47}{50}\eta>\frac{\eta}{5},
        \]
        contradicting the bound~\eqref{eq:overlap-eta-over-5}. Setting $\tau_A = \tau_{K_0+1}$, we complete the proof of Proposition~\ref{prop:trap}.
    \end{proof}
    \paragraph{Construction of a Graph and Applying Ramsey Theory} We now construct an appropriate graph $\G=(V,E)$ satisfying the following:
    \begin{itemize}
        \item We have $V=[T]$, where vertex $i$ corresponds to the interpolation path $Y_i(\cdot)$ per~\eqref{eq:interpol-path}. 
        \item For any $1\le i<j\le T$, $(i,j)\in E$ iff there exists a $t\in\{1,2,\dots,Q\}$ such that
        \[
        \Overlap^{(ij)}(\tau_t) \in(\beta-\eta,\beta).
        \]
    \end{itemize}
    In particular, $\G$ has exactly $T$ vertices and certain edges. Next, we construct a $Q$-coloring of the edges of $\G$. Specifically, color $(i,j)\in E$ with color $t\in\{1,2,\dots,Q\}$ where $t$ is the first time such that
    \[
    \Overlap^{(ij)}(\tau_t) \in(\beta-\eta,\beta).
    \]
    We next make an important observation about $\G$. Fix any $A\subset V=[T]$ with $|A|=m$. Using Proposition~\ref{prop:trap}, we find that there exists distinct $i_A,j_A\in A$ such that $(i_A,j_A)\in E$. In particular, the largest independent set of $\G$ is of size at most $m-1$, $\alpha(\G)\le m-1$.

    Equipped with the observation above, we establish the following.
    \begin{proposition}\label{prop:mc-clique}
     $\G=(V,E)$ constructed above contains a monochromatic $m$-clique $K_m$.
    \end{proposition}
    \begin{proof}[Proof of Proposition~\ref{prop:mc-clique}]
        First, per~\eqref{eq:param-Q-and-T}, we have 
        \[
        |V| = T = \exp_2\left(2^{4mQ\log_2 Q}\right).
        \]
        Set
        \begin{equation}\label{eq:M}
            M=Q^{mQ} = 2^{mQ\log_2 Q}.
        \end{equation}
        We now observe that
        \[
        T = \exp_2\left(M^4\right)>2^{2M}>\binom{2M-2}{M-1}\ge R_2(M,M),
        \]
        where the last inequality is due to~Theorem~\ref{thm:2color-ramsey}. Recall from the definition of Ramsey numbers that any graph with at least $R_2(M,M)$ vertices contains either an independent set size $M$ or a clique of size $M$. Since $\alpha(\G)\le m-1<M$, we find that $\G$ necessarily contains a $K_M$, where each edge of $K_M$ is colored with one of $Q$ colors. 

        Next, as 
        \[
        M=Q^{Qm}\ge R_q(m)
        \]
        per~\eqref{eq:M} and Theorem~\ref{thm:Ramsey-multicolor} in that order, we obtain that $K_M$ further contains a monochromatic $K_m$. Since this monochromatic $K_m$ is a subgraph of $\G$, we establish Proposition~\ref{prop:mc-clique}.
            \end{proof}
        Equipped with Proposition~\ref{prop:mc-clique}, we now conclude the proof of Theorem~\ref{thm:stable-hardness}. Take the monochromatic $K_m$ extracted from $\G$ via Proposition~\ref{prop:mc-clique}. There exists an $m$-tuple $1\le i_1<i_2,\cdots<i_m\le T$ and a time $t\in\{1,2,\dots,Q\}$ such that
        \[
        \Overlap^{(i_u,i_v)}(\tau_t) \in (\beta-\eta,\beta),\quad 1\le u<v\le m.
        \]
        Set 
        \[
        \widetilde{\bs}_u \triangleq 
 \bs_{i_u}(\tau_t) =\A^*\bigl(Y_{i_u}(\tau_t)\bigr), \quad 1\le u\le m.
        \]
        Observe that:
        \begin{itemize}
            \item Due to the fact $\mathcal{E}_{\rm main}\subset \mathcal{E}_{\rm suc}$, we have
            \[
            \max_{1\le u\le m}H\Bigl(\widetilde{\bs}_u, Y_{i_u}(\tau_t)\Bigr)\le 2^{-\epsilon n}.
            \]
            \item For $1\le u<u'\le m$, 
            \[
            \beta-\eta<\frac1n\ip{\widetilde{\bs}_u}{\widetilde{\bs}_{u'}}< \beta.
            \]
        \end{itemize}
        Having fixed $\zeta = \{i_1,\dots,i_m\}$, we find that $\mathcal{S}_\zeta\triangleq \mathcal{S}(m,\beta,\eta,\epsilon n,\mathcal{I})$ per Definition~\ref{def:overlap-set} with indices from $\zeta$ and $\mathcal{I}=\{\tau_0,\dots,\tau_Q\}$ is non-empty. Namely, 
        \[
        \mathbb{P}_{\rm pl}\bigl[\exists \zeta\subset[T]:|\zeta|=m,\mathcal{S}_\zeta\ne\varnothing\bigr]\ge \mathbb{P}_{\rm pl}\bigl[\mathcal{E}_{\rm main}\bigr]\ge \frac23-2^{-\Theta(n)},
        \]
        where the last inequality is due to~\eqref{eq:penul}. On the other hand, the $m$-OGP result, Theorem~\ref{thm:m-ogp-planted-npp}, yields that
        \[
        \mathbb{P}_{\rm pl}\bigl[\exists \zeta\subset[T]:|\zeta|=m,\mathcal{S}_\zeta\ne\varnothing\bigr]\le \binom{T}{m}2^{-\Theta(n)} = 2^{-\Theta(n)}
        \]
        as $\binom{T}{m}=O(1)$. So, we obtain
        \[
        2^{-\Theta(n)}\ge \frac23-2^{-\Theta(n)},
        \]
        which is a clear contradiction for $n$ sufficiently large. This completes the proof of Theorem~\ref{thm:stable-hardness}.
      \subsubsection*{Acknowledgments}
    The author wishes to thank David Gamarnik and Muhammed Emre \c{S}ahino\u{g}lu for valuable discussions during the initial stages of this work, as well as Cindy Rush for a valuable discussion during which she suggested a different version of planting. The research of the author is supported by Columbia University, with the Distinguished Postdoctoral Fellowship in Statistics.
    \bibliographystyle{amsalpha}
\bibliography{bibliography2}

\end{document}